\documentclass[leqno,12pt]{article}

\usepackage{pgf,tikz} 
\usetikzlibrary{shapes,arrows,chains}
\usepackage{setspace} 

\usepackage{enumerate} 
\usepackage{enumitem}
\usepackage{amsmath,amsfonts,amssymb,amsthm,stmaryrd,bbm,graphicx,mathtools,enumerate}
\usepackage{mathrsfs}
\usepackage[utf8]{inputenc} 
\usepackage[english]{babel}
\usepackage{wasysym} 
\usepackage{color}
\usepackage{xcolor}

\usepackage[dvipsnames]{xcolor}

\usepackage{graphicx}

\usepackage[most]{tcolorbox}

\usepackage[final]{hyperref}  
\hypersetup{
    linktoc=page,
   linkcolor=red,          
  citecolor=blue,        
    filecolor=blue,      
   urlcolor=cyan,
    colorlinks=true           
}

\usepackage[refpage]{nomencl}

\usepackage{makeidx}
\makeindex

\evensidemargin\oddsidemargin

\newcommand{\eqnsection}{
\renewcommand{\theequation}{\thesection.\arabic{equation}}
   \makeatletter
   \csname  @addtoreset\endcsname{equation}{section}
   \makeatother}
\eqnsection

\def\R{{\mathbb R}}

\def\P{{\bf P}}
\def\E{{\bf E}}
\def\Q{{\bf Q}}
\def\Z{{\mathbb Z}}
\def\N{{\mathbb N}}
\def\T{{\mathbb T}}
\def\cL{{\mathcal L}}
\def\cR{{\mathcal R}}
\def\B{{\mathcal B}}

\def\d{\, \mathrm{d}}

\def\nf{\mathcal{F}}

\newtheorem{theorem}{Theorem}[section]
\newtheorem{proposition}[theorem]{Proposition}
\newtheorem{lemma}[theorem]{Lemma}

\newtheorem{remark}[theorem]{Remark}

\newcommand{\floor}[1]{{\left\lfloor #1 \right\rfloor}}

\newcommand{\ind}[1]{\mathbf{1}_{\left\{ #1 \right\}}}
\renewcommand{\epsilon}{\varepsilon}

\newcommand{ \eps }{\varepsilon}

\newcommand{\V}{\mathbb{V}ar}

\usepackage{stmaryrd} 

\title{\bf Critical branching random walks on $\mathbb{Z}^2$: local survival probabilities and Yaglom limit theorems
}
 
\date{\today}

\author{Tianyi Bai\footnote{State Key Laboratory of Mathematical Sciences, Academy of Mathematics and Systems Science, Chinese Academy of Sciences, Beijing, China, E-mail: \texttt{tianyi.bai73@amss.ac.cn};
\newline\vspace{0.1cm}\hspace{0.2cm} $\dag$ Beijing Normal University, School of Mathematical Sciences, Beijing, China, E-mail: \texttt{xinxin.chen@bnu.edu.cn};
\newline\vspace{0.1cm}\hspace{0.2cm} $\ddag$ Sorbonne Universit\'e, Universit\'e Paris Cit\'e, CNRS, Laboratoire de Probabilit\'es, Statistique et Mod\'elisation, LPSM, F-75005 Paris, France, E-mail: \texttt{shen.lin@sorbonne-universite.fr}}, $\ $Xinxin Chen$^\dag$, and  Shen Lin$^\ddag$ }

\begin{document}

\maketitle

\begin{abstract}

We consider a branching random walk on \(\Z^2\) with critical offspring of mean \(1\) and spatial motion governed by the jumps of a lazy simple random walk. For every site \(x\in\Z^2\), we obtain a uniform asymptotical estimate for the local survival probability, i.e., the probability that there are particles at \(x\) at large time \(n\). Using Stein's method, we establish a Yaglom-type theorem for the number of particles at \(x\) at time \(n\) when \(x\) is at distance of order \(\sqrt n\) from the origin. Moreover, at the position occupied by a typical particle at time \(n\), the number of particles at that site, normalized by \(\log n\), converges in law to a Gamma distribution, thereby confirming a conjecture of Lalley and Zheng [Ann.~Probab.~39 (2011), 327--368]. Finally, we prove that, conditional on local survival, the total number of particles at time \(n\), divided by \(n\), also converges weakly to a Gamma distribution.

\medskip

\noindent\textit{Keywords: Branching random walk, Yaglom theorem, Stein's method.
}

\end{abstract}

\section{Introduction: background, motivation and main results}
\label{Intro}
\subsection{Critical branching random walk and its occupation statistics}

Let us consider a nearest-neighbor branching random walk on the integer lattice $\Z^d$: we start at time 0 with a single ancestor located at some site $x\in \Z^d$. At time 1, the ancestor dies and produces a random number of children according to the offspring law $\mu$. They form the first generation, in which each particle makes an independent jump from $x$ according to the jump law $\nu= \frac{1}{2d+1}\sum_{\|e\|\le 1}\delta_e$, with $\|\cdot\|$ denoting the Euclidean norm. In other words, each child moves to one of the $2d$ nearest-neighbors of $x$ or stays at $x$, with equal probability. Inductively, at time $n+1$ for $n\geq 1$, the particles of the $n$-th generation die and give birth independently according to the same offspring law $\mu$ to some children forming the $(n+1)$-th generation. Each new born individual immediately makes an independent jump from the position of its parent according to~$\nu$. The system continues to evolve in this way. 

The genealogical tree $\T$ of this discrete-time branching system is a Bienaym\'e--Galton--Watson (BGW) tree with offspring law $\mu$. 
This particle system dies out almost surely when $\mu$ has mean at most $1$.
In particular, we focus on the critical case where $\mu$ has mean $1$, $\T$ is almost surely finite.
This critical BGW tree $\T$ will be viewed as a random finite subset of $\mathcal{U}=\{\varnothing\}\cup(\cup_{k\ge 1}\N^k)$ where $\N=\{1,2,\ldots\}$, and $\varnothing$ represents the ancestor (also called the root of the tree $\T$). 
For any individual $u$ at the $n$-th generation with $n\ge1$, it is labeled as $u=(u_1,\cdots, u_{n-1}, u_n)\in \N^n$, which means that it is the $u_n$-th child of the $u_{n-1}$-th child of $\cdots$ of the $u_1$-th child of $\varnothing$. 
The generation of $u$ is defined as $|u|=n$. In particular, $|\varnothing|=0$. We write $\prec$ for the lexicographic order on $\mathcal{U}$, so that $\varnothing\prec 1 \prec (1,1) \prec 2$ for instance. 
For $0\le k\le |u|$, we set $u\vert_k = (u_1,\cdots, u_k)$ to be the ancestor of $u$ at $k$-th generation. The genealogical (partial) order $\leq$ on $\mathcal{U}$ is defined by saying that $u\leq v$ if and only if there exists some $0\leq k \leq |v|$ such that $u=v\vert_k$. In this case, we also say that $v$ is a descendant of $u$. We write $u<v$ if $u\leq v$ and $u\neq v$.

For each particle $u\in\T$, let $S_u\in \Z^d$ denote its spatial position. The critical branching random walk (CBRW for short) is therefore given by $\{(u,S_u), u\in\T\}$. For a CBRW starting from $S_\varnothing=z$, we define it under the probability measure $\P_z$. The corresponding expectation is denoted by $\E_z$. For simplicity, we write $\P$ and $\E$ for $\P_0$ and $\E_0$ respectively. 

We are interested in the spread of the individuals in $\Z^d$. For this purpose, we define the counting measure $Z_n(\cdot)$ at time $n$ by
\[
Z_n(B):=\sum_{|u|=n}\ind{S_u\in B}, \quad \forall B\subset\Z^d,
\]
with the convention that $\sum_\varnothing=0$. 
Given any subset $B\subset \Z^d$, we write the local survival probability as $u_n(B) := \P(Z_n(B) \ge 1)$.
For convenience, we set $Z_n := Z_n(\Z^d)$, and write $Z_n(z) := Z_n(\{z\})$, $u_n(z):=u_n(\{z\})$ for all $z\in \Z^d$.

Clearly, $\{Z_n, n\ge0\}$ is a BGW process with offspring law $\mu$. 
Throughout this work, we will always assume that 
\begin{equation}\label{hyp}
\E[Z_1]=\sum_{k\geq 0} k\mu(k)=1 \quad \textrm{ and }\quad  \sigma^2:=\V_\P(Z_1)= \sum_{k\geq 0} k^2\mu(k)-1 \in (0,\infty).
\end{equation}
A well-known estimate of Kolmogorov says that
\[
\lim_{n\to\infty}n\P(Z_n\ge 1) = \frac{2}{\sigma^2 }.
\]
Furthermore, the classical Yaglom theorem \cite{Yaglom} states that conditional on $\{Z_n\ge 1\}$, $\frac{Z_n}{\sigma^2n/2}$ converges in law to the exponential distribution $\mathrm{Exp}(1)$ with parameter $1$. 

If the CBRW starts with $n$ ancestral particles at time 0, the number of ancestral particles whose descendants survive until time $n$ follows, approximately for large $n$, a Poisson distribution with mean $\frac{2}{\sigma^2}$, and the number of particles $Z_n$ alive at time $n$ is of order $n$. According to Watanabe \cite{Watanabe}, under suitable hypotheses on the initial distribution of ancestral particles, as $n\to \infty$, the measure-valued process $Z_n(\cdot)$ associated with the CBRW converges, after being properly rescaled, to the super-Brownian motion $\mathbf{X}_t$. In dimensions 2 and higher, the random measure $\mathbf{X}_t$ is, for each $t > 0$, almost surely singular with respect to the Lebesgue measure on $\R^d$. When $d\geq 3$, it was established by Perkins \cite{Perkins} that the super-Brownian motion $\mathbf{X}_t$ spreads its mass over its closed support in a uniform manner according to some deterministic Hausdorff measure. A less precise result \cite[Theorem 2]{Perkins} also holds for $d=2$.

In our discrete setting, the occupation statistics of CBRW have been studied by Lalley \cite{Lalley} in the one-dimensional case, and by Lalley and Zheng \cite{LZ11} in dimensions two and higher. See also R\'ev\'esz \cite[Chapter 7]{Revesz}.
To be more precise, recall that $Z_n(x)$ is the number of particles at site $x\in \Z^d$ at time $n$. Let $R_n$ be the total number of sites occupied by the CBRW at time $n$. For $j\geq 1$, we write $M_n(j)$ for the number of sites in $\Z^d$ occupied by exactly $j$ particles of generation~$n$. Finally, we set $V_n:=\max_{x\in \Z^d} Z_n(x)$ to be the maximal number of particles at a single site at time $n$. 
\begin{itemize}
    \item When the dimension $d\geq 3$, Theorem 5 in \cite{LZ11} extends the classical theorem of Yaglom as follows. For certain (non-explicit) constants $(\kappa_j)_{j\geq 1}$ such that $\sum_{j=1}^\infty j\cdot \kappa_j=1$, under $\P(\cdot\,\vert Z_n\ge1)$, the following joint convergence in law holds:
\begin{equation}
    \label{eq:Yaglom-high-dim}
    \bigg(\frac{Z_n}{n}, \left\{\frac{M_n(j)}{n}\right\}_{j\geq 1}, \frac{R_n}{n} \bigg)\xrightarrow[n\to\infty]{\mathrm{(d)}} \bigg(1,\{\kappa_j\}_{j\geq 1}, \sum_j \kappa_j \bigg)\times Y,
\end{equation}
where $Y$ is an exponentially distributed random variable with mean $2/\sigma^2$. 
Moreover, when $d\geq 3$, if the offspring law $\mu$ has finite $k$-th moment for some integer $k\geq 2$, the conditional distributions of $V_n/n^{1/k}$ given survival to generation $n$ are tight according to Theorem 2 of \cite{LZ11}. If $\mu$ has some finite exponential moment, Theorems 3 and 4 of \cite{LZ11} improve the last result by showing that $V_n$ is of order $\log n$ on the event of survival at generation $n$.

\item When the dimension $d=2$, since the two-dimensional random walk is recurrent, we expect different behaviors for the occupation statistics of CBRW. First, assuming that the offspring law $\mu$ has some finite exponential moment, we only know that the order of magnitude of $V_n$ lies between $\log n$ and $(\log n)^2$ (see again Theorems 3 and 4 of \cite{LZ11}). In comparison, for one-dimensional CBRW $V_n$ is shown to be of order $\sqrt n$ (see \cite[Theorem 7.10]{Revesz} and \cite{Lalley}). 
Secondly, consider a typical particle $u(n)$ at time $n$, which is chosen uniformly from the $n$-th generation of the CBRW. Let $T_n=Z_n(S_{u(n)})$ be the number of particles at time~$n$ at the location $S_{u(n)}$ of the typical particle. Theorem 6 of \cite{LZ11} states that in dimension $2$, the conditional distributions of $T_n/\log n$ given survival to generation $n$ are tight. Furthermore, for some sufficiently small $\varepsilon>0$, there exists $\delta>0$ such that 
\[
    \liminf_{n\to \infty} \P(T_n\geq \varepsilon \log n \vert Z_n\geq 1)\geq \delta.
\]
Combining this with Theorem 7 in \cite{LZ11}, Lalley and Zheng have identified $n/\log n$ as the correct order of magnitude for the number of occupied sites $R_n$ at time $n$. 

Let us mention that complete arguments for the preceding results on $T_n$ and $R_n$ are only given in \cite{LZ11} in the special case of binary fission, where the offspring law is $\mu=\frac12\delta_0+\frac12\delta_2$ (the ``double-or-nothing'' case).
\end{itemize}

\subsection{Main results for 2-dimensional CBRW}

The present work focuses on the two-dimensional CBRW. Given the previous results of Lalley and Zheng, it is natural to ask whether a Yaglom-type theorem similar to \eqref{eq:Yaglom-high-dim} holds for $d=2$, and whether the conditional distributions of $\frac{T_n}{\log n}$ given survival up to generation $n$ converge in law as $n\to \infty$. 

Throughout this paper, we follow Lalley and Zheng \cite{LZ11} and fix the jump law as
\begin{align}\label{eq:jump}
    \nu=\frac15\delta_{(0,0)}+\frac15\sum_{\|e\|=1, e\in\Z^2}\delta_e,
\end{align}
in order to avoid some non-essential periodicity problems. 
Our arguments work directly for any aperiodic, irreducible and symmetric jump law on $\Z^2$ with finite support. 
We expect our main results stated below to hold for more general jump laws.
Needless to say, the constant factors that depend on the jump law should be changed accordingly.

Under the probability measure $\mathbb{P}_0$, we define the ordinary random walk $\{S_n, n\ge0\}$ on $\Z^2$ starting at the origin, with i.i.d.~increments distributed as $\nu$. It will be referred to as the lazy simple random walk. We use $P_n(x):=\mathbb{P}_0(S_n=x)$ to denote the associated $n$-step transition probability. 
We set 
\[
    p_n(x):=\frac{5}{4\pi n}\exp\Big(-\frac{5\|x\|^2}{4n}\Big).
\]
It is known from \cite[(2.5)]{Lawler-Limic} 
that uniformly in $x\in\Z^2$,
\begin{equation}
\label{Pnx}
|P_n(x)- p_n(x)|= \frac{O(1)}{n^2} \quad \mbox{as }n\to \infty.
\end{equation}
Conditioning on the first generation of CBRW, and by induction on $n$, we obtain the elementary identity 
\[
    \E[Z_n(x)]=P_n(x).
\]
Recall that $u_n(x)=\P(Z_n(x) \ge 1)$, we thus have $\E[R_n]=\sum_{x\in \Z^2} u_n(x)$.

Our first theorem gives the precise asymptotics of the local survival probability $u_n(x)$. Hereafter, we denote by $o_n(1)$ any term that vanishes as $n\to\infty$. When such a term depends on a variable $x$, we say that it is uniform in $x$ if the convergence to $0$ is uniform in $x$, i.e. $\sup_x |o_n(1)|\to 0$. We say that $a_n=O(1)$ if the sequence $(a_n)$ is bounded, i.e. $\sup_n |a_n|<\infty$. Likewise, we say that $a_n(x)=O(1)$ uniformly in $x$ if $\sup_{n,x}|a_n(x)|<\infty$.

\begin{theorem}\label{thm: Sprobab}
Under the assumption \eqref{hyp}, uniformly in $x\in\Z^2$, we have 
\begin{equation}
\label{asymp-unx}
u_n(x) = \frac{P_n(x) }{\alpha \log n} + \frac{o_n(1)}{n\log n}, \quad \mbox{as }n\to \infty,
\end{equation}
where $\alpha = \frac{5\sigma^2}{16\pi}$. 
If we assume further that the offspring law has a finite third moment, then for any $\varepsilon\in (0,1)$, uniformly in $x\in\Z^2$, we have
\begin{equation}
\label{asymp-unx-strong}
u_n(x) = \frac{P_n(x) }{\alpha \log n} + \frac{O(1)}{n(\log n)^{3/2-\varepsilon}}, \quad \mbox{as }n\to \infty.
\end{equation}
\end{theorem}

\begin{remark}
The uniform estimate \eqref{asymp-unx} is sufficiently good for the following results. 
\end{remark}

\begin{remark}
If \eqref{eq:jump} is replaced by another aperiodic, irreducible and symmetric jump law such that $P_n(0)= \frac {A+o_n(1)}n$ for some constant $A>0$, then
the constant $\alpha$ that appears in Theorem \ref{thm: Sprobab} is
\[
    \alpha=\frac{\sigma^2 A}{4}.
\]

\end{remark}

In view of \eqref{Pnx}, for $\|x\|=o(\sqrt{n})$, we derive from \eqref{asymp-unx} that
\[
u_n(x) = \frac{4+o_n(1)}{\sigma^2 n\log n}.
\]

In the binary fission case $\mu=\frac12\delta_0+\frac12\delta_2$, an upper bound of type $P_n(x)/\log n$ for $u_n(x)$  has already been established in \cite[Proposition 33]{LZ11}, while a similar lower bound can be easily obtained by the Paley--Zygmund inequality (see \cite[Proposition 31]{LZ11}). 
For the sake of completeness, our proof of Theorem \ref{thm: Sprobab} does not rely on \cite{LZ11}.

Our analysis of the local survival probability $u_n(x)$ goes through the spinal decomposition of CBRW, and is partially inspired by the idea of Zhu \cite{Zhu} of considering the leftmost individual in the lexicographic order at generation $n$ that arrives at $x$. 
The second and third authors of the present work have used similar techniques in \cite{CL} to study the total occupation time of CBRW on $\Z^d$. 
Moreover, one can easily generalize the arguments of Theorem \ref{thm: Sprobab} to any fixed bounded set $B\subset \Z^2$, which gives 
\begin{equation}
    \label{eq:local-survival-prob2} 
    \P(Z_n(B) \ge 1) \sim \frac{4}{\sigma^2 n\log n},
\end{equation}
as $n\to\infty$.

For the continuum counterpart, Fleischmann \cite{Fle1978} studied the critical branching Brownian motion (CBBM) in $\R^2$. Let $Z_t(B)$ be the number of particles at time $t$ in the bounded set $B\subset \R^2$, of Lebesgue measure $m(B)$. If CBBM starts with one initial particle, it is shown in \cite{Fle1978} that there are specific constants $c_1$ and $c_2$ so that 
\[
t (\log t) \P\left(\frac{1}{\log t}\frac{Z_t(B)}{m(B)}>\lambda\right)\xrightarrow[t\to\infty]{} c_1 e^{- c_2 \lambda}, \qquad\forall \lambda>0.
\]
This result was generalized by Durrett in Section 8 of \cite{Du1979} for continuous-time BRW on $\R^2$ with critical binary branching. Both works rely on the estimates of $\E[Z_t(B)^k]$ for all $k$, and derive such results from the limiting behavior of the moment generating function of $Z_t(B)/(m(B)\log t)$. When the particles of the critical branching system move on the line $\R$, similar results can also be found in \cite{Fle1978}.
Later, Lee \cite{Lee1991} has given a complete treatment of this problem for CBBM with binary branching in $\R^d$, for all dimensions $d$, by analytic techniques of differential equations. In particular, Theorem 2.2 in \cite{Lee1991} proves that in dimension 2,  as $t\to \infty$, the rescaled local survival probability
$t(\log t)\P(Z_t(B) \ge 1)$ converges to some explicit constant.

For dimensions higher than two, we should compare \eqref{eq:local-survival-prob2} with a recent result of Rapenne \cite{Rapenne} for CBRW. If $B$ is a fixed closed ball in $\R^d, d\geq 3$, for a class of general jump laws on $\R^d$, Lemma 2.10 in \cite{Rapenne} implies the existence of a constant $C_B>0$ (depending on $B$) such that 
\[
    \P(Z_n(B) \ge 1) \sim \frac{C_B}{n^{d/2}}.
\]
The analog of this estimate for CBBM when $d\geq 3$ can be found in Theorem 3.2 of \cite{Lee1991}, where the one-dimensional case is also stated in Theorem 1.2 of the same paper. 

Next, we have the following Yaglom theorem for CBRW in dimension 2.

\begin{theorem}\label{thm: Yaglom}
Under the assumption \eqref{hyp}, for any fixed bounded set $B\subset \Z^2$ and for any sequence $(z_n)_{n\geq 1}$ in $\Z^2$ satisfying $\|z_n\|\leq K\sqrt{n}$ with some fixed real number $K>0$, we write $z_n+B=\{z_n+b\in \Z^2\colon b\in B\}$. Then, under $\P_x(\cdot\,\vert Z_n(z_n+B)\ge1)$, we have the convergence in law
\[
    \frac{Z_n(z_n+B)}{\alpha |B| \log n} \xrightarrow[n\to\infty]{\mathrm{(d)}} Y,
\]
where $\alpha = \frac{5\sigma^2}{16\pi}$ and $|B|$ denotes the cardinality of the set $B$. In the limit, $Y$ is an $\mathrm{Exp}(1)$-distributed random variable.
\end{theorem}

When the dimension $d\geq 3$, for a fixed closed ball $B\subset \R^d$, Proposition 2.13 in Rapenne \cite{Rapenne} implies that $Z_n(\cdot)$ converges in law to some random point process supported on $B$ as $n\to \infty$.
For the continuum counterpart, the work of Lee \cite{Lee1991} previously mentioned contains such Yaglom-type results for CBBM in all dimensions. To be more specific, Theorem 2.4 of \cite{Lee1991} is the analog of our Theorem \ref{thm: Yaglom}. 

In the context of critical branching L\'evy process, the one-dimensional counterparts of Theorems \ref{thm: Sprobab} and \ref{thm: Yaglom} can be found in a recent work of Hou, Ren and Song~\cite{HRS}\footnote{Hou et al. \cite{HRS} cited an earlier “work in progress” version of the present work as their reference [4]. The result stated in their display (1.5) now appears as Theorem \ref{thm: Sprobab} above, where it is proved in greater generality.}
, see Theorems 1.4 and 1.6 therein.
If we consider a CBRW on $\R$ and replace the fixed target set by a growing half-line interval $(-\infty, x\sqrt{n}]$, Hong and Liang~\cite{HL} have recently shown a central limit theorem for $n^{-1}Z_n((-\infty, x\sqrt{n}])$ conditionally on $Z_n\geq 1$. When the one-dimensional branching random walk is subcritical, the analogous result was obtained by Hong and Yao~\cite{HY}.

Our proof of Theorem \ref{thm: Yaglom} relies on Stein's method for exponential distributions developed by Pek\"oz and R\"ollin in \cite{PR2011a}. In fact, as $n\to \infty$, the Wasserstein 1-distance between the conditional law of $\frac{Z_n(z_n+B)}{\alpha |B| \log n}$ given $Z_n(z_n+B) \ge 1$ and the exponential distribution $\mathrm{Exp}(1)$ converges to zero.
To be more precise, we will prove that, uniformly for $z_n\in\{z\in\Z^2\colon \|z\|\le K \sqrt{n}\}$ with some fixed $K>0$, we have
\begin{equation}
\label{dW-uniformcvg}
\lim_{n\to \infty} \sup_{f\in\mathcal{F}_W}\Bigg\vert \E\bigg[f\Big(\frac{Z_n(z_n+B)}{\alpha |B| \log n} \Big)\bigg\vert Z_n(z_n+B) \ge 1 \bigg]- \int_0^\infty e^{-t} f(t) \d t\Bigg\vert = 0,
\end{equation}
where $\mathcal{F}_W:=\{f\colon \R\to\R \,\vert\,  f\textrm{ is Lipschitz continuous with Lipschitz constant $\le 1$} \}$. This is sufficient to conclude the weak convergence stated in Theorem \ref{thm: Yaglom}. We emphasize that the uniformity in the estimate \eqref{asymp-unx} and in the convergence \eqref{dW-uniformcvg} will play a crucial role in the proof of the following Theorems \ref{thm: Typical} and \ref{thm: Span}.

Stein's method for exponential approximations was first used by Pek\"oz and R\"ollin \cite{PR2011a} to obtain a rate of convergence in the Wasserstein distance for the classical Yaglom theorem. 
The same framework was applied to the nearly critical BGW process in Pek\"oz and R\"ollin \cite{PR2011b}, and similar results for critical BGW process in a varying environment can be found in Cardona-Tob\'on, Jaramillo and Palau \cite{CJP}. 
For CBRW occupation statistics in the current setting, if the offspring law has finite higher moments, Pek\"oz, R\"ollin and Ross have obtained in \cite{PRR20} a rate of convergence for \eqref{eq:Yaglom-high-dim} (without considering $R_n$) and even a second order fluctuation result when the dimension $d\geq 7$. 
In our context, to control the Wasserstein distance, we construct in Section \ref{sub:coupling} an appropriate coupling of two branching random walks with spine. We emphasize that, to the best of our knowledge, the two-spine coupling idea was first introduced by Chen, Gu, and Zhao \cite{CGZ2026}---with the second author of the present work as a coauthor---
in the context of Stein’s method for spatial branching processes in a random environment. The present work adapts this two-spine coupling to a new framework, and it is different in nature from the coupling used in \cite{PRR20}.

By Theorem \ref{thm: Yaglom}, if a site is occupied at time $n$, we expect to find there roughly $\log n$ particles at the same time. 
Now let us move to the location $S_{u(n)}\in \Z^2$ of a typical particle $u(n)$ at time $n$. Recall that $T_n = Z_n(S_{u(n)})$. 
The following theorem confirms the conjecture of Lalley and Zheng stated after Theorem 6 in \cite{LZ11}, by showing that the conditional distribution of $T_n/\log n$ given $\{Z_n\ge 1\}$ converges in law to the Gamma distribution $\Gamma(2,1)$.

\begin{theorem}\label{thm: Typical}
Assume \eqref{hyp} and let $\alpha = \frac{5\sigma^2}{16\pi}$. Then, under $\P(\cdot\,\vert Z_n\ge1)$, we have the joint convergence in law
\[
\Big(\frac{T_n}{\alpha \log n}, \frac{Z_n}{\sigma^2n/2}\Big) \xrightarrow[n\to\infty]{\mathrm{(d)}} (X,Y),
\]
where in the limit, $X$ is a $\Gamma(2,1)$-distributed random variable, and $Y$ is an $\mathrm{Exp}(1)$-distributed random variable. Moreover, $X$ and $Y$ are independent. 
\end{theorem}

Recall that the probability density function of $X$ is $xe^{-x}\ind{x>0}$, showing that the Gamma distribution $\Gamma(2,1)$ is the size-biased distribution of $\mathrm{Exp}(1)$.
It is interesting to note that the sum of two independent $\mathrm{Exp}(1)$-distributed random variables also follows the Gamma distribution $\Gamma(2,1)$. 

As we will see in the proof of Theorem \ref{thm: Typical}, if we look at the spine formed by the ancestral line of the chosen particle $u(n)$, the main contribution to $T_n$ comes from the descendants of the individual on the spine that is at generation $n-n/\log n$, whereas for $Z_n$ it suffices to count the progeny of all particles on the spine at an earlier stage. This decorrelation explains the asymptotic independence between $T_n$ and $Z_n$ under the conditional law $\P(\cdot\,\vert Z_n\ge1)$.

Our last theorem concerns the behavior of $Z_n$ conditionally on local survival.

\begin{theorem}\label{thm: Span}
Assume \eqref{hyp} and let $\alpha = \frac{5\sigma^2}{16\pi}$. Fix some real number $K>0$ and a starting point $x\in \Z^2$. For any sequence $(z_n)_{n\geq 1}$ in $\Z^2$ satisfying $\|z_n\|\leq K\sqrt{n}$, under $\P_x(\cdot\,\vert Z_n(z_n)\ge1)$, we have the joint convergence in law
\begin{equation}
  \label{eq:Span}  
  \Big(\frac{Z_n}{\sigma^2 n/2}, \frac{Z_n(z_n)}{\alpha \log n}\Big) \xrightarrow[n\to\infty]{\mathrm{(d)}} (X,Y),
\end{equation}
where in the limit, $X$ is a $\Gamma(2,1)$-distributed random variable, and $Y$ is an $\mathrm{Exp}(1)$-distributed random variable. Moreover, $X$ and $Y$ are independent. 
\end{theorem}

The marginal convergence of the second component in \eqref{eq:Span} is obviously a simple consequence of Theorem \ref{thm: Yaglom}. 
After comparison with the classical Yaglom theorem, we observe that, from an asymptotic point of view, the rescaled number of particles $\frac{Z_n}{\sigma^2 n/2}$ under the local survival conditioning $\{Z_n(z_n)\ge1\}$ stochastically dominates itself under the global survival conditioning $\{Z_n\ge1\}$. However, we do not have an intuitive explanation for this phenomenon.

Our proof of Theorem \ref{thm: Span} shares the same strategy with Theorem \ref{thm: Typical}. The independence between $X$ and $Y$ in the limit can be established in a similar way.

The paper is organized as follows. 
Section \ref{sec:preliminaries} presents some preliminary results on 2-dimensional random walk and recalls the spinal decomposition for CBRW.
In Section \ref{sec:local_survival_probability}, we prove Theorem \ref{thm: Sprobab}. 
In Section \ref{sec:yaglom_theorem_proof}, we apply Stein's method to establish Theorem \ref{thm: Yaglom}. 
Sections \ref{sec:lalley_zheng_conjecture_proof} and \ref{sec:proof_of_theorem_ref_thm_span} are devoted to Theorems \ref{thm: Typical} and \ref{thm: Span} respectively, where we prove the weak convergence of the random vector by identifying the limit of the joint Laplace transform. 

\section*{Acknowledgments}
X.~C.~is supported by National Key R\&D Program of China (No.~2022YFA1006500).
This work was supported by the National Natural Science Foundation of China
(Grant No.~12571148) and by the Fundamental Research Funds for the Central Universities (No.~310432104). 
T.~B.~is supported by the National Natural Science Foundation of China (Grant No.12595284, 12595280, 12501195).

\section{Preliminaries} 
\label{sec:preliminaries}

\subsection{Preliminary results on 2-dimensional random walk}

In this subsection, we give a preliminary lemma for the 2-dimensional lazy simple random walk. 
Let $(S_n)_{n\ge0}$ under $\mathbb{P}_x$ be a lazy simple random walk on $\Z^2$ started from $x$. We write directly $\mathbb{P}=\mathbb{P}_0$ for simplicity.
Recall that $P_n(x)=\mathbb{P}_0(S_n=x)$. 
It is well-known that there exists some constant $C_\eqref{upbdPnx}>0$ such that for all $n\geq 1$,
\begin{equation}
\label{upbdPnx}
\sup_{x\in\Z^2} P_n(x) \le \frac{C_\eqref{upbdPnx}}{n}.
\end{equation}
Consider 
\[
\Gamma_k^n:= \sum_{j=k}^n P_j(S_j), \quad\forall 1\le k\le n.
\]
It has been established in \cite[Lemma 27]{LZ11} that as $n\to \infty$,
\begin{equation}
\mathbb{E}[\Gamma_2^n]\sim \frac{5}{8\pi}\log n \quad \textrm{ and } \quad \V_{\mathbb{P}}\Big(\frac{\Gamma_2^n}{\log n}\Big)=o_n(1).
\end{equation}
This immediately implies the convergence in probability
\begin{equation}
\label{eq:cv-proba-Gamma}
\frac{\Gamma_2^n}{\frac{5}{8\pi}\log n}\xrightarrow[n\to\infty]{\textrm{ in } \mathbb{P}} 1.
\end{equation}

The following lemma gives a stronger version for this convergence in probability, which will be useful in the proof of Theorems \ref{thm: Sprobab} and \ref{thm: Yaglom}. Given $A\subset\Z^2$, for $1\leq k\leq n$ we define 
\[
\Gamma_k^n(A):=\sum_{j=k}^n P_j(A+ S_j),
\]
where $P_j(A+z):=\sum_{y\in A}P_j(y+z)$. If $A=\{z\}$ contains a single point, we write $\Gamma_k^n(z)=\Gamma_k^n(\{z\})$.

\begin{lemma}
\label{lem: SRW}
We fix an integer $K\ge 2$, and consider a sequence $(z_n)_{n\geq 1}\subset\Z^2$ such that $\frac{\log (1+\|z_n\|)}{\log n} \to 0$ as $n\to \infty$.
For any $\varepsilon>0$ and for any $p\ge 2$, we have
\begin{equation}\label{eq: SRWcvg}
\mathbb{P}\left( \Big|\Gamma_{K}^n(z_n) -\frac{5}{8\pi}\log n\Big| \ge \varepsilon \log n\right) = \frac{o_n(1)}{(\log n)^p},
\end{equation}
and for any fixed bounded set $B\subset \Z^2$, we also have
\begin{equation*}
\mathbb{P}\left( \Big|\Gamma_{K}^n(z_n+B) -\frac{5}{8\pi}|B|\log n\Big| \ge \varepsilon \log n\right) = \frac{o_n(1)}{(\log n)^p}.
\end{equation*}
Moreover, these results still hold if we replace the fixed integer $K$ by $\lfloor \log n\rfloor+1$.
\end{lemma}

In order to establish \eqref{asymp-unx-strong}, we will need the following technical refinement.

\begin{lemma}
\label{lem: SRW-sqrt}
For any $\varepsilon\in(0,1/2)$, we have
\begin{equation}
\label{eq:SRWcvg-sqrt}
\mathbb{P}\left( \Big|\Gamma_{\lfloor \log n\rfloor +1}^n(0) -\frac{5}{8\pi}\log n\Big| \ge (\log n)^{\frac12+\varepsilon}\right) = \frac{O(1)}{(\log n)^{3/2-\varepsilon}}.
\end{equation}
\end{lemma}

The proof of these two lemmas is postponed to Appendix \ref{appendix}.

\subsection{Change of measure and the spinal decomposition}
\label{changeofm}

Recall that $\{Z_n\}_{n\ge0}$ is a critical BGW process. For any $x\in\Z^2$, $\{Z_n\}_{n\ge0}$ is a martingale under $\P_x$ with respect to the natural filtration $\{\mathcal{F}_n:=\sigma((u, S_u)_{|u|\le n})\}_{n\ge0}$. We can define a new probability measure $\Q_x$ via the martingale $\{Z_n\}$ as follows:
\[
\frac{\d \Q_x}{\d \P_x}\Big\vert_{\mathcal{F}_n}=Z_n, \quad \forall n\ge0.
\]

Next, we introduce a probability measure $\Q_x^*$ on the branching random walk $\{(u,S_u), u\in\T\}$ with a marked ray $(w_n)_{n\ge0}$, so that the marginal law of $\{(u,S_u), u\in\T\}$ under $\Q_x^*$ is exactly $\Q_x$.

\begin{itemize}
\item At time $0$, we start with an initial ancestor $w_0=\varnothing$,  located at $S_\varnothing =x$. 
\item For any time $n\ge0$, assume that the system is well constructed up to time $n$. At time $n+1$, $w_n$ dies and gives birth to a random number of children according to the size-biased offspring law $\widehat{\mu}=\{\widehat{\mu}(k) := k\mu(k)\}_{k\ge0}$, while any other individual of the $n$-th generation dies and produces independently a random number of children according to $\mu$. Each new born individual makes an independent lazy simple random walk jump from the position of its parent. All children of the individuals of the $n$-th generation form the $(n+1)$-th generation. 
\item Among the children of $w_n$, we choose uniformly one to be $w_{n+1}$. The system continues to evolve in this way.
\end{itemize}

By abuse of notation, we still use $\T$ to represent the genealogical tree for this system, rooted at $w_0=\varnothing$. For any individual $u\in\T$, its position is still denoted by $S_u$, and its generation is still denoted by $|u|$. The marked ray $(w_n)_{n\ge0}$ is called the spine of $\T$. 
By considering the lexicographic order $\prec$ on $\T$, we write $\mathcal{L}(w_{n+1})$ (or $\mathcal{R}(w_{n+1})$) to be the set of children of $w_n$ which are on the left (or right respectively) of $w_{n+1}$. To make this notion clear, $u \prec w_{n+1}$ for each $u\in \mathcal{L}(w_{n+1})$, and $w_{n+1}\prec v$ for each $v\in \mathcal{R}(w_{n+1})$.
The corresponding cardinalities of $\mathcal{L}(w_{n+1})$ and $\mathcal{R}(w_{n+1})$ are denoted by $L(w_{n+1})$ and $R(w_{n+1})$. It follows immediately from the construction that
\begin{equation}
\label{jointLR}
\Q^*_x(L(w_{n+1})=i, R(w_{n+1})=j) = \mu(i+j+1), \quad \forall i,j\ge0, \forall n\ge0.
\end{equation}

As a consequence of the above construction, we have the following proposition. One can refer to the lecture notes of Shi \cite{Shi-book} for a proof. 

\begin{proposition}\label{spine}
Under $\Q^*_x$, the following assertions hold.
\begin{enumerate}
\item For any $n\ge0$,
\[
\Q_x^*(w_n= u \vert \mathcal{F}_n ) = \frac{\ind{|u|=n}}{Z_n}, \forall u\in\T \mbox{ such that } |u|=n,
\]
where $\mathcal{F}_n=\sigma((u,S_u), |u|\le n)$.
\item The random vectors $((L(w_n), R(w_n)))_{n\ge 1}$ are i.i.d.
\item Along the spine $(w_n)_{n\ge0}$, the spatial movement $(S_{w_n})_{n\ge0}$ is a lazy simple random walk on $\Z^2$, and it is independent of $((L(w_n), R(w_n)))_{n\ge 1}$.
\item Conditionally on $\mathcal{G}_\infty:=\sigma(\{(w_n, S_{w_n});  (u, S_u)_{u\in \cL(w_{n+1})\cup\cR(w_{n+1})}\}_{n\ge0})$, for every individual $u\in \bigcup_{n\ge1}(\cL(w_n)\cup\cR(w_n))$, its descendants $(v, S_v)_{v\in\T_u}$ form an independent branching random walk following the law of $\P_{S_u}$. Here, $\T_u$ represents the subtree of $\T$ rooted at $u$ that is formed by all the descendants of $u$ in $\T$.
\end{enumerate}

\end{proposition}

In the rest of the paper, we will use $\E_{\Q_x}$ and $\E_{\Q^*_x}$ to denote the expectation with respect to $\Q_x$ and $\Q^*_x$ respectively.  When $x=0$, we will simply write $\Q$ and $\Q^*$ instead of $\Q_0$ and $\Q^*_0$ respectively.

Recall that the ancestral line of $u\in\T$ is denoted by $(u\vert_0=\varnothing, u\vert_1,\cdots, u\vert_{|u|}=u)$. The following Many-to-One lemma is derived from Proposition \ref{spine}.

\begin{lemma}
For any integer $n\in\N$ and $x\in\Z^2$, and for any non-negative measurable function $F\colon \mathbb{R}^{n+1}\to \R$, we have
\[
\E_x\Bigg[ \sum_{|u|=n} F(S_{u\vert_k}, 0\le k\le n)\Bigg] = \mathbb{E}_x\big[F(S_k, 0\le k\le n)\big].
\]
\end{lemma}

A direct application of the Many-to-One lemma implies that for any $B\subset\Z^2$ and for any $n\ge1$,
\[
 \E[Z_n(B)] = \sum_{y\in B}P_n(y) =: P_n(B).
\]

\section{Local survival probability: proof of Theorem \ref{thm: Sprobab}}
\label{sec:local_survival_probability}

In this section, we study the asymptotics of $u_n(x) =\P(Z_n(x) \ge 1)$ and prove Theorem \ref{thm: Sprobab}.

On the event $\{Z_n(x) \ge 1\}$, let $\ell_n(x)$ be the leftmost individual in the lexicographic order of the set $\{u\in\T \colon |u|=n, S_u =x\}$. 
Then, observe that
\[
\ind{Z_n(x) \ge 1} = \sum_{|u|=n}\ind{u = \ell_n(x) }.
\]
Let us rewrite $u_n(x)$ using the change of measures introduced in Section \ref{changeofm}:
\begin{align}
\label{changeofm-unx}
u_n(x) = \,& \P(Z_n(x) \ge 1) = \E\Bigg[\sum_{|u|=n}\ind{u=\ell_n(x)}\Bigg]\\
=\,& \E_{\Q}\Bigg[ \sum_{|u|=n } \frac{1}{Z_n} \ind{u = \ell_n(x)} \Bigg] = \E_{\Q^*} \Bigg[ \sum_{|u|=n } \frac{1}{Z_n} \ind{u = \ell_n(x)} \Bigg] \nonumber\\
= \,&  \E_{\Q^*} \Bigg[ \sum_{|u|=n } \Q^*(w_n= u \vert \nf_n) \ind{u = \ell_n(x)} \Bigg] = \E_{\Q^*}\Bigg[\sum_{|u|=n}\ind{w_n=u} \ind{u = \ell_n(x)} \Bigg]\nonumber\\
=\,& \Q^*(w_n = \ell_n(x) ).\nonumber
\end{align}
Note that $w_n= \ell_n(x)$ means that $S_{w_n}=x$ and for any $u\in\cup_{k=1}^n \cL(w_k)$, $u$ has no descendant located at $x$ at time $n$. 
For any individual $u$,
given any $A\subset \Z^2$, we denote by $Z^u_m(A)$ the number of the descendants of $u$ located in $A+S_u$ at time 
$|u|+m$.  
For a singleton $A=\{z\}$, we write for simplicity $Z^u_m(z)=Z^u_m(\{z\})$.
Then, for all $u\in\bigcup_{k=1}^n(\cL(w_k)\cup\cR(w_k))$, $\{Z_m^u(\cdot)\}_{m\ge0}$ are independent and all distributed as $\{Z_m(\cdot)\}_{m\ge0}$ under $\P_0$.
We thus obtain 
\begin{align*}
u_n(x) = & \,\E_{\Q^*}\left[ \ind{S_{w_n}=x} \prod_{k=1}^n \prod_{u\in\cL(w_k)}\ind{Z^u_{n-k}(x-S_u) =0}\right]\\
=&\, \E_{\Q^*}\left[  \ind{S_{w_n}=x}  \prod_{k=1}^n \Big(1-\sum_y u_{n-k}(x-S_{w_{k-1}}-y)\nu(y)\Big)^{L(w_k)} \right],
\end{align*}
where, as a reminder, $\nu$ is the jump law of the lazy simple random walk. 

We can also deal with $\E[Z_n(x)]$ similarly as \eqref{changeofm-unx} to get 
\begin{align*}
\E[Z_n(x)] = & \,\E[Z_n(x)\ind{Z_n(x) \ge 1}]\\
=& \,\E_{\Q^*}[Z_n(x) \ind{w_n = \ell_n(x)}]\\
= & \,\E_{\Q^*}\left[\bigg(1+\sum_{j=1}^n \sum_{v\in\cR(w_j)} Z^v_{n-j}(x-S_v) \bigg)\ind{S_{w_n}=x} \prod_{k=1}^n \prod_{u\in\cL(w_k)}\ind{Z^u_{n-k}(x-S_u) =0}\right],
\end{align*}
by use of the fact that on $w_n=\ell_n(x)$, only $w_n$ and the descendants of $v\in \bigcup_{j=1}^n \cR(w_j)$ can contribute to $Z_n(x)$.

Since $\E[Z_n(x)]=P_n(x)$, by Proposition \ref{spine} and conditioning on $\mathcal{G}_\infty$, it follows that
\begin{align}
\label{changeofm-pnx}
P_n(x) 
=& \,\E_{\Q^*}\left[\bigg(1 + \sum_{j=1}^n R(w_j)\sum_{y}P_{n-j}(x-S_{w_{j-1}}-y)\nu(y)\bigg)\ind{S_{w_n}=x} \right.\nonumber\\
&\qquad\qquad\qquad\qquad \left.\times \prod_{k=1}^n \Big(1-\sum_y u_{n-k}(x-S_{w_{k-1}}-y)\nu(y)\Big)^{L(w_k)} \right].
\end{align}
Note that $P_1=\nu$ and that $\sum_yP_m(z-y)P_1(y)=P_{m+1}(z)$. After a change of index, \eqref{changeofm-pnx} becomes 
\begin{align*}
P_n(x) =& \,\E_{\Q^*}\left[ \bigg(1+ \sum_{j=1}^{n} R(w_{n-j+1}) P_{j}(x- S_{w_{n-j}})\bigg) \ind{S_{w_n}=x} \right.\nonumber\\
&\qquad\qquad\qquad\qquad \left.\times \prod_{k=1}^{n} \Big(1-\sum_y u_{k-1}(x-S_{w_{n-k}}-y)\nu(y)\Big)^{L(w_{n-k+1})} \right].
\end{align*}
Now we apply Proposition \ref{spine} and time reversal (since the jump law $\nu$ is symmetric) to obtain that
\begin{align*}
P_n(x) = \mathbb{E}_x\left[\bigg(1+ \sum_{j=1}^n R_j P_j(x-S_j) \bigg)\ind{S_n =0}  \prod_{k=1}^{n} \Big(1-\sum_y u_{k-1}(x-S_k-y)\nu(y)\Big)^{L_k} \right],
\end{align*}
where, for any $x\in \Z^2$, $(L_j, R_j)_{j\ge1}$ under $\mathbb{P}_x$ are i.i.d.~random vectors distributed as \eqref{jointLR}, and are independent of $(S_k)_{k\ge0}$. By shifting $S_0$ to the origin and the fact that $(-S_j)_j$ is distributed as $(S_j)_j$ under $\mathbb{P}_0$, it follows that
\begin{equation}\label{spine-Pnx}
P_n(x) = \mathbb{E}\left[ \bigg( 1+ \sum_{j=1}^n R_j P_j( S_j) \bigg) \ind{ S_n = x } \prod_{k=1}^n ( 1-  u_{k-1}\ast \nu (S_k) )^{L_k} \right]
\end{equation}
where $u_{k}\ast\nu(z) := \sum_y u_k(z-y)\nu(y)$.
Similarly, we also have 
\begin{equation}\label{spine-unx}
u_n(x) = \mathbb{E} \left[ \ind{ S_n = x } \prod_{k=1}^n ( 1-  u_{k-1}\ast \nu (S_k) )^{L_k} \right].
\end{equation}

In the rest of this section, we are going to approximate $\sum_{j=1}^n R_j P_j(S_j)$ by $\alpha \log n$ and verify that as $n\to \infty$, uniformly for all $x\in \Z^2$,
\begin{equation}
    P_n(x) = (\alpha \log n) u_n(x) + \frac{o_n(1)}{n} \quad \mbox{ if the offspring law $\mu$ has finite second moment;} \label{eq:Pn-un-2nd}
\end{equation}
and for any $\varepsilon\in (0,1/2)$, 
\begin{equation}
    P_n(x) = (\alpha \log n) u_n(x) + \frac{O(1)}{n(\log n)^{\frac12-\varepsilon}} \mbox{ if $\mu$ has finite third moment}. \label{eq:Pn-un-3rd}
\end{equation}
These estimates are equivalent to Theorem \ref{thm: Sprobab}. The proof will be divided into two steps.

\textbf{Step 1.} We first establish a uniform upper bound for $u_n(x)$, by showing the existence of a constant $C_\eqref{upbdu}>0$ such that for any $n\ge 2$,
\begin{equation}
\label{upbdu}
\sup_{x\in\Z^2}u_n(x) \le \frac{ C_\eqref{upbdu}  }{n\log n}.
\end{equation}
In fact, we will prove a stronger estimate that for any fixed $a\in(0,1)$, there exists a positive constant $C_{\eqref{upbduv2}}(a)$ depending on $a$ such that for any $x\in\Z^2$, 
\begin{equation}
\label{upbduv2}
u_n(x) \le \frac{C_{\eqref{upbduv2}}(a)}{\log n}\Big(P_n(x) +\frac{1}{\log n} \sup_{\|z\|\le n^a}P_{n-\lfloor n^a\rfloor}(x-z)\Big).
\end{equation}
Clearly, \eqref{upbdu} follows from \eqref{upbduv2} by applying \eqref{upbdPnx}.
Notice that $(L_j, R_j)_{j\ge1}$ are i.i.d., while for fixed $j$, $L_j$ and $R_j$ are in general correlated. Let $f_L(s):=\mathbb{E}[s^{L_1}]$ denote the generating function of $L_1$ and let $H_j(z):=\mathbb{E}[R_j(1- u_{j-1}\ast \nu (z))^{L_j}]$. Then
\begin{align}
\label{indepLR}
&P_n(x) =  u_n(x) +\sum_{j=1}^n \mathbb{E}\left[  R_j P_j( S_j) \ind{S_n =x}  \prod_{k=1}^{n} (1- u_{k-1}\ast \nu (S_k) )^{L_k} \right]\\
=\,&  u_n(x) +\sum_{j=1}^n \mathbb{E}\left[ P_j(S_j) \ind{S_n =x } \prod_{k\neq j, 1\le k\le n} f_L(1- u_{k-1}\ast \nu (S_k) ) H_j(S_j) \right]\nonumber.
\end{align}
Note that by \eqref{jointLR}, $\mathbb{E}[L_j]=\mathbb{E}[R_j]=\frac{\sigma^2}{2}$. One sees that for $h\in(0,1)$,
\begin{align}
\Big|\mathbb{E}\Big[R_j(1-h)^{L_j}\Big]- \mathbb{E}[R_j] \mathbb{E}\Big[(1-h)^{L_j}\Big]\Big| \le &\,\mathbb{E}\left[ R_j \Big| (1-h)^{L_j} - \mathbb{E}\big[ (1-h)^{L_j} \big] \Big| \right]\label{eq:RLdecouple}\\
\le & \,\mathbb{E}\big[ R_j ( 1 \wedge (hL_j + h\mathbb{E}[L_j]) \big], \nonumber
\end{align}
where the second line comes from the inequality that $|(1-h)^{L_j} - 1 |\le h L_j$. The dominated convergence theorem implies that
\begin{equation*}
\lim_{h\to0+} \mathbb{E}\Big[R_j(1-h)^{L_j}\Big]- \mathbb{E}[R_j] \mathbb{E}\Big[(1-h)^{L_j}\Big] = 0 \quad \mbox{ and }\quad \lim_{h\to0+}\mathbb{E}\Big[(1-h)^{L_j}\Big] = 1.
\end{equation*}
Consequently, 
\begin{equation}\label{approxLR}
\lim_{h\to0+} \frac{\mathbb{E}[R_j(1-h)^{L_j}] }{\mathbb{E}[R_j] \mathbb{E}[(1-h)^{L_j}]}= \lim_{h\to0+} \frac{\mathbb{E}[R_j(1-h)^{L_j}] }{\sigma^2f_L(1-h)/2} = 1.
\end{equation}
To simplify the notation, we will write $f(n)\lesssim g(n)$ to say that there exists some constant $C>0$ independent of $n$ such that $f(n)\leq C g(n)$ for all $n$.
It is clear from Kolmogorov's estimate that for any $n\ge1$, $\sup_z u_n(z) \le \P(Z_n \ge 1) \lesssim \frac1n$. Then, 
\[
\sup_z u_n\ast \nu( z) \le \sup_z u_n(z) \times \sum_y \nu(y) \lesssim \frac1n.
\]
So, by taking $K_n := \lfloor\log n \rfloor$ and by \eqref{approxLR}, for every $K_n< j \le n$, we can approximate 
\[
H_j(S_j)=\mathbb{E}\Big[ R_j (1- u_{j-1}\ast \nu (S_j))^{L_j}\Big\vert (S_k)_{k\le n}\Big]
\]
by $\frac{\sigma^2}{2} f_L(1- u_{j-1}\ast \nu (S_j))$ in a uniform manner, so that \eqref{indepLR} implies that
\begin{align}
\label{roughbdPnx}
P_n(x) &\ge \sum_{j=1}^n \mathbb{E}\left[ P_j(S_j) \ind{S_n =x } \prod_{k\neq j, 1\le k\le n} f_L(1- u_{k-1}\ast \nu (S_k) ) H_j(S_j)\right] \\
&\ge  (1+o_n(1))
\sum_{j =K_n+1}^{\lfloor n/\log n\rfloor }  \frac{\sigma^2}{2}\mathbb{E}\left[ P_j(S_j) \ind{S_n =x } \prod_{ 1\le k\le n} f_L(1- u_{k-1}\ast \nu (S_k) )   \right] \nonumber \\
&= (1+o_n(1)) \mathbb{E}\left[  \frac{\sigma^2}{2} \ind{S_n=x} \prod_{k=1}^n (1-u_{k-1}\ast \nu(S_k))^{L_k} \cdot \sum_{j=K_n+1}^{\lfloor n/\log n\rfloor } P_j(S_j)\right].\nonumber
\end{align}
The $o_n(1)$ term that appears in \eqref{roughbdPnx} does not depend on $x\in \Z^2$.
Recall that $\Gamma_{K_n+1}^{\lfloor n/\log n \rfloor}(0)=\sum_{j=K_n+1}^{\lfloor n/\log n\rfloor } P_j(S_j) $. Let us write 
\[
\mathbb{E}_\eqref{roughbdPnx}(x):= \mathbb{E}\left[  \frac{\sigma^2}{2} \ind{S_n=x} \prod_{k=1}^n (1-u_{k-1}\ast \nu(S_k))^{L_k} \cdot \Gamma_{K_n+1}^{\lfloor n/\log n \rfloor}(0)\right].
\]
On the one hand, by \eqref{roughbdPnx} and \eqref{upbdPnx}, one sees that for sufficiently large $n$,
\begin{equation}\label{uproughbdPnx}
\mathbb{E}_\eqref{roughbdPnx}(x) \lesssim P_n(x) \lesssim \frac{1}{n}.
\end{equation}
On the other hand, comparing it with \eqref{spine-unx}, we have
\[
    \Big\vert \mathbb{E}_\eqref{roughbdPnx} (x) - u_n(x)  \frac{5\sigma^2}{16\pi}\log n \Big\vert \le   \mathbb{E}\left[  \frac{\sigma^2}{2} \ind{S_n=x} \prod_{k=1}^n (1-u_{k-1}\ast \nu(S_k))^{L_k} \Big\vert \Gamma_{K_n +1}^{\lfloor n/\log n \rfloor} (0)- \frac{5}{8\pi}\log n \Big\vert \right].
\]
Let us take $0<\epsilon<1$. By considering the following event 
\[
    \textrm{Bad}_n:=\Big\{\Big\vert \Gamma_{K_n +1}^{\lfloor n/\log n \rfloor}(0) - \frac{5}{8\pi}\log n \Big\vert  \ge \epsilon \log n\Big\},
\]
and its complementary, the previous inequality yields that
\begin{equation}\label{Pnx-bad}
\Big\vert \mathbb{E}_\eqref{roughbdPnx}(x) - u_n(x)  \frac{5\sigma^2}{16\pi}\log n \Big\vert \le  \frac{\sigma^2}{2}  \mathbb{E}_\eqref{Pnx-bad}(x) +  \frac{\sigma^2}{2}  u_n(x) \epsilon \log n,
\end{equation}
where 
\[
\mathbb{E}_\eqref{Pnx-bad}(x):=  \mathbb{E}\left[ \ind{S_n=x} \prod_{k=1}^n (1-u_{k-1}\ast \nu(S_k))^{L_k} \Big\vert \Gamma_{K_n+1}^{\lfloor n/\log n \rfloor}(0) - \frac{5}{8\pi}\log n \Big\vert \ind{\textrm{Bad}_n}\right] .
\]
To bound $\mathbb{E}_\eqref{Pnx-bad}(x)$, observe that $\prod_{k=1}^n (1-u_{k-1}\ast \nu(S_k))^{L_k}\le 1$, and thus
\begin{align*}
\mathbb{E}_\eqref{Pnx-bad}(x) \le & \,\mathbb{E}\left[ \ind{S_n=x} \Big\vert \Gamma_{K_n+1}^{\lfloor n/\log n \rfloor}(0) - \frac{5}{8\pi}\log n \Big\vert \ind{\textrm{Bad}_n}\right] \\
= & \, \mathbb{E}\left[ \Big\vert \Gamma_{K_n+1}^{\lfloor n/\log n \rfloor}(0)  - \frac{5}{8\pi}\log n \Big\vert \ind{\textrm{Bad}_n} \times \mathbb{P}_{S_{\lfloor n/\log n\rfloor}} ( S_{n-\lfloor n/\log n\rfloor} = x ) \right],
\end{align*}
by the Markov property of random walk at time $\lfloor n/\log n\rfloor$. Notice that each jump of random walk is at most of distance 1, so that 
\[
     \mathbb{P}_{S_{\lfloor n/\log n\rfloor}} ( S_{n-\lfloor n/\log n\rfloor} = x ) = P_{n-\lfloor n/\log n\rfloor}(x-S_{\lfloor n/\log n\rfloor})\\ \le \sup_{\|z\|\le n/\log n}P_{n-\lfloor n/\log n\rfloor}(x-z).
\]
By applying the uniform bound \eqref{upbdPnx}, we know that
\[
\Gamma_{K_n+1}^{\lfloor n/\log n\rfloor }(0) = \sum_{j= K_n+1}^{\lfloor n/\log n\rfloor} P_j(S_j) \le \sum_{j= K_n+1}^{\lfloor n/\log n\rfloor} \frac{ C_\eqref{upbdPnx}}{j} \lesssim \log n.
\]
Moreover, by Lemma \ref{lem: SRW}, $ \mathbb{P}\left(\textrm{Bad}_n\right) = \frac{o_n(1)}{(\log n)^2}$.
As a result, we see that
\begin{align*}
\mathbb{E}_\eqref{Pnx-bad}(x) \leq  & \,  \mathbb{E}\left[ \Big\vert \Gamma_{K_n+1}^{\lfloor \frac{n}{\log n}\rfloor }(0) - \frac{5}{8\pi}\log n \Big\vert \ind{\textrm{Bad}_n} \right] \cdot \sup_{\|z\|\le \frac{n}{\log n}}P_{n-\lfloor \frac{n}{\log n}\rfloor}(x-z)\\
\lesssim & \,\frac{o_n(1)}{\log n} \sup_{y\in \Z^2} P_{n-\lfloor n/\log n\rfloor}(y),
\end{align*}
which is $\frac{o_n(1)}{n\log n}$ by \eqref{upbdPnx}. Here the $o_n(1)$ term is uniform over all $x\in \Z^2$. Going back to \eqref{Pnx-bad}, we deduce that for an arbitrary $\epsilon\in (0,1)$,
\begin{equation}
\label{eq:estimate-roughbdPnx}
 \Big\vert \mathbb{E}_\eqref{roughbdPnx} (x)- u_n(x)  \frac{5\sigma^2}{16\pi}\log n \Big\vert \le  \frac{\sigma^2}{2}  u_n(x) \epsilon \log n +  \frac{o_n(1)}{n\log n},
\end{equation}
Again, the $o_n(1)$ term on the right-hand side is uniform over all $x\in\Z^2$. The last inequality together with \eqref{uproughbdPnx} suffices to show \eqref{upbdu}.

To derive the claimed improvement \eqref{upbduv2} with any fixed $a\in (0,1)$, notice that in the arguments above, if we replace $\Gamma_{K_n+1}^{\lfloor n/\log n\rfloor}(0)$ by $\Gamma_{K_n+1}^{\lfloor n^a\rfloor}(0)$ and replace $\frac{5}{8\pi}\log n$ by $\frac{5}{8\pi}a\log n$, then all the inequalities up to \eqref{Pnx-bad} still hold. 
Meanwhile, we replace the upper bound of $\mathbb{E}_{\eqref{Pnx-bad}}(x)$ by 
\[
    \frac{o_n(1)}{a\log n} \sup_{\|z\|\le n^a}P_{n-\lfloor n^a\rfloor}(x-z)    
\]
with an $o_n(1)$ term that is uniform over all $x\in\Z^2$. Consequently, we obtain 
\begin{align}
P_n(x) \geq &\, (1+o_n(1))\mathbb{E}\left[  \frac{\sigma^2}{2} \ind{S_n=x} \prod_{k=1}^n (1-u_{k-1}\ast \nu(S_k))^{L_k} \sum_{j=K_n+1}^{\lfloor n^a\rfloor } P_j(S_j)\right] \nonumber\\
 \ge & \, (1+o_n(1))\frac{\sigma^2}{2} \bigg(u_n(x) \Big(\frac{5}{8\pi}- \varepsilon\Big) a\log n -  \frac{o_n(1)}{a\log n} \sup_{\|z\|\le n^a}P_{n-\lfloor n^a\rfloor}(x-z) \bigg)\nonumber,
\end{align}
where all the $o_n(1)$ terms are uniform over all $x\in\Z^2$. 
The upper bound \eqref{upbduv2} for $u_n(x)$ readily follows.

\textbf{Step 2.} Now let us study the precise asymptotics of $u_n(x)$. 
In view of \eqref{spine-unx} and \eqref{spine-Pnx}, we have
\begin{equation}\label{Pnx-decomp}
     P_n(x) = u_n(x) +  \mathbb{E}_{\eqref{smallj}} (x)+ \mathbb{E}_\eqref{middlej}(x)  + \mathbb{E}_{\eqref{largej}}(x) , 
\end{equation}
where $K_n=\lfloor \log n\rfloor$ and 
\begin{align}
\mathbb{E}_\eqref{smallj}(x) :=& \mathbb{E}\left[   \ind{S_n=x} \prod_{k=1}^n (1-u_{k-1}\ast \nu(S_k))^{L_k} \sum_{j=1}^{K_n} R_j P_j(S_j)   \right], \label{smallj}\\
\mathbb{E}_\eqref{middlej}(x) :=& \mathbb{E}\left[   \ind{S_n=x} \prod_{k=1}^n (1-u_{k-1}\ast \nu(S_k))^{L_k} \sum_{j=K_n+1}^{\lfloor n/\log n\rfloor} R_j P_j(S_j)   \right], \label{middlej}\\
\mathbb{E}_\eqref{largej}(x) := &  \mathbb{E}\left[   \ind{S_n=x} \prod_{k=1}^n (1-u_{k-1}\ast \nu(S_k))^{L_k} \sum_{j=\lfloor n/\log n\rfloor+1}^n R_j P_j(S_j)   \right] \label{largej}.
\end{align}

Similarly as \eqref{indepLR}, for any integers $a_n, b_n$ such that $1\le a_n < b_n \le n$, one has by independence
\begin{align}\label{Pnxa-b}
&\mathbb{E}_\eqref{Pnxa-b}(x):= \mathbb{E}\left[   \ind{S_n=x} \prod_{k=1}^n (1-u_{k-1}\ast \nu(S_k))^{L_k} \sum_{j=a_n}^{b_n} R_j P_j(S_j)   \right] \\
 =& \sum_{j=a_n}^{b_n} \mathbb{E}\left[  P_j(S_j) \ind{S_n=x} \prod_{k\neq j, 1\le k\le n} (1- u_{k-1}\ast \nu(S_k))^{L_k} H_j(S_j)\right].\nonumber
\end{align}
It is obvious that 
\[
    H_j(z)=\mathbb{E}\big[ R_j (1-u_{j-1}\ast \nu(z) )^{L_j} \big] \le \mathbb{E}[R_j] =\frac{\sigma^2}{2}.
\]
On the other hand, for any $h\in[0,1]$, $\mathbb{E}[(1-h)^{L_j}]\geq \mathbb{P}(L_j=0)>0$. 
It implies that 
\begin{equation}
    \label{eq:productL-lowerbdd}
    \inf_{z\in\Z^2} \inf_{j\ge1} \mathbb{E}\big[(1-u_{j-1}\ast \nu(z) )^{L_j}\big] \ge \mathbb{P}(L_1=0)>0.
\end{equation}
Hence, there exists some constant $C>0$ only depending on the offspring law such that for any $z\in\Z^2$ and for any $j\ge1$, 
\begin{equation}\label{upbdLR}
\mathbb{E}\big[ R_j (1-u_{j-1}\ast \nu(z) )^{L_j} \big] \le \mathbb{E}[R_j] \le C\cdot \mathbb{E}[R_j] \mathbb{E}\big[(1-u_{j-1}\ast \nu(z) )^{L_j}\big]\le \frac{C\sigma^2}{2}.
\end{equation}

Now applying \eqref{upbdLR} to \eqref{Pnxa-b} yields that uniformly in $x$,
\begin{align*}
\mathbb{E}_\eqref{Pnxa-b}(x) \lesssim &\,\sum_{j=a_n}^{b_n} \mathbb{E}\left[  P_j(S_j) \ind{S_n=x} \prod_{k\neq j, 1\le k\le n} (1- u_{k-1}\ast \nu(S_k))^{L_k} 
\right]\\
= &\,\mathbb{E}\left[ \ind{S_n=x}\prod_{k=1}^n (1-u_{k-1}\ast \nu(S_k))^{L_k} \sum_{j=a_n}^{b_n} P_j(S_j)\right]\\
\lesssim &\, \log\Big(\frac{b_n+1}{a_n}\Big)\mathbb{E}\left[ \ind{S_n=x}\prod_{k=1}^n (1-u_{k-1}\ast \nu(S_k))^{L_k} \right]\\
 = &  \log\Big(\frac{b_n+1}{a_n}\Big) u_n(x),
\end{align*}
where the last line comes from \eqref{spine-unx} and the fact that by \eqref{upbdPnx},
\[
\sum_{j=a_n}^{b_n} P_j(S_j) \lesssim \sum_{j=a_n}^{b_n} \frac1j \le \log\Big(\frac{b_n+1}{a_n}\Big).
\]
Then it follows that
\begin{align*}
\sup_{x\in\Z^2}\mathbb{E}_{\eqref{smallj}}(x) \lesssim &  \log\log n\times \sup_{x\in\Z^2}u_n(x),\\
\sup_{x\in\Z^2}\mathbb{E}_\eqref{largej}(x) \lesssim & \log\log n \times \sup_{x\in\Z^2}u_n(x).
\end{align*}
Together with \eqref{upbdu}, we get
\begin{align*}
\sup_{x\in\Z^2}\mathbb{E}_{\eqref{smallj}}(x) \lesssim & \frac{\log\log n}{n\log n},\\
\sup_{x\in\Z^2}\mathbb{E}_\eqref{largej}(x) \lesssim & \frac{\log\log n}{n\log n}.
\end{align*}

Under the assumption \eqref{hyp}, our previous arguments starting from \eqref{indepLR} up to \eqref{roughbdPnx} have shown that 
\[
    \mathbb{E}_{\eqref{middlej}}(x)=\mathbb{E}_{\eqref{roughbdPnx}}(x)(1+o_n(1)),
\]
with an $o_n(1)$ term that is uniform over all $x\in \Z^2$. Plugging \eqref{upbdu} into the right-hand side of \eqref{eq:estimate-roughbdPnx}, and using the fact that $\varepsilon$ in \eqref{eq:estimate-roughbdPnx} can be arbitrarily small, we deduce that
\[
    \mathbb{E}_{\eqref{middlej}}(x)=u_n(x)  \frac{5\sigma^2}{16\pi}\log n +\frac{o_n(1)}{n},
\]
where the $o_n(1)$ term is still uniform over all $x\in \Z^2$.
Combining all the estimates above for $\mathbb{E}_{\eqref{smallj}}(x), \mathbb{E}_{\eqref{middlej}}(x)$ and $\mathbb{E}_\eqref{largej}(x) $, we finally establish \eqref{eq:Pn-un-2nd}.

\medskip

Let us turn to prove \eqref{eq:Pn-un-3rd} when the offspring law is assumed to have a finite third moment. First, in \eqref{Pnx-decomp}, note that 
\[
\sup_{x\in\Z^2}\left[u_n(x) + \mathbb{E}_{\eqref{smallj}}(x) + \mathbb{E}_\eqref{largej}(x)\right] = \frac{O(1)}{n(\log n)^{\frac12}}.
\]
So, it suffices to verify that for arbitrarily small $\epsilon\in(0,\frac12)$,
\[
\sup_{x\in\Z^2}|\mathbb{E}_{\eqref{middlej}} - (\alpha \log n) u_n(x)| = \frac{O(1)}{n(\log n)^{\frac12-\epsilon}}.
\]
As the offspring law has a finite third moment, the expectation $\mathbb{E}[L_j R_j]$ is bounded. Under this additional assumption, we can improve our estimate of $\mathbb{E}_{\eqref{middlej}}(x) $. 
Recall that by \eqref{eq:RLdecouple}, for any $h\in (0,1)$,
\[
    \Big|\mathbb{E}\Big[R_j(1-h)^{L_j}\Big]- \mathbb{E}[R_j] \mathbb{E}\Big[(1-h)^{L_j}\Big]\Big| \le \mathbb{E}\big[ h R_j (L_j + \mathbb{E}[L_j])\big] \lesssim h.
\]
Hence, for any random variable $X$ independent of $(L_j, R_j)$ and for any functions $F$ and $G$ taking values in $(0,1)$, 
\[
    \Big|\mathbb{E}\Big[R_jF(X)(1-G(X))^{L_j}\Big]- \mathbb{E}[R_j] \mathbb{E}\Big[F(X)(1-G(X))^{L_j}\Big]\Big|  \lesssim \mathbb{E}[F(X)G(X)].
\]
Applying the last estimate, we see that for each $j$,
\begin{multline*}
     \Bigg|\mathbb{E}\bigg[ R_j \ind{S_n=x} \prod_{k=1}^n (1-u_{k-1}\ast \nu(S_k))^{L_k}  P_j(S_j)  \bigg]-\mathbb{E}[R_j]\mathbb{E}\bigg[ \ind{S_n=x} \prod_{k=1}^n (1-u_{k-1}\ast \nu(S_k))^{L_k}  P_j(S_j)  \bigg]\Bigg|\\
  \lesssim \,\mathbb{E}\Bigg[\ind{S_n=x} \bigg(\prod_{k\neq j,1\leq k\leq n}^n (1-u_{k-1}\ast \nu(S_k))^{L_k} \bigg) P_j(S_j) u_{j-1}\ast \nu(S_j)\Bigg].
\end{multline*}
Putting this into \eqref{middlej}, we have
\begin{align*}
 &\,\Bigg|\mathbb{E}_\eqref{middlej}(x)- \sum_{j=K_n+1}^{\lfloor n/\log n\rfloor}\mathbb{E}[R_j]\mathbb{E}\Bigg[ \ind{S_n=x} \prod_{k=1}^n (1-u_{k-1}\ast \nu(S_k))^{L_k}  P_j(S_j)   \Bigg] \Bigg|\\
\lesssim &\, \sum_{j=K_n+1}^{\lfloor n/\log n\rfloor} \mathbb{E}\Bigg[\ind{S_n=x} \bigg(\prod_{k\neq j,1\leq k\leq n} (1-u_{k-1}\ast \nu(S_k))^{L_k} \bigg)u_{j-1}\ast \nu(S_j) P_j(S_j) \Bigg]\\
\leq &\, \sum_{j=K_n+1}^{\lfloor n/\log n\rfloor} \sup_{z\in\Z^2} u_{j-1}(z)\cdot \mathbb{E}\Bigg[\ind{S_n=x} \bigg(\prod_{k\neq j,1\leq k\leq n} (1-u_{k-1}\ast \nu(S_k))^{L_k} \bigg) P_j(S_j) \Bigg]\\
\lesssim &\, \frac{1}{\log n\cdot \log\log n} \sum_{j=K_n+1}^{\lfloor n/\log n\rfloor} \mathbb{E}\Bigg[\ind{S_n=x} \bigg(\prod_{k\neq j,1\leq k\leq n} (1-u_{k-1}\ast \nu(S_k))^{L_k} \bigg) P_j(S_j) \Bigg],
\end{align*}
where we have used the uniform bound \eqref{upbdu} in the last line. 
Notice that by independence, we deduce from \eqref{eq:productL-lowerbdd} that
\[
     \mathbb{E}\Bigg[\ind{S_n=x} \bigg(\prod_{k\neq j,1\leq k\leq n}^n (1-u_{k-1}\ast \nu(S_k))^{L_k} \bigg) P_j(S_j) \Bigg]
    \lesssim \mathbb{E}\Bigg[\ind{S_n=x} \bigg(\prod_{k=1}^n (1-u_{k-1}\ast \nu(S_k))^{L_k} \bigg) P_j(S_j) \Bigg].
\]
Therefore, by use of \eqref{spine-Pnx} and \eqref{upbdPnx}, we know that uniformly for all $x\in \Z^2$,
\begin{align*}
&\,\Bigg|\mathbb{E}_\eqref{middlej}(x)- \sum_{j=K_n+1}^{\lfloor n/\log n\rfloor}\mathbb{E}[R_j]\mathbb{E}\Bigg[ \ind{S_n=x} \prod_{k=1}^n (1-u_{k-1}\ast \nu(S_k))^{L_k}  P_j(S_j)   \Bigg] \Bigg|\\
\lesssim  &\,\frac{1}{\log n\cdot \log\log n} \sum_{j=1}^{n} \mathbb{E}\Bigg[\ind{S_n=x} \bigg(\prod_{k=1}^n (1-u_{k-1}\ast \nu(S_k))^{L_k} \bigg) P_j(S_j) \Bigg]\\
\lesssim  &\, \frac{1}{\log n\cdot \log\log n} P_n(x) \lesssim  \frac{1}{n\log n\cdot \log\log n}.
\end{align*}

Finally, it remains to show that for any $\varepsilon\in (0,1/2)$,
\begin{align*}
\mathbb{E}_\eqref{roughbdPnx}(x)= &
    \sum_{j=K_n+1}^{\lfloor n/\log n\rfloor}\mathbb{E}[R_j]\mathbb{E}\Bigg[ \ind{S_n=x} \prod_{k=1}^n (1-u_{k-1}\ast \nu(S_k))^{L_k}  P_j(S_j)   \Bigg] \\
    = &\,(\alpha \log n) u_n(x) + \frac{O(1)}{n(\log n)^{1/2-\varepsilon}}.
\end{align*}
Recall that the sum on the left-hand side above is $\mathbb{E}_\eqref{roughbdPnx}(x)$. 
By considering the new event 
\[
    \overline{\mathrm{Bad}}_n:=\Big\{\Big\vert \Gamma_{K_n +1}^{\lfloor n/\log n \rfloor}(0) - \frac{5}{8\pi}\log n \Big\vert  \ge (\log n)^{\frac12+\varepsilon}\Big\}
\]
and its complementary, similarly as \eqref{Pnx-bad}, one has 
\begin{equation}
\label{eq:Pnx-bad-new}
\Big\vert \mathbb{E}_\eqref{roughbdPnx}(x) - u_n(x)  \frac{5\sigma^2}{16\pi}\log n \Big\vert \le  \frac{\sigma^2}{2}  \mathbb{E}_\eqref{eq:Pnx-bad-new}(x) +  \frac{\sigma^2}{2}  u_n(x) (\log n)^{\frac12+\varepsilon},
\end{equation}
where 
\[
\mathbb{E}_\eqref{eq:Pnx-bad-new}(x):=  \mathbb{E}\left[ \ind{S_n=x} \prod_{k=1}^n (1-u_{k-1}\ast \nu(S_k))^{L_k} \Big\vert \Gamma_{K_n+1}^{\lfloor n/\log n \rfloor}(0) - \frac{5}{8\pi}\log n \Big\vert \ind{\overline{\mathrm{Bad}}_n}\right] .
\]
Following the same arguments from \eqref{Pnx-bad} to \eqref{eq:estimate-roughbdPnx}, and using Lemma \ref{lem: SRW-sqrt}, we see that as $n\to \infty$,
\[
    \sup_{x\in \Z^2}\mathbb{E}_\eqref{eq:Pnx-bad-new}(x)=\frac{O(1)}{n(\log n)^{1/2-\varepsilon}}.
\]
Together with \eqref{upbdu} and \eqref{eq:Pnx-bad-new}, this completes the proof of \eqref{eq:Pn-un-3rd}.

\begin{remark}
Our proof of \eqref{asymp-unx} can be adapted to show that for any fixed bounded set $B\subset \Z^2$, uniformly in $x\in\Z^2$,
\begin{equation}\label{unxB}
u_n(x+B) = \frac{P_n(x)}{\alpha \log n} + \frac{o_n(1)}{n\log n} =\frac{P_n(x+B)}{\alpha |B|\log n} + \frac{o_n(1)}{n\log n}.
\end{equation}
The details are omitted.
\end{remark}

\section{Yaglom theorem: proof of Theorem \ref{thm: Yaglom}} 
\label{sec:yaglom_theorem_proof}
This section is devoted to proving Theorem \ref{thm: Yaglom}. 
Inspired by \cite{PR2011a}, our basic idea is to apply Stein's method for the exponential distribution. 
According to \cite[Theorem 3.5]{CJP}, for a non-negative random variable $W$ with a finite second moment, it holds that
\begin{equation}\label{Stein-key}
d_W(\mathscr{L}(W, \P), \mathrm{Exp}(1)) \le 2\E\big[|W^e- W|\big] + \big|\E[W]-1 \big|,
\end{equation}
where 
\begin{itemize}
\item $\mathscr{L}(W,\P)$ denotes the law of $W$ under $\P$, and $\textrm{Exp}(1)$ denotes the exponential distribution with parameter $1$;
\item $d_W(P, Q) := \sup\{|\int f \d P- \int f \d Q| \colon f\in \mathcal{F}_W\}$ denotes the Wasserstein 1-distance between two probability measures $P,Q$ on $\R$, where
\[
    \mathcal{F}_W:=\{f\colon \R\to\R \mid f\textrm{ is Lipschitz continuous with Lipschitz constant $\le 1$} \}.
\] 
\item $W^e$ denotes a random variable following the equilibrium law of $W$, meaning that for all $x\in\R_+$, 
\[
    \P(W^e \le x)= \frac{1}{\E [W]}\int_0^x \P(W>y) \d y.
\]
\end{itemize}
The inequality \eqref{Stein-key} in the case $\E[W]=1$ was initially stated in \cite[Theorem 2.1]{PR2011a}.

As mentioned in the introduction, we are going to show that, given a fixed bounded set $B\subset \Z^2$, uniformly for $z_n\in\{z\in\Z^2: \|z\| \le K\sqrt{n}\}$ with any fixed $K>0$, as $n\to \infty$, we have
\begin{equation}\label{Wdistance}
d_W\bigg( \mathscr{L} \Big(\frac{Z_n(z_n+B)}{ \alpha |B|\log n}, \P(\cdot\vert Z_n(z_n+B) \ge 1)\Big), \mathrm{Exp}(1) \bigg) = o_n(1).
\end{equation}

Notice that by \eqref{unxB}, uniformly for $z_n\in\{z\in\Z^2: \|z\| \le K\sqrt{n}\}$,
\begin{equation}
    \label{eq:mean-1}
    \E\bigg[\frac{Z_n(z_n+B)}{ \alpha |B|\log n} \Big\vert Z_n(z_n+B) \ge 1\bigg] = \frac{1}{\alpha |B|\log n} \cdot \frac{P_n(z_n+B)}{u_n(z_n+B)}= 1+o_n(1).
\end{equation}
In order to bound the Wasserstein distance between the law of $W$ and $\textrm{Exp}(1)$ by \eqref{Stein-key}, we need to couple $W^e$ and $W$ in a suitable way. 
In what follows, let us start with several preliminary facts. 
Then, in Section~\ref{sub:construction} we will use the spinal decomposition to reconstruct $\mathscr{L}( \frac{Z_n(z_n+B)}{\alpha |B| \log n}, \P(\cdot\vert Z_n(z_n+B) \ge 1) )$ and the corresponding equilibrium law. Finally, in Section~\ref{sub:coupling} we will couple them in some common probability space $(\Omega, \mathcal{F}, \P^{2\textrm{-spine}})$.

\subsection{Preliminary facts to apply Stein's method}

Let us first verify that the second moment of $Z_n(z_n+B)$ conditioned on $Z_n(z_n+B) \ge 1$ is finite. By change of measure,
\begin{align*}
&\E\big[Z_n(z_n+B)^2\big] =  \,\E\Bigg[ \sum_{|u|=n} \ind{S_u\in z_n +B} Z_n(z_n+B) \Bigg]\\
=\,& \,\E_{\Q^*}\Bigg[ \ind{S_{w_n} \in z_n+B} \Bigg( 1+ \sum_{j=1}^n \sum_{u\in\cL(w_j)\cup\cR(w_j)} Z^u_{n-j}(z_n+B-S_u) \Bigg) \Bigg].
\end{align*}
By Proposition \ref{spine} and the fact that $\E[L(w_j)]=\E[R(w_j)]=\sigma^2/2$, the last expectation is equal to 
\[
 \mathbb{E}\left[ \ind{S_{n} \in z_n+B} \left( 1+ \sum_{j=1}^n \sigma^2 P_{n-j+1}(z_n+B - S_{j-1})\right) \right].
 \]
It follows from \eqref{upbdPnx} that
\[
\E\big[Z_n(z_n+B)^2\big] \lesssim  \left( 1+ \sum_{j=1}^n \frac{|B|}{n-j+1} \right) P_n(z_n+B)\lesssim \frac{\log n}{n} |B|^2 .
\]
On the other hand, for $\|z_n\|\le K\sqrt{n}$, then \eqref{Pnx} implies that
\[
P_n(z_n) \ge \inf_{\|z\|\le K\sqrt{n} } P_n(z) \gtrsim \frac1n, 
\]
Together with \eqref{unxB}, we see that 
\[
u_n(z_n+B) = \frac{P_n(z_n)}{\alpha\log n} + \frac{o_n(1)}{n\log n} \gtrsim \frac{1}{n\log n}.
\]
Consequently, there exists some finite constant $C>0$ such that for all sufficiently large $n$ and $\|z_n\|\le K\sqrt{n}$,
\begin{equation}\label{bd2mom}
\E\Bigg[ \left(\frac{Z_n(z_n+B)}{\alpha |B|\log n}\right)^2 \bigg\vert Z_n(z_n+B) \ge 1\Bigg] \le C.
\end{equation}

Secondly, we present a construction of the equilibrium law via size biasing for a non-negative random variable $W$ with finite mean under $\P$. Let $W^{sb}$ be a non-negative random variable following the corresponding size-biased law, that is, for any continuous and bounded function $f\colon\R_+\to\R$,
\[
\E[f(W^{sb})] = \frac{1}{\E[W]} \E[ W f(W)].
\]
Take a uniform random variable $U$ in $(0,1)$ which is independent of all other random variables. Then the product $UW^{sb}$ has the law of $W^e$, see for instance \cite[Section 2.1.1]{PR2011a}. Moreover, let $W^+$ be a random variable distributed as $\mathscr{L}(W, \P(\cdot\vert W>0))$. It is clear that $(W^+)^{sb}$ has the same law as $W^{sb}$. Consequently, 
\begin{equation}\label{uwsb}
\mathscr{L}( (W^+)^e, \P) = \mathscr{L}( UW^{sb}, \P)=\mathscr{L}(W^e, \P) .
\end{equation}
Notice that $\mathscr{L}( W, \P) = \mathscr{L}( UW^{sb}, \P)$ if and only if $W$ is exponentially distributed under $\P$.

\subsection{Construct the conditional law and the corresponding equilibrium law}
\label{sub:construction}

For notational ease, we write
\[
Y_n := \frac{Z_n(z_n+B)}{\alpha |B|\log n},
\]
and let $Y_n^+$ be a random variable distributed as $\mathscr{L}( \frac{Z_n(z_n+B)}{\alpha |B|\log n}, \P(\cdot\vert Z_n(z_n+B) \ge 1) )$ under $\P$. In view of \eqref{uwsb}, we only need to construct $UY_n^{sb}$ which has the equilibrium law of $Y_n^+$. We claim that 
\[
\mathscr{L}( Y_n, \Q^*(\cdot\vert S_{w_n}\in z_n +B)) = \mathscr{L}( Y_n^{sb}, \P).
\]
In fact, for any  continuous and bounded function $f\colon \R\to\R$, by change of measure, we have
\begin{align*}
\E\big[f(Y_n^{sb})\big] = & \,\frac{1}{\E[Z_n(z_n+B)]} \E[ Z_n(z_n + B) f(Y_n)]\\
=&\,  \frac{1}{P_n(z_n +B) } \E_{\Q^*}\Bigg[ \sum_{|u|=n} \frac{1}{Z_n} \ind{S_u\in z_n +B} f(Y_n) \Bigg] \\
= &\, \frac{1}{P_n(z_n +B) } \E_{\Q^*}\left[ \sum_{|u|=n} \Q^*(w_n = u \vert \nf_n) \ind{S_u\in z_n +B} f(Y_n) \right]\\
= & \,\frac{1}{\Q^*(S_{w_n}\in z_n +B)} \E_{\Q^*}\big[\ind{S_{w_n}\in z_n +B} f(Y_n)\big] \\
= & \,\E_{\Q^*}\big[ f(Y_n) \vert S_{w_n}\in z_n +B \big].
\end{align*}
Consider the lexicographic order $\prec$ of individuals on $\T$. For any $u$ at generation $n$, we define
\[
    \mathsf{L}^{\prec u}_n(B):= |\{v\in\T\colon |v|=n, v\prec u, S_v\in B\}| 
    \textrm{ and } \mathsf{R}^{\succ u}_n(B):= |\{v\in\T \colon |v|=n, u\prec v, S_v\in B\}|.
\]

\begin{lemma}
    Under $\Q^*(\cdot\vert S_{w_n}\in z_n +B)$, conditionally on $Z_n(z_n+B)$, $\mathsf{R}_n^{\succ w_n}(z_n+B)$ is uniformly distributed in $\{0,\cdots, Z_n(z_n+B)-1\}$. Furthermore, if $U^*$ is a uniform random variable in $(0,1)$ independent of the branching random walk with a marked spine under $\Q^*$, then
\begin{equation}\label{uysb}
\mathscr{L}\bigg(  \frac{ \mathsf{R}_n^{\succ w_n}(z_n+B) + U^* }{\alpha|B| \log n} , \Q^*(\cdot\vert S_{w_n}\in z_n +B)\bigg) = \mathscr{L}( UY_n^{sb}, \P) = \mathscr{L}((Y_n^+)^e, \P).
\end{equation}
\end{lemma}

\begin{proof}
For every $m\ge 1$ and every $1\le i\le m$, observe that
\begin{align*}
&\Q^*\big( \mathsf{R}_n^{\succ w_n}(z_n+B) = i-1, Z_n(z_n+B) = m \vert S_{w_n}\in z_n +B\big) \\
= & \, \frac{1}{P_n(z_n +B) }\E_{\Q^*}\left[ \sum_{|u|=n} \ind{w_n = u, S_u\in z_n+B} \ind{ \mathsf{R}_n^{\succ u}(z_n+B) = i-1, Z_n(z_n+B) = m}\right]\\
= & \, \frac{1}{P_n(z_n +B) }\E_{\Q^*}\left[ \sum_{|u|=n} \frac{1}{Z_n}\ind{ S_u\in z_n+B} \ind{\mathsf{R}_n^{\succ u}(z_n+B) = i-1}\ind{Z_n(z_n+B)=m}\right]\\
=&  \, \frac{1}{P_n(z_n +B) }\E_{\Q^*}\left[\frac{1}{Z_n} \ind{Z_n(z_n+B)=m}\right],
\end{align*}
which is independent of $i$.
Consequently, for any continuous and bounded function~$f$, 
\begin{align*}
&\E_{\Q^*}\big[ f(  \mathsf{R}_n^{\succ w_n}(z_n+B) +  U^* ) \vert S_{w_n}\in z_n +B\big] \\
= \,&\sum_{i\ge0}\sum_{m\ge i+1} \int_0^1 f(i+u) \Q^*\big(  \mathsf{R}_n^{\succ w_n}(z_n+B) = i, Z_n(z_n+B) = m \vert S_{w_n}\in z_n +B\big) \d u\\
=\, & \sum_{m=1}^\infty  \frac{1}{P_n(z_n +B) }\E_{\Q^*}\left[\frac{1}{Z_n} \ind{Z_n(z_n+B)=m}\right] \sum_{i=1}^m \int_0^1 f(i-1+u)\d u .
\end{align*}
Because $\sum_{i=1}^m \int_0^1 f(i-1+u)\d u = \int_0^m f(r)\d r =m \int_0^1 f(m u) \d u$, we have
\begin{align*}
&\E_{\Q^*}\big[ f(  \mathsf{R}_n^{\succ w_n}(z_n+B) +  U^* ) \vert S_{w_n}\in z_n +B\big] \\
=\,&  \frac{1}{P_n(z_n +B) }\int_0^1 \E_{\Q^*}\left[\sum_{m=1}^\infty \frac{f(mu)m}{Z_n}\ind{Z_n(z_n+B)=m}\right] \d u\\
=\,&  \frac{1}{P_n(z_n +B) } \int_0^1 \E_{\Q^*}\left[ \frac{f( Z_n(z_n+B)u)Z_n(z_n+B)}{Z_n} \right] \d u.
\end{align*}
By change of measure, $\E_{\Q^*}\left[ \frac{f( Z_n(z_n+B)u)Z_n(z_n+B)}{Z_n} \right] = \E\left[ f(u Z_n(z_n+B))Z_n(z_n+B) \right]$. We thus end up with 
\begin{align*}
&\E_{\Q^*}\big[ f(  \mathsf{R}_n^{\succ w_n}(z_n+B) +U^* ) \vert S_{w_n}\in z_n +B\big] \\
=\, &\frac{1}{P_n(z_n +B) } \int_0^1 \E\left[ f( Z_n(z_n+B)u) Z_n(z_n+B)\right] \d u = \E\left[ f(UZ_n^{sb}(z_n +B) ) \right].
\end{align*}
Then \eqref{uwsb} allows us to conclude. 
\end{proof} 

Next, let us reconstruct $\mathscr{L}(Y_n^+, \P)$ from the branching random walk with a marked ray. For non-negative random variables, it suffices to consider their Laplace transform. For any $\lambda>0$, 
\begin{align*}
\E\left[ e^{-\lambda Y_n^+ } \right] = & \E\left[ \exp\Big(\!-\lambda \frac{Z_n(z_n+B)}{\alpha |B|\log n}\Big) \Big\vert Z_n(z_n+B)\ge 1\right] \\
= & \frac{1}{\P(Z_n(z_n+B)\ge 1 ) } \E\left[ \exp\Big(\!-\lambda \frac{Z_n(z_n+B)}{\alpha |B|\log n}\Big) \ind{Z_n(z_n+B)\ge 1}\right].
\end{align*}
Similarly to the derivation of \eqref{changeofm-pnx}, it follows from the change of measure that 
\begin{align*}
&\E\left[ e^{-\lambda Y_n^+ } \right] =  \frac{\E_{\Q^*}\left[ \exp\Big(-\lambda \frac{ 1+\mathsf{R}_n^{\succ w_n}(z_n +B)}{\alpha|B|\log n} \Big)\ind{ \mathsf{L}_n^{\prec w_n}(z_n+B)=0, S_{w_n}\in z_n +B)}\right]}{\Q^*(\mathsf{L}_n^{\prec w_n}(z_n+B)=0, S_{w_n}\in z_n +B)}  \\
=& \sum_{\substack{x_1,\cdots, x_{n-1}\in \Z^2\\ x_n \in z_n+B}} \frac{ \Q^*( \mathsf{L}_n^{\prec w_n}(z_n+B)=0, (S_{w_k})_{1\le k\le n} =(x_k)_{1\le k\le n}) }{ \Q^*(\mathsf{L}_n^{\prec w_n}(z_n+B)=0, S_{w_n}\in z_n +B) } \\
& \qquad \times \E_{\Q^*}\left[ \exp\Big(\!-\lambda \frac{1+ \mathsf{R}_n^{\succ w_n}(z_n +B) }{\alpha|B|\log n} \Big) \Big\vert \,\begin{array}{l}\mathsf{L}_n^{\prec w_n}(z_n+B)=0, \\(S_{w_k})_{1\le k\le n} =(x_k)_{1\le k\le n} \end{array}\right].
\end{align*}
On the one hand, we have
\begin{align*}
 &\frac{ \Q^*( \mathsf{L}_n^{\prec w_n}(z_n+B)=0, (S_{w_k})_{1\le k\le n} =(x_k)_{1\le k\le n}) }{ \Q^*(\mathsf{L}_n^{\prec w_n}(z_n+B)=0, S_{w_n}\in z_n +B) } \\
 =\,& \Q^*\big( (S_{w_k})_{1\le k\le n} =(x_k)_{1\le k\le n}) \vert \mathsf{L}_n^{\prec w_n}(z_n+B)=0, S_{w_n}\in z_n +B \big). 
\end{align*}
On the other hand,  we set $\lambda_n := \frac{\lambda}{\alpha |B|\log n}$ and note that
\begin{align}
& \mathsf{R}_n^{\succ w_n}(z_n +B) =  \sum_{j=1}^n \mathsf{R}_{n,j}(z_n+B),  \mathsf{L}_n^{\prec w_n}(z_n +B) =   \sum_{j=1}^n \mathsf{L}_{n,j}(z_n+B)\textrm{ with }\label{LRn}\\
&
\mathsf{R}_{n,j}(z_n+B):= \sum_{u\in\cR(w_j)}Z_{n-j}^u(z_n+B-S_u), \quad
 \mathsf{L}_{n,j}(z_n+B):=  \sum_{u\in\cL(w_j)}Z_{n-j}^u(z_n+B-S_u). 
\label{LRnj}
\end{align}
For each $u$ in $\cR(w_j)$ or $\cL(w_j)$, let $X_u:=S_u-S_{w_{j-1}}$ denote the spatial jump from $w_{j-1}$ to $u$.
It follows from independence that
\begin{align*}
&\E_{\Q^*}\left[ \exp\Big(\!-\lambda \frac{1+ \mathsf{R}_n^{\succ w_n}(z_n +B) }{\alpha|B|\log n} \Big) \Big\vert \,
     \mathsf{L}_n^{\prec w_n}(z_n+B)=0, (S_{w_k})_{1\le k\le n} =(x_k)_{1\le k\le n} \right] \\
=\, & e^{-\lambda_n} \prod_{j=1}^n \E_{\Q^*_{(x_k)_{0\le k\le n}}}\Bigg[ e^{-\lambda_n \sum_{u\in\cR(w_j)} Z^u_{n-j}(z_n+B-x_{j-1}-X_u) } \bigg\vert \sum_{v\in\cL(w_j)} Z^v_{n-j} (z_n+B-x_{j-1}-X_v) =0 \Bigg]\\
=\, & e^{-\lambda_n} \prod_{j=1}^n \E_{\Q^*_{(x_k)_{0\le k\le n}}} \left[ e^{-\lambda_n \mathsf{R}_{n,j}(z_n+B) } \big\vert \, \mathsf{L}_{n,j}(z_n+B) =0 \right],
\end{align*}
where $\Q^*_{(x_k)_{0\le k\le n}}(\cdot ):= \Q^*(\cdot \vert (S_{w_k})_{1\le k\le n} = (x_{k})_{1\le k\le n})$ and $x_0=0$. 

Therefore, we end up with
\begin{align}\label{lawofW}
\E\left[ e^{-\lambda Y_n^+ } \right] = &\sum_{x_1,\cdots, x_{n-1}\in\Z^2, x_n\in z_n+B} \Q^*\big( (S_{w_k})_{1\le k\le n} =(x_k)_{1\le k\le n}) \vert \mathsf{L}_n^{\prec w_n}(z_n+B)=0, S_{w_n}\in z_n +B \big) \nonumber\\
&\qquad\qquad\qquad \times  e^{-\lambda_n} \prod_{j=1}^n \E_{\Q^*_{(x_k)_{0\le k\le n}}} \left[ e^{-\lambda_n \mathsf{R}_{n,j}(z_n+B) } \big\vert \mathsf{L}_{n,j}(z_n+B) =0 \right].
\end{align}
This identity leads to the following construction of $\P^{\textrm{spine-I}}$:
\begin{enumerate}
\item We first take a finite ray $(w_0,\cdots, w_n)$ to be the spine up to the $n$-th generation. It starts at the root $w_0=\varnothing$.
\item Take a random vector $(S^I_{w_k})_{0\le k\le n}$ with $S^I_{w_0}=x_0=0$ to be the corresponding path of the spine in $\Z^2$. It satisfies that for any $\{x_{k},1\le k\le n\}\subset \Z^2$,
\begin{multline}\label{spine-1}
\P^{\textrm{spine-I}}\big( (S^I_{w_k})_{0\le k\le n} = (x_k)_{0\le k\le n}\big)  \\
= \Q^*\big( (S_{w_k})_{0\le k\le n} =(x_k)_{0\le k\le n}) \vert \mathsf{L}_n^{\prec w_n}(z_n+B)=0, S_{w_n}\in z_n +B \big).
\end{multline}
\item Let $(L^I(w_j), R^I(w_j))_{1\le j\le n}$ be i.i.d.~random vectors under $\P^{\textrm{spine-I}}$, distributed as $\mathscr{L}((L_1,R_1),\P)$ and independent of $(S^I_{w_k})_{1\le k\le n}$. Conditionally on the path $(S^I_{w_k})_{0\le k\le n} = (x_k)_{0\le k\le n}$, to each $w_j$ on the spine we attach $L^I(w_j)$ siblings on the left and $R^I(w_j)$ siblings on the right, each sibling $u$ of $w_j$ makes an independent jump $X^I_u$ from $x_{j-1}$ according to the jump law $\nu$. We set $S^I_u = x_{j-1} +X^I_u$. Again, we denote the set of the left siblings and the right siblings of $w_j$ by $\cL^I(w_j)$ and $\cR^I(w_j)$ respectively.
\item For each $u\in\cup_{j=1}^n\cL^I(w_j)\cup\cR^I(w_j)$, run under $\P^{\textrm{spine-I}}$ an independent $\P$-distributed CBRW starting from $S^I_u$. 
Similarly as in \eqref{LRnj}, let $\mathsf{L}^I_{n,j}(z_n+B)$ stand for the number of individuals at generation~$n$ which are located in $z_n+B$ and are descendants of $ \cL^I(w_j)$. In the same way we define $\mathsf{R}^I_{n,j}(z_n+B)$ by looking at the descendants of $\cR^I(w_j)$ at generation $n$. 
\item Given the path $(S^I_{w_k})_{0\le k\le n} = (x_k)_{0\le k\le n}$, take $\{\mathsf{R}'_{n,j}(z_n+B)\}_{1\le j\le n}$ to be independent random variables which are independent of all the other random variables such that $\mathsf{R}_{n,j}'(z_n+B)$ is distributed as $(\mathsf{R}^I_{n,j}(z_n+B)|\mathsf{L}^I_{n,j}(z_n+B)=0)$.
\item We define 
\[
\mathsf{R}^*_{n,j}(z_n+B) := \mathsf{R}^I_{n,j}(z_n + B) \ind{\mathsf{L}^I_{n,j}(z_n+B) = 0} + \mathsf{R}'_{n,j}(z_n +B) \ind{ \mathsf{L}^I_{n,j}(z_n+B) >0},
\]
and
\begin{equation}\label{modifiedRn}
\mathsf{R}^*_n(z_n+B):=\sum_{j=1}^n \mathsf{R}^*_{n,j}(z_n+B).
\end{equation}
\end{enumerate}
According to the construction above, $\mathsf{R}^*_{n,j}(z_n+B)$ has the same distribution as $\mathsf{R}^I_{n,j}(z_n+B)$ conditioned on $\mathsf{L}^I_{n,j}(z_n+B)=0$.
So we obtain from \eqref{lawofW} that 
\begin{equation}\label{w-spine}
\mathscr{L}(Y_n^+, \P)=\mathscr{L}\bigg( \frac{1+\mathsf{R}^*_n(z_n+B)}{\alpha|B|\log n} ,\P^{\textrm{spine-I}} \bigg).
\end{equation}
To summarize, we have constructed the law of $Y_n^+$ and its equilibrium law in \eqref{w-spine} and \eqref{uysb} respectively.

\subsection{Coupling of \texorpdfstring{$Y_n^+$ and $(Y_n^+)^e$}{}}
\label{sub:coupling}

In this section, we couple  
\[
    \mathscr{L}\bigg( \frac{1+\mathsf{R}^*_n(z_n+B)}{\alpha|B|\log n} ,\P^{\textrm{spine-I}} \bigg) \textrm{ and } \mathscr{L}\bigg(  \frac{ \mathsf{R}_n^{\succ w_n}(z_n+B) + U^* }{\alpha|B| \log n} , \Q^*(\cdot\vert S_{w_n}\in z_n +B)\bigg) 
\]
in the same probability space. Note that the difference between $\P^{\textrm{spine-I}}$ and $\Q^*(\cdot\vert S_{w_n}\in z_n +B) $ comes from the different trajectories of the spine, and also from the difference between $\mathsf{R}_n^*$ and $\mathsf{R}^{\succ w_n}_n$.

To control the first difference, 
fix $\delta\in(0,1/2)$, and
let $(S^{I\!I}_{w_k})_{0\le k\le n}$ be a random vector under $\P$, distributed as $\Q^*( (S_{w_k})_{0\le k\le n} \in\cdot \vert S_{w_n}\in z_n +B)$. We are going to study the total variation distance between the law of $(S^I_{w_k})_{0\le k\le m_n}$ and that of $(S^{I\!I}_{w_k})_{0\le k\le m_n}$ with
\begin{align}\label{eq:mn}
m_n:= n- \lfloor n^\delta\rfloor.
\end{align}
Denote this total variation distance by 
\begin{equation*}
D_{TV}(m_n):= \| \mathscr{L}( (S^I_{w_k})_{k\le m_n},\P^{\textrm{spine-I}}) - \mathscr{L}( (S^{I\!I}_{w_k})_{k\le m_n},\P)  \|_{TV}.
\end{equation*}
We claim that uniformly in $\|z_n\|\le K\sqrt{n}$,
\begin{equation}\label{bdDTVspine}
D_{TV}(m_n) = o_n(1).
\end{equation} 
The proof of \eqref{bdDTVspine} is postponed to Section \ref{boundDTV}.

Now we are ready to couple $Y_n^+$ and $(Y_n^+)^e$ under the same probability measure $\P^{\textrm{2-spine}}$. 
First, thanks to the optimal coupling for the total variation distance, we can couple two trajectories $(S_{w_k^I})_{k\le n}$ and $(S_{w_k^{I\!I}})_{k\le n}$ of the spines $(w_k^I)_{k\le n}$ and $(w_k^{I\!I})_{k\le n}$ under $\P^{\textrm{2-spine}}$ such that
\begin{enumerate}
\item $\mathscr{L}( (S_{w_k^I})_{k\le n}, \P^{\textrm{2-spine}}) = \mathscr{L}( (S^I_{w_k})_{k\le n},\P^{\textrm{spine-I}}) $,
\item $\mathscr{L}( (S_{w_k^{I\!I}})_{k\le n}, \P^{\textrm{2-spine}}) =  \mathscr{L}( (S^{I\!I}_{w_k})_{k\le n},\P) $,
\item $\P^{\textrm{2-spine}}\left( (S_{w_k^I})_{k\le m_n}\neq (S_{w_k^{I\!I}})_{k\le m_n} \right) = D_{TV}(m_n) = o_n(1)$.
\end{enumerate}
Both trajectories $(S_{w_k^I})_{k\le n}$ and $(S_{w_k^{I\!I}})_{k\le n}$ start from $0\in\Z^2$, and end at some site in $z_n+B$.
Let
\[
\mathcal{S}^{\textrm{2-spine}}:=\max\big\{0\le m \le n \vert (S_{w_k^I})_{k\le m} = (S_{w_k^{I\!I}})_{k\le m} \big\} \in [0,n]\cap \mathbb{Z}
\]
be the separating time of the two spines. Then uniformly in $\|z_n\|\le K\sqrt{n}$,
\begin{equation}\label{separating}
\P^{\textrm{2-spine}}\left( \mathcal{S}^{\textrm{2-spine}} \ge m_n \right) = 1-o_n(1).
\end{equation}

Secondly, let us construct the siblings of $(w_k^I)_{1\le k\le n}$ and $(w_k^{I\!I})_{1\le k\le n}$ under $\P^{\textrm{2-spine}}$. Given the trajectories $(S_{w_k^I})_{k\le n}$ and $(S_{w_k^{I\!I}})_{k\le n}$, 
\begin{itemize}
    \item for all $0\leq k\leq \mathcal{S}^{\textrm{2-spine}}$, we do not distinguish $w_k^I$ and $w_k^{I\!I}$ and let them share the same siblings on the left and right of the spine, denoted by $\cL_k$ and $\cR_k$; 
    \item while for $\mathcal{S}^{\textrm{2-spine}} <  k \le n$, each of $w_k^I$ and $w_k^{I\!I}$ has its own left and right siblings, denoted by $(\cL_k^I, \cR_k^I)$ and $(\cL^{I\!I}_k, \cR_k^{I\!I})$ respectively, in such a way that $(|\cL_k|, |\cR_k|)_{k\le \mathcal{S}^{\textrm{2-spine}}}$, $(|\cL_k^I|, |\cR_k^I|)_{\mathcal{S}^{\textrm{2-spine}} <  k \le n}$,  $(|\cL_k^{I\!I}|, |\cR_k^{I\!I}|)_{\mathcal{S}^{\textrm{2-spine}} <  k \le n}$ are all i.i.d.~random vectors distributed as $\mathscr{L}((L_1,R_1),\P)$. 
\end{itemize}
Next, the spatial positions of those siblings are as follows:
\begin{enumerate}
\item for $1\le k\le \mathcal{S}^{\textrm{2-spine}} $ and $u\in \cL_k\cup\cR_k$, let $X_u$ be an independent random jump distributed as $\nu$, and let the position of $u$ be $S_u = X_u + S_{w^I_{k-1}}$;
\item for $\mathcal{S}^{\textrm{2-spine}} <  k \le n$ and $u \in \cL_k^I\cup\cR^I_k$ (or $u\in \cL_k^{I\!I}\cup\cR_k^{I\!I}$), let $X_u$ be an independent  jump distributed as $\nu$, and let the position of $u$ be $S_u = X_u + S_{w^I_{k-1}}$ (or $S_u= X_u + S_{w_{k-1}^{I\!I}}$ respectively);
\item to any $u$ that belongs to
\[
\bigg(\bigcup_{1\le k\le \mathcal{S}^{\textrm{2-spine}}  } (\cL_k\cup\cR_k)\bigg) \bigcup \bigg(\bigcup_{ \mathcal{S}^{\textrm{2-spine}} <  k \le n }(\cL_k^I\cup\cR_k^I)\cup(\cL_k^{I\!I}\cup\cR_k^{I\!I})\bigg),
\] 
attach an independent $\P$-distributed CBRW starting from $S_u$.
\end{enumerate}

\begin{figure}
    \centering
    \includegraphics[scale=1]{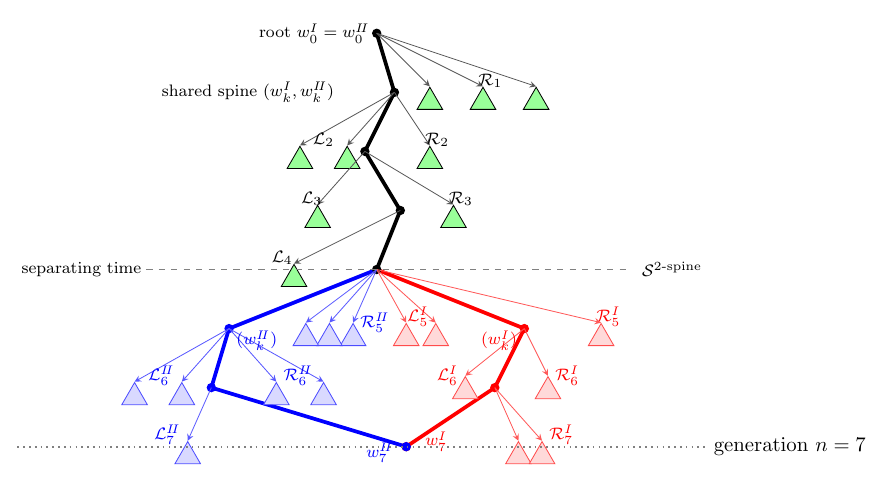}
    \caption{Construction of branching random walk with 2 spines.}
    \label{fig:placeholder}
\end{figure}

Similarly as above, we count the number of individuals at the $n$-th generation which are located in $z_n+B$ and are descendants of the left and the right siblings of $w^I_j$ (or that of $w^{I\!I}_j$). 
Such numbers are still denoted as $\mathsf{L}_{n,j}^I(z_n+B)$ and $\mathsf{R}_{n,j}^I(z_n+B)$ (or $\mathsf{L}_{n,j}^{I\!I}(z_n+B)$ and $\mathsf{R}^{I\!I}_{n,j}(z_n+B)$ respectively).  
Before the separation, for all $1\le j\le \mathcal{S}^{\textrm{2-spine}}$,
\[
(\mathsf{L}_{n,j}^I(z_n+B), \mathsf{R}_{n,j}^I(z_n+B))=(\mathsf{L}_{n,j}^{I\!I}(z_n+B),  \mathsf{R}^{I\!I}_{n,j}(z_n+B)).
\]
Moreover, let
\[
\mathsf{R}^I_n(z_n+B):= \sum_{j=1}^n \mathsf{R}_{n,j}^I(z_n+B) \textrm{ and } \mathsf{R}^{I\!I}_n(z_n+B):= \sum_{j=1}^n \mathsf{R}_{n,j}^{I\!I}(z_n+B).
\]
Under  $\P^{\textrm{2-spine}}$, let $U^{I\!I}$ be an independent uniform random variable in $(0,1)$. Then, \eqref{uysb} yields that 
\[
\mathscr{L}\bigg( \frac{\mathsf{R}^{I\!I}_n(z_n+B) + U^{I\!I}}{\alpha|B| \log n}, \P^{\textrm{2-spine}}\bigg) = \mathscr{L}( (Y_n^+)^e, \P).
\]
On the other hand, with a little abuse of notation, we define $\{\mathsf{R}'_{n,j}(z_n+B)\}_{1\le j\le n}$ in the same way under $\P^{\textrm{2-spine}}$ as in the construction of $\P^{\textrm{spine-I}}$.
Similarly to \eqref{modifiedRn}, let 
\[
\mathsf{R}^*_{n,j}(z_n+B) := \mathsf{R}^I_{n,j}(z_n + B) \ind{\mathsf{L}^I_{n,j}(z_n+B) = 0} + \mathsf{R}'_{n,j}(z_n +B) \ind{ \mathsf{L}^I_{n,j}(z_n+B) >0}
\]
and $\mathsf{R}^*_n(z_n+B):=\sum_{j=1}^n \mathsf{R}^*_{n,j}(z_n+B)$.  Then, we obtain the analog of \eqref{w-spine} that
\[
\mathscr{L}\bigg( \frac{1+\mathsf{R}^{*}_n(z_n+B)}{\alpha |B|\log n}, \P^{\textrm{2-spine}}\bigg) = \mathscr{L}( Y_n^+, \P).
\]
Together with \eqref{uysb}, it follows from Stein's method \eqref{Stein-key} that
\begin{align}\label{Steinbd}
&d_W(  \mathscr{L}( Y_n^+, \P), \mathrm{Exp}(1))\le  2 \E_{\P^{\textrm{2-spine}}} \left[ \left| \frac{\mathsf{R}^{I\!I}_n(z_n+B) + U^{I\!I}}{\alpha|B| \log n} -   \frac{\mathsf{R}^{*}_n(z_n+B) + 1}{\alpha|B| \log n} \right |\right] +|\E[Y_n^+]-1|\\
 &\qquad \lesssim  \frac{1}{\log n} + |\E[Y_n^+]-1|  +  \E_{\P^{\textrm{2-spine}}} \left[ \left| \frac{\mathsf{R}^{I\!I}_n(z_n+B) }{\alpha|B| \log n} -   \frac{\mathsf{R}^{*}_n(z_n+B) }{\alpha|B| \log n} \right |\ind{ \mathcal{S}^{\textrm{2-spine}} < m_n }\right]\nonumber\\
&\qquad \quad +\E_{\P^{\textrm{2-spine}}} \left[ \left| \frac{\mathsf{R}^{I\!I}_n(z_n+B) }{\alpha|B| \log n} -   \frac{\mathsf{R}^{*}_n(z_n+B) }{\alpha |B|\log n} \right |\ind{ \mathcal{S}^{\textrm{2-spine}} \ge m_n }\right].\nonumber
\end{align}
By \eqref{eq:mean-1}, $|\E[Y_n^+]-1|=o_n(1)$. 
Then, by Cauchy--Schwarz inequality, \eqref{bd2mom} and \eqref{separating}, we know that uniformly in $\|z_n\|\le K\sqrt{n}$,
\begin{align*}
\E_{ \P^{\textrm{2-spine}}} \left[  \frac{\mathsf{R}^{*}_n(z_n+B) }{\alpha|B| \log n}  \ind{ \mathcal{S}^{\textrm{2-spine}} < m_n }\right] \le &\,\E_{\P^{\textrm{2-spine}}} \left[  \left(\frac{\mathsf{R}^{*}_n(z_n+B) }{\alpha|B| \log n} \right)^2 \right]^{1/2} \P^{\textrm{2-spine}}( \mathcal{S}^{\textrm{2-spine}} < m_n)^{1/2}\\
\le &\, \E[(Y_n^+)^2]^{1/2} \P^{\textrm{2-spine}}( \mathcal{S}^{\textrm{2-spine}} < m_n)^{1/2}=o_n(1).
\end{align*}
On the other hand, by conditioning on the two spines,
\begin{align}
 &\,\E_{\P^{\textrm{2-spine}}} \left[  \frac{\mathsf{R}^{I\!I}_n(z_n+B)}{\alpha|B| \log n} \ind{ \mathcal{S}^{\textrm{2-spine}} < m_n }\right] \nonumber\\
 =& \, \E_{\P^{\textrm{2-spine}}} \left[  \E_{\P^{\textrm{2-spine}}} \left[ \frac{\sum_{j=1}^n \mathsf{R}^{I\!I}_{n,j}(z_n+B)}{\alpha|B| \log n} \Big\vert (S_{w_k^I}, S_{w_k^{I\!I}})_{k\le n} \right] \ind{ \mathcal{S}^{\textrm{2-spine}} < m_n }\right] \nonumber\\
 = & \, \E_{\P^{\textrm{2-spine}}} \left[ \frac{ \sum_{j=1}^n \E[R_1] P_{n-j+1}(z_n+B-S_{w_{j-1}^{I\!I}}) }{\alpha|B|\log n} \ind{ \mathcal{S}^{\textrm{2-spine}} < m_n }\right]. \label{eq:condition-2spine}
\end{align}
Note that for all $n\ge2$, by \eqref{upbdPnx},
\[
\sum_{j=1}^n \E[R_1] P_{n-j+1}(z_n+B-S_{w_{j-1}^{I\!I}}) \lesssim \sum_{j=1}^n \sup_x P_j(x) \lesssim \log n.
\]
Consequently, we obtain that uniformly in $\|z_n\|\le K\sqrt{n}$,
\begin{align*}
 &\E_{\P^{\textrm{2-spine}}} \left[  \frac{\mathsf{R}^{I\!I}_n(z_n+B)}{\alpha|B| \log n} \ind{ \mathcal{S}^{\textrm{2-spine}} < m_n }\right] 
 \lesssim  \P^{\textrm{2-spine}}( \mathcal{S}^{\textrm{2-spine}} < m_n) =o_n(1).
\end{align*}
Putting the previous estimates together, we see that for the early separating scenario, uniformly in $\|z_n\|\le K\sqrt{n}$,
\begin{equation}\label{earlyS}
\E_{\P^{\textrm{2-spine}}} \left[ \left| \frac{\mathsf{R}^{I\!I}_n(z_n+B)}{\alpha|B| \log n} -   \frac{\mathsf{R}^{*}_n(z_n+B) }{\alpha|B| \log n} \right |\ind{ \mathcal{S}^{\textrm{2-spine}} < m_n }\right]  =o_n(1).
\end{equation}

It remains to bound $\E_{\P^{\textrm{2-spine}}} \left[ \left| \frac{\mathsf{R}^{I\!I}_n(z_n+B) }{\alpha|B| \log n} -   \frac{\mathsf{R}^{*}_n(z_n+B) }{\alpha|B| \log n} \right | \ind{ \mathcal{S}^{\textrm{2-spine}} \ge m_n }\right]$. 
According to the construction of $\mathsf{R}^{I\!I}_{n,j}(z_n+B)$ and $\mathsf{R}^*_{n,j}(z_n+B)$ under $\P^{\textrm{2-spine}}$, we have
\begin{align}\label{lateseparating}
& \left| \frac{\mathsf{R}^{I\!I}_n(z_n+B) }{\alpha|B| \log n} -   \frac{\mathsf{R}^{*}_n(z_n+B) }{\alpha|B| \log n} \right| \ind{ \mathcal{S}^{\textrm{2-spine}} \ge m_n } \\
\le \,& \frac{1}{\alpha|B|\log n}  \sum_{j=1}^n \left| \mathsf{R}^{I\!I}_{n,j}(z_n+B) - \mathsf{R}^*_{n,j}(z_n+B)  \right| \ind{ \mathcal{S}^{\textrm{2-spine}} \ge m_n }\nonumber\\
\le \,&  \frac{1}{\alpha|B|\log n} \sum_{j=1}^{ m_n }  \left| \mathsf{R}^I_{n,j}(z_n+B)- \mathsf{R}'_{n,j}(z_n+B)\right| \ind{ \mathsf{L}^I_{n,j}(z_n+B) \ge 1 }  \nonumber\\
&\quad +   \frac{1}{\alpha|B|\log n} \sum_{j=1+ m_n }^{n} \left| \mathsf{R}^{I\!I}_{n,j}(z_n+B) - \mathsf{R}^*_{n,j}(z_n+B) \right|=: \Sigma_\eqref{lateseparating}(1) + \Sigma_\eqref{lateseparating}(2).\nonumber
\end{align}
For the first term $\Sigma_\eqref{lateseparating}(1)$ on the right hand side of \eqref{lateseparating}, 
\begin{align}\label{bdlateS1}
&\E_{\P^{\textrm{2-spine}}} \left[ \Sigma_\eqref{lateseparating}(1)\right] 
\lesssim \frac{1}{\log n} \sum_{j=1}^{m_n} \E_{\P^{\textrm{2-spine}}} \left[ \mathsf{R}^I_{n,j}(z_n+B)\ind{ \mathsf{L}^I_{n,j}(z_n+B) \ge 1 } \right] \\
&\hspace{5cm} + \frac{1}{\log n} \sum_{j=1}^{m_n} \E_{\P^{\textrm{2-spine}}} \left[ \mathsf{R}'_{n,j}(z_n+B)\ind{ \mathsf{L}^I_{n,j}(z_n+B) \ge 1 } \right] . \nonumber
\end{align}
For convenience, we define the conditional probability measure
\[
\P^{\textrm{2-spine}}_{\vert_S}(\cdot) := \P^{\textrm{2-spine}}\big(\cdot \vert (S_{w_k^I}, S_{w_k^{I\!I}})_{k\le n}  \big).
\]
Then, for $1\le j\le m_n$, similarly as in \eqref{eq:condition-2spine},
\begin{align*}
&\E_{\P^{\textrm{2-spine}}} \left[ \mathsf{R}^I_{n,j}(z_n+B)\ind{ \mathsf{L}^I_{n,j}(z_n+B) \ge 1 } \right] \\
=\,& \E_{\P^{\textrm{2-spine}}} \left[  \E_{\P^{\textrm{2-spine}}} \left[ \mathsf{R}^I_{n,j}(z_n+B)\ind{ \mathsf{L}^I_{n,j}(z_n+B) \ge 1 } \Big\vert (S_{w_k^I}, S_{w_k^{I\!I}})_{k\le n} \right]\right] \\
=\, &  \E_{\P^{\textrm{2-spine}}} \left[  P_{n-j+1}(z_n+B-S_{w_{j-1}^I}) \overline{H}_j(z_n+B-S_{w_{j-1}^I}) \right],
\end{align*}
where $\overline{H}_j(z):=\E_{\P^{\textrm{2-spine}}_{\vert_S}}\left[ |\overline{\cR_j}| \left(1- \left(1- u_{n-j}\ast\nu(z) \right)^{|\overline{\cL_j}|}\right)\right]$. In the last expectation above, $(|\overline{\cL_j}|,|\overline{\cR_j}|)$ stands for $(|\cL_j|, |\cR_j|)$ (if $j\leq \mathcal{S}^{\textrm{2-spine}}$) or for $(|\cL^I_j|, |\cR^I_j|)$ (if $j> \mathcal{S}^{\textrm{2-spine}}$). In both cases, $(|\overline{\cL_j}|,|\overline{\cR_j}|)$ is distributed as $\mathscr{L}((L_1,R_1),\P)$, and is independent of $(S_{w_k^I}, S_{w_k^{I\!I}})_{k\le n}$.

Recall from \eqref{eq:mn} that $n-m_n=\lfloor n^\delta\rfloor$.
By \eqref{upbdu}, we see that uniformly for $1\le j\le m_n$ and for $z_n\in\Z^2$, 
\[
    u_{n-j}\ast\nu(z_n+B-S_{w_{j-1}^I}) = o_n(1).
\] 
Moreover,
\[
\E_{\P^{\textrm{2-spine}}_{\vert_S}}\left[ |\overline{\cR_j}| (1-(1-h)^{|\overline{\cL_j}|})\right] \le \E_{\P^{\textrm{2-spine}}_{\vert_S}}\big[|\overline{\cR_j}| (1\wedge h|\overline{\cL_j}|)\big]=\E[R_1(1\wedge hL_1)],
\]
which goes to zero when $h\to 0$ by dominated convergence theorem.
It follows that 
\[
    \E_{\P^{\textrm{2-spine}}_{\vert_S}}\left[ |\overline{\cR_j}| \left(1- \left(1- u_{n-j}\ast\nu(z_n+B-S_{w_{j-1}^I}) \right)^{|\overline{\cL_j}|}\right)\right]=o_n(1)
\]
uniformly for $1\le j\le m_n$ and $z_n\in\Z^2$. 
From here up to \eqref{bdR'nj}, all quantities related to $j$ that are shown to be $o_n(1)$ are supposed to be uniform for $1\le j\le m_n$. 
By \eqref{upbdPnx}, we deduce that uniformly in $z_n\in\Z^2$,
\begin{align}\label{bdRnj}
&\E_{\P^{\textrm{2-spine}}} \left[ \mathsf{R}^I_{n,j}(z_n+B)\ind{ \mathsf{L}^I_{n,j}(z_n+B) \ge 1 } \right] = \frac{o_n(1)}{ n-j +1}.
\end{align}
On the other hand, by the construction of $\mathsf{R}'_{n,j}(z_n+B)$, 
\begin{align*}
&\E_{\P^{\textrm{2-spine}}} \left[ \mathsf{R}'_{n,j}(z_n+B)\ind{ \mathsf{L}^I_{n,j}(z_n+B) \ge 1 } \right] \\
= &\E_{\P^{\textrm{2-spine}}} \Biggl[ \E_{ \P^{\textrm{2-spine}}_{\vert_S} } \left[ \mathsf{R}^I_{n,j}(z_n+B) \vert \mathsf{L}^I_{n,j}(z_n+B)=0 \right] \times\P^{\textrm{2-spine}}_{\vert_S} \left(  \mathsf{L}^I_{n,j}(z_n+B) \ge 1 \right) \Biggr]\\
= &\E_{\P^{\textrm{2-spine}}} \left[ \frac{ P_{n-j+1}(z_n+B-S_{w^I_{j-1}})\E[ R_1(1-u_{n-j}\ast\nu(z_n+B-S_{w^I_{j-1}}))^{L_1}]}{ \E[ (1-u_{n-j}\ast\nu(z_n+B-S_{w^I_{j-1}}))^{L_1}] } \right.\\
&\hspace{5cm} \times\E\left[ 1- (1-u_{n-j}\ast\nu(z_n+B-S_{w^I_{j-1}}))^{L_1}\right]\Bigg].
\end{align*}
Similarly as above, $\E\left[ 1- (1-u_{n-j}\ast\nu(z_n+B-y))^{L_1}\right] = o_n(1)$ uniformly in $z_n$ and $y\in\Z^2$. Moreover, by \eqref{upbdPnx}, for all $1\le j\le n$, uniformly in $z_n$ and $y\in\Z^2$,
\begin{align}\label{bdR*}
&\frac{ P_{n-j+1}(z_n+B-y)\E\left[ R_1(1-u_{n-j}\ast\nu(z_n+B-y))^{L_1}\right]}{ \E\left[ (1-u_{n-j}\ast\nu(z_n+B-y))^{L_1}\right] }\\
\lesssim \, & \frac{1}{n-j+1} \frac{\E[R_1]}{\P(L_1=0)} \lesssim \frac{1}{n-j+1}.\nonumber
\end{align}
As a consequence, we get that uniformly in $z_n\in\Z^2$,
\begin{align}\label{bdR'nj}
\E_{\P^{\textrm{2-spine}}} \left[ \mathsf{R}'_{n,j}(z_n+B)\ind{ \mathsf{L}^I_{n,j}(z_n+B) \ge 1 } \right] = \frac{o_n(1)}{n-j+1}.
\end{align}
Plugging \eqref{bdRnj} and \eqref{bdR'nj} into \eqref{bdlateS1} yields that uniformly in $z_n\in\Z^2$,
\begin{align}\label{bdlateS1+}
\E_{\P^{\textrm{2-spine}}} \left[ \Sigma_\eqref{lateseparating}(1)\right] \lesssim & \sum_{j=1}^{m_n} \frac{1}{\log n}\frac{o_n(1)}{n-j+1} +  \sum_{j=1}^{m_n} \frac{1}{\log n}\frac{o_n(1)}{n-j+1} = o_n(1).
\end{align}

It remains to deal with $\Sigma_\eqref{lateseparating}(2)$. Note that
\begin{align}\label{bdlateS2}
&\E_{\P^{\textrm{2-spine}}} \left[ \Sigma_\eqref{lateseparating}(2)\right] = \E_{\P^{\textrm{2-spine}}} \left[   \frac{1}{\alpha|B|\log n} \sum_{j=m_n +1}^{n} \left| \mathsf{R}^{I\!I}_{n,j}(z_n+B) - \mathsf{R}^*_{n,j}(z_n+B) \right|\right] \\
\lesssim\, & \frac{1}{\log n} \sum_{j=m_n+1}^n  \E_{\P^{\textrm{2-spine}}} \left[    \mathsf{R}^{I\!I}_{n,j}(z_n+B) \right] +  \frac{1}{\log n} \sum_{j=m_n+1}^n  \E_{\P^{\textrm{2-spine}}} \left[    \mathsf{R}^*_{n,j}(z_n+B) \right].\nonumber
\end{align}
Conditioning on the two spines and using \eqref{upbdPnx}, we see that
\begin{align*}
 &\E_{\P^{\textrm{2-spine}}} \left[    \mathsf{R}^{I\!I}_{n,j}(z_n+B) \right] =  \E_{\P^{\textrm{2-spine}}} \left[   \E_{\P^{\textrm{2-spine}}} \left[    \mathsf{R}^{I\!I}_{n,j}(z_n+B) \big\vert  (S_{w_k^I}, S_{w_k^{I\!I}})_{k\le n} \right]\right] \\
 =\,& \E[R_1] \E_{\P^{\textrm{2-spine}}} \left[   P_{n-j+1}(z_n+B-S_{w_{j-1}^{I\!I}})\right] \lesssim \frac{1}{n-j+1}.
\end{align*}
In addition, by \eqref{bdR*} and \eqref{upbdPnx}, we have
\begin{align*}
&\E_{\P^{\textrm{2-spine}}} \left[    \mathsf{R}^*_{n,j}(z_n+B) \right] = \E_{\P^{\textrm{2-spine}}} \left[   \E_{\P^{\textrm{2-spine}}_{\vert_S}}  \left[\mathsf{R}^I_{n,j}(z_n+B) \vert \mathsf{L}^I_{n,j}(z_n+B)=0 \right]\right]\\
=\,& \E_{\P^{\textrm{2-spine}}} \left[   \frac{ P_{n-j+1}(z_n+B-y= S_{w^I_{j-1}})\E[ R_1(1-u_{n-j}\ast\nu(z_n+B-y= S_{w^I_{j-1}}))^{L_1}]}{ \E[ (1-u_{n-j}\ast\nu(z_n+B-y= S_{w^I_{j-1}}))^{L_1}] } \right] \lesssim \frac{1}{n-j+1}.
\end{align*}
Using this in \eqref{bdlateS2} leads to 
\begin{equation}\label{bdlateS2+}
\E_{\P^{\textrm{2-spine}}} \left[ \Sigma_\eqref{lateseparating}(2)\right]  \lesssim \frac{1}{\log n} \sum_{j=1+m_n}^n \frac{1}{n-j+1} \lesssim \delta.
\end{equation}

Finally, taking into account \eqref{lateseparating}, \eqref{bdlateS1+} and \eqref{bdlateS2+}, we obtain that uniformly in $z_n\in\Z^2$,
\begin{align*}
\E_{\P^{\textrm{2-spine}}} \left[ \left| \frac{\mathsf{R}^{I\!I}_n(z_n+B) }{\alpha|B| \log n} -   \frac{\mathsf{R}^{*}_n(z_n+B) }{\alpha|B| \log n} \right | \ind{ \mathcal{S}^{\textrm{2-spine}} \ge m_n }\right] \lesssim o_n(1) + \delta.
\end{align*}
Combining this with \eqref{earlyS} and \eqref{Steinbd} implies that uniformly in $\|z_n\|\le K\sqrt{n}$,
\begin{align*}
d_W(  \mathscr{L}( Y_n^+, \P), \mathrm{Exp}(1)) \lesssim o_n(1) + \delta.
\end{align*}
Letting $n\to\infty$ and then $\delta\downarrow 0$, we conclude that uniformly in $\|z_n\|\le K\sqrt{n}$,
\[
\lim_{n\to\infty} d_W(  \mathscr{L}( Y_n^+, \P), \mathrm{Exp}(1)) = 0,
\]
which is exactly \eqref{Wdistance}.

\subsection{Coupling of two spines: proof of \texorpdfstring{\eqref{bdDTVspine}}{}}\label{boundDTV}

To prove \eqref{bdDTVspine}, recall that 
\[
\P\big( (S^{I\!I}_{w_k})_{k\le m_n} =  (x_k)_{k\le m_n}\big)= \Q^* \big( (S_{w_k})_{k\le m_n}= (x_k)_{k\le m_n}\vert S_{w_n}\in z_n +B\big),
\]
and 
\begin{align*}
&\P^{\textrm{spine-I}}\big((S^I_{w_k})_{k\le m_n}= (x_k)_{k\le m_n}\big) \\
= \,&\Q^*\big( (S_{w_k})_{k\le m_n} =(x_k)_{k\le m_n}) \vert S_{w_n}\in z_n +B, \mathsf{L}_n^{\prec w_n}(z_n+B)=0 \big)\\
=\,&\Q^*\big( (S_{w_k})_{k\le m_n} =(x_k)_{k\le m_n}) \vert S_{w_n}\in z_n +B, \cap_{j=1}^n\{\mathsf{L}_{n,j}(z_n+B)=0\} \big).
\end{align*}
Recall from \eqref{eq:mn} that $1\le m_n = n - \floor{n^\delta}$. Let us introduce a new random vector $(\Xi^{m_n}_k)_{k\le m_n}$ under $\P$, which is distributed as
\[
    \Q^*\left( (S_{w_k})_{k\le m_n}\in \cdot \vert S_{w_n}\in z_n +B, \cap_{j=1}^{m_n}\{\mathsf{L}_{n,j}(z_n+B)=0\}\right).
\] 
It immediately yields that
\begin{align*}
&D_{TV}(m_n) \\
\le\,\, & \|\mathscr{L}( (S^I_{w_k})_{k\le m_n},\P^{\textrm{spine-I}}) - \mathscr{L}( (\Xi^{m_n}_{k})_{k\le m_n},\P)  \|_{TV} + \| \mathscr{L}( (\Xi^{m_n}_{k})_{k\le m_n},\P)-   \mathscr{L}( (S^{I\!I}_{w_k})_{k\le m_n},\P)\|_{TV}\\
=:\,& D^I_{TV}(m_n) + D^{I\!I}_{TV}(m_n).
\end{align*}
We will study $D^I_{TV}(m_n)$ and $D^{I\!I}_{TV}(m_n)$ separately. 

\paragraph{Convergence of $D^{I\!I}_{TV}(m_n)$.} 
For convenience, we define $\Q^*_{0,z_n+B}(\cdot) := \Q^*(\cdot\vert S_{w_n}\in z_n +B)$. By definition,
\begin{align*}
&\P((\Xi^{m_n}_k)_{k\le m_n}=(x_k)_{k\le m_n})\\
=\,& \frac{\Q^*\left( (S_{w_k})_{k\le m_n} =(x_k)_{k\le m_n}, S_{w_n}\in z_n +B, \cap_{j=1}^{m_n}\{\mathsf{L}_{n,j}(z_n+B)=0\}\right)}{\Q^*( S_{w_n}\in z_n +B, \cap_{j=1}^{m_n}\{\mathsf{L}_{n,j}(z_n+B)=0\} )} \\
=\,& \frac{\Q^*_{0,z_n+B}\left( \cap_{j=1}^{m_n}\{\mathsf{L}_{n,j}(z_n+B)=0\} \vert (S_{w_k})_{k\le m_n} =(x_k)_{k\le m_n}  \right)}{ \Q^*_{0,z_n+B}( \cap_{j=1}^{m_n}\{\mathsf{L}_{n,j}(z_n+B)=0\} )}\times \Q^*_{0,z_n+B}( (S_{w_k})_{k\le m_n} = (x_k)_{k\le m_n} ),
\end{align*}
from which we deduce that
\begin{align*}
& 2 D^{I\!I}_{TV}(m_n)=   \sum_{x_1,\cdots, x_{m_n}\in\Z^2} \left|   \P((\Xi^{m_n}_k)_{k\le m_n}=(x_k)_{k\le m_n})-  \Q^*_{0,z_n+B}( (S_{w_k})_{k\le m_n} = (x_k)_{k\le m_n} ) \right| \\
 = \, &\sum_{x_1,\cdots, x_{m_n}\in\Z^2} \left| \frac{\Q^*_{0,z_n+B}( \cap_{j=1}^{m_n}\{\mathsf{L}_{n,j}(z_n+B)=0\} \vert (S_{w_k})_{k\le m_n} =(x_k)_{k\le m_n} )}{ \Q^*_{0,z_n+B}( \cap_{j=1}^{m_n}\{\mathsf{L}_{n,j}(z_n+B)=0\} )}-1\right| \\
 &\hspace{10cm}  \times \Q^*_{0,z_n+B}( (S_{w_k})_{k\le m_n}= (x_k)_{k\le m_n})  \\
=\, & \E_{\Q^*_{0,z_n+B}}\left[ \left| \frac{ \E_{\Q^*_{0,z_n+B}}\big[ \ind{\cap_{j=1}^{m_n}\{\mathsf{L}_{n,j}(z_n+B)=0\}} \big\vert (S_{w_k})_{k\le m_n}  \big] }{ \E_{\Q^*_{0,z_n+B}}\big[ \ind{\cap_{j=1}^{m_n}\{\mathsf{L}_{n,j}(z_n+B)=0\}} \big]}-1\right| \right].
\end{align*}
We claim that for $m_n=n-\lfloor n^\delta\rfloor$ with fixed $\delta\in(0,1/2)$, uniformly in $\|z_n\|\le K\sqrt{n}$, under $\Q^*_{0,z_n+B}$, 
\begin{equation}\label{goodL}
\E_{\Q^*_{0,z_n+B}}\left[ \ind{\cap_{j=1}^{m_n}\{\mathsf{L}_{n,j}(z_n+B)=0\}} \Big\vert (S_{w_k})_{k\le m_n}  \right] \xrightarrow[n\to\infty]{\textrm{ in probability }} \delta.
\end{equation}
The proof of \eqref{goodL} is postponed to the end of this section. Let us move forward by admitting it here. 
It follows by bounded convergence that as $n\to \infty$, uniformly in $\|z_n\|\le K\sqrt{n}$, 
\[
 \E_{\Q^*_{0,z_n+B}}\left[ \ind{\cap_{j=1}^{m_n}\{\mathsf{L}_{n,j}(z_n+B)=0\}} \right] = \E_{\Q^*_{0,z_n+B}}\left[ \E_{\Q^*_{0,z_n+B}}\left[ \ind{\cap_{j=1}^{m_n}\{\mathsf{L}_{n,j}(z_n+B)=0\}} \Big\vert (S_{w_k})_{k\le m_n}  \right]  \right] \to \delta.
\]
Therefore, using again bounded convergence, we obtain that uniformly in $\|z_n\|\le K\sqrt{n}$,
\[
D^{I\!I}_{TV}(m_n) \xrightarrow[n\to\infty]{} 0.
\]

\paragraph{Convergence of $D^I_{TV}(m_n)$.} 

According to Lemma 17 in \cite{LZ11}, if there is an event $F$ in some probability space $(\Omega, \mathcal{F}, P)$ such that $P(F^c)=P(\Omega\setminus F)\le \varepsilon<1$, then
\[
\| P(\cdot) - P(\cdot\vert F) \|_{TV} \le 2\varepsilon.
\] 
Instead of comparing directly $\mathscr{L}( (S^I_{w_k})_{k\le m_n},\P^{\textrm{spine-I}})$ and $ \mathscr{L}( (\Xi^{m_n}_{k})_{k\le m_n},\P) $, we fix some large $M>0$ and introduce the events
\[
    F_n(\Xi):=\big\{ \|  \Xi_{m_n}^{m_n} -z_n \| < M n^{\delta/2} \big\} \textrm{ and } F_n(S^I):=\big\{ \|S^I_{w_{m_n}}-z_n \| < M n^{\delta/2}\big\}.
\] 
After verifying that both $\P^{\textrm{spine-I}}( F_n(S^I)^c )$ and $\P(F_n(\Xi)^c  )$ are small, we will turn to show that 
\[
    \lim_{n\to \infty} \big\|  \P( (\Xi^{m_n}_{k})_{k\le m_n}\in \cdot \vert F_n(\Xi)) -  \P^{\textrm{spine-I}}((S^I_{w_k})_{k\le m_n}\in\cdot \vert F_n(S^I) )\big\|_{TV}=0.
\] 
Therefore, uniformly in $\|z_n\|\le K\sqrt{n}$,
\begin{equation}
    \label{eq:D^I_TV}
    D^{I}_{TV}(m_n) \xrightarrow[n\to\infty]{} 0, 
\end{equation}
and $D_{TV}(m_n) \to 0$ readily follows.

\medskip
\textbf{Step 1.} Let us verify that uniformly in $\|z_n\|\le K\sqrt{n}$, $\P(F_n(\Xi)^c  ) = o_n(1) + o_{M}(1)$.

By \eqref{Pnx}, for all $y\in \Z^2$ satisfying $M n^{\delta/2}\le \|z_n-y\|\le n^\delta$ and $\|z_n\| \le K\sqrt{n}$, the ratio $P_{m_n}(y)/P_n(z_n)$ is bounded away from infinity in a uniform way, not only for $n$, but also for $M>0$. This is also true for the ratio $P_{m_n}(y)/P_n(z_n+B)$.
For sufficiently large $n$, we also have
\[
\sum_{y: \|z_n-y\|\ge Mn^{\delta/2}}P_{\lfloor n^\delta\rfloor}(z_n-y+B) \le |B|\cdot\mathbb{P}\Big(\max_{j\le n^\delta}\|S_j\| \ge \frac{M}{2} n^{\delta/2} \Big).
\]
By Proposition 2.1.2(b) of \cite{Lawler-Limic}, there exists some constant $c>0$ such that for all $n\geq 1$ and all $M>0$,
\[
    \mathbb{P}\Big(\max_{j\le n^\delta}\|S_j\| \geq \frac{M}{2} n^{\delta/2} \Big)\lesssim e^{- c M^2}.
\]
By Markov property, we deduce from the preceding discussions that
\begin{align*}
    \Q^*( \|S_{w_{m_n}}- z_n \| \ge M n^{\delta/2} \vert S_{w_n} \in z_n +B)= & \sum_{y: \|z_n-y\|\ge M n^{\delta/2}} P_{m_n}(y)\frac{P_{\lfloor n^\delta\rfloor }(z_n-y+B) }{P_n(z_n+B)}\\
    =& \sum_{y:\|z_n-y\|\ge M n^{\delta/2}} P_{\lfloor n^\delta\rfloor}(z_n-y+B) \frac{P_{m_n}(y)}{P_n(z_n+B)}\\
    \lesssim&\,\, e^{-cM^2}= o_{M}(1).
\end{align*}
Together with the fact $D^{I\!I}_{TV}(m_n)=o_n(1)$, this entails that uniformly in $\|z_n\|\le K\sqrt{n}$,
\begin{align}\label{restrictXi}
 \P\left( \|\Xi^{m_n}_{m_n}-z_n \| \ge M n^{\delta/2}  \right) = o_n(1) +o_{M}(1).
\end{align}

\textbf{Step 2.} Let us verify that uniformly in $\|z_n\|\le K\sqrt{n}$, $\P^{\textrm{spine-I}}( F_n(S^I)^c )=o_n(1) + o_{M}(1)$. 

In fact, as 
\[
u_n(z_n+B) = \Q^*( S_{w_n}\in z_n +B, \cap_{j=1}^{n}\{\mathsf{L}_{n,j}(z_n+B)=0\} ),
\]
we have by definition of $F_n(S^I)$ that 
\begin{align*}
&\P^{\textrm{spine-I}}( F_n(S^I)^c ) =\sum_{y: \|z_n-y\|\ge M n^{\delta/2} }\frac{ \Q^*\left( S_{w_{ m_n}} =y, S_{w_n}\in z_n +B, \cap_{j=1}^{n}\{\mathsf{L}_{n,j}(z_n+B)=0\}\right) }{u_n(z_n+B)}\\
\le \,& \sum_{y: \|z_n-y\|\ge M n^{\delta/2}} \frac{ \Q^*\left( S_{w_{ m_n}} =y, S_{w_n}\in z_n +B, \cap_{j=m_n+1}^{n}\{\mathsf{L}_{n,j}(z_n+B)=0\}\right) }{u_n(z_n+B)}.
\end{align*}
Let us write $\|B\|:=\sup_{b\in B}\|b\|<\infty$.
By Markov property at time $m_n$, it follows that
\begin{align}
 \P^{\textrm{spine-I}}( F_n(S^I)^c )  \le & \sum_{y: \|z_n-y\|\ge M n^{\delta/2}} \Q^*\left( S_{w_{ m_n}} =y \right) \frac{ \Q^*_y( S_{w_{n^\delta}} \in z_n+B, \cup_{j=1}^{\lfloor n^\delta\rfloor}\{ \mathsf{L}_{\lfloor n^\delta\rfloor,j}(z_n+B) =0\} )}{ u_n(z_n +B)} \nonumber \\
=&  \sum_{y\colon n^\delta+\|B\| \ge \|z_n-y\|\ge M n^{\delta/2}} \frac{ P_{m_n}(y)  }{ u_n(z_n+B) } \cdot u_{\lfloor n^\delta\rfloor}(z_n+B-y). \label{eq:Fn-c}
\end{align}
In the last line, we have used the fact that each jump of random walk is at most of distance 1. As a result, $u_n(x)=0$ if $\|x\|>n$.

Again, for $\|z_n\|\le K\sqrt{n}$ and $n^\delta +\|B\| \ge \|z_n-y\|\ge M n^{\delta/2}$, recall that the ratio $P_{m_n}(y)/P_n(z_n)$ is bounded away from infinity in a uniform way, not only for $n$, but also for $M>0$. 
Meanwhile, \eqref{Pnx} and \eqref{unxB} imply that uniformly in $\|z_n\|\le K\sqrt{n}$, $u_n(z_n+B) = \frac{P_n(z_n)(1+o_n(1))}{\alpha\log n}$. So, we have
\begin{equation}
    \label{eq:sup-ratio}
    \sup_{\substack{\|z_n\|\le K\sqrt{n}\\ n^\delta +\|B\| \ge \|z_n-y\|\ge M n^{\delta/2}}} \frac{ P_{m_n}(y)  }{ u_n(z_n+B) } \lesssim \log n.
\end{equation}
Moreover, by \eqref{upbduv2}, for any fixed $a\in(0,1)$, there exists a constant $C_a$ such that for any $z\in \Z^2$,
\[
    u_{\lfloor n^\delta\rfloor}(z) \le  \frac{C_a}{\delta \log n}  \Big(P_{\lfloor n^\delta\rfloor}(z) + \frac{1}{\delta\log n} \sup_{\|x\|\le n^{a\delta}}P_{\lfloor n^\delta \rfloor -\lfloor n^{a\delta}\rfloor}(z-x)\Big).
\]
Using $u_n(z+B) \le \sum_{b\in B}u_n(z+b)$, we deduce that 
\begin{align}
 &\sum_{y\colon n^\delta+\|B\| \ge \|z_n-y\|\ge M n^{\delta/2}} u_{\lfloor n^\delta\rfloor}(z_n+B-y) \leq  \sum_{z \colon n^\delta+\|B\| \ge \|z\|\ge M n^{\delta/2}} \sum_{b\in B} u_{\lfloor n^\delta\rfloor}(z+b)  \label{eq:sum-undelta}  \\
 \lesssim\, &\frac{1}{\delta \log n}\bigg(\sum_{n^\delta+\|B\| \ge \|z\| \ge Mn^{\delta/2}} \sum_{b\in B} \Big(P_{\lfloor n^\delta\rfloor }(z+b) +\frac{1}{\delta\log n} \sup_{\|x\|\le n^{a\delta}}P_{\lfloor n^\delta \rfloor -\lfloor n^{a\delta}\rfloor}(z+b-x)\Big)\bigg) \nonumber.
\end{align}
Here we apply Theorem 2.1.1, formula (2.4) in \cite{Lawler-Limic}, which gives the following bound that uniformly in $x$ and $n$,
\[
|P_n(x) - p_n(x) | \lesssim \frac{1}{n^{2} \|x/\sqrt{n}\|^2}.
\]
Together with the monotonicity of $p_n(x)$ with respect to $\|x\|$, we see that for sufficiently large $n$,
\begin{align*}
  & \sum_{n^\delta+\|B\| \ge \|z\| \ge Mn^{\delta/2}} \sum_{b\in B} P_{\lfloor n^\delta\rfloor }(z+b)\\
  \lesssim & \sum_{n^\delta+\|B\| \ge \|z\| \ge Mn^{\delta/2}} \sum_{b\in B} \Big(p_{\lfloor n^\delta\rfloor }(z+b) + \frac{1}{n^{2\delta} M^2}\Big) \\
  \lesssim & \sum_{n^\delta+\|B\| \ge \|z\| \ge Mn^{\delta/2}} \sum_{b\in B} \Big(p_{\lfloor n^\delta\rfloor }(z/2) + \frac{1}{n^{2\delta} M^2}\Big) 
  = o_M(1) + o_n(1).
\end{align*}
We can argue in a similar way to get
\[
    \frac{1}{\delta\log n}  \sum_{n^\delta+\|B\| \ge \|z\| \ge Mn^{\delta/2}} \sum_{b\in B} \sup_{\|x\|\le n^{a\delta}}P_{\lfloor n^\delta \rfloor -\lfloor n^{a\delta}\rfloor}(z+b-x) = o_n(1).
\]
Putting these estimates into \eqref{eq:sum-undelta}, then combining it with \eqref{eq:sup-ratio} and \eqref{eq:Fn-c}, we conclude that uniformly in $\|z_n\|\le K\sqrt{n}$,
\[
 \P^{\textrm{spine-I}}( F_n(S^I)^c ) \lesssim \frac{\log n}{\delta \log n} (o_M(1) + o_n(1)) = o_n(1) + o_{M}(1).
\]

\textbf{Step 3.} It remains to bound $ \|\P( (\Xi^{m_n}_{k})_{k\le m_n}\in \cdot \vert F_n(\Xi)  )- \P^{\textrm{spine-I}}((S^I_{w_k})_{k\le m_n}\in\cdot \vert F_n(S^I) )\|_{TV}$. Note that
\begin{align*}
&2\|  \P( (\Xi^{m_n}_{k})_{k\le m_n}\in \cdot \vert F_n(\Xi)) -  \P^{\textrm{spine-I}}((S^I_{w_k})_{k\le m_n}\in\cdot \vert F_n(S^I) )\|_{TV}\\
=& \sum_{\substack{\|x_{m_n}-z_n\| \le M n^{\delta/2},\\ x_1,\cdots, x_{m_n-1}\in \Z^2}} \left|  \P( (\Xi^{m_n}_{k})_{k\le m_n} = (x_k)_{k\le m_n} \vert F_n(\Xi) ) -  \P^{\textrm{spine-I}}((S^I_{w_k})_{k\le m_n} = (x_k)_{k\le m_n} \vert F_n(S^I)) \right|\\
= &  \sum_{\substack{\|x_{m_n}-z_n\| \le M n^{\delta/2},\\ x_1,\cdots, x_{m_n-1}\in \Z^2}}\P( (\Xi^{m_n}_{k})_{k\le m_n}=(x_k)_{k\le m_n}\in \cdot \vert F_n(\Xi))\times\left| 1 - \frac{\P^{\textrm{spine-I}}((S^I_{w_k})_{k\le m_n} = (x_k)_{k\le m_n} \vert F_n(S^I))}{ \P( (\Xi^{m_n}_{k})_{k\le m_n} = (x_k)_{k\le m_n} \vert F_n(\Xi) ) } \right|.
\end{align*}
Observe that by definition,
\begin{align*}
&\frac{\P^{\textrm{spine-I}}((S^I_{w_k})_{k\le m_n} = (x_k)_{k\le m_n} \vert F_n(S^I))}{ \P( (\Xi^{m_n}_{k})_{k\le m_n} = (x_k)_{k\le m_n} \vert F_n(\Xi) ) } \\
=\,& \frac{\Q^*\left( (S_{w_k})_{k\le m_n} =(x_k)_{k\le m_n}, S_{w_n}\in z_n +B, \cap_{j=1}^{m_n}\{\mathsf{L}_{n,j}(z_n+B)=0\}\right)}{\Q^*\left( (S_{w_k})_{k\le m_n} =(x_k)_{k\le m_n}, S_{w_n}\in z_n +B, \cap_{j=1}^{n}\{\mathsf{L}_{n,j}(z_n+B)=0\}\right)}\\
&\hspace{3cm}\times \frac{ \Q^*(\|S_{w_{m_n}}-z_n\|<Mn^{\delta/2}, S_{w_n}\in z_n +B, \cap_{j=1}^{n}\{\mathsf{L}_{n,j}(z_n+B)=0\} )}{ \Q^*( \|S_{w_{m_n}}-z_n\|<Mn^{\delta/2}, S_{w_n}\in z_n +B, \cap_{j=1}^{m_n}\{\mathsf{L}_{n,j}(z_n+B)=0\} )}.
\end{align*}
By Markov property at time $m_n$,
\begin{align*}
&\frac{ \Q^*\left( (S_{w_k})_{k\le m_n} =(x_k)_{k\le m_n}, S_{w_n}\in z_n +B, \cap_{j=1}^{n}\{\mathsf{L}_{n,j}(z_n+B)=0\}\right) }{ \Q^*\left( (S_{w_k})_{k\le m_n} =(x_k)_{k\le m_n}, S_{w_n}\in z_n +B, \cap_{j=1}^{m_n}\{\mathsf{L}_{n,j}(z_n+B)=0\}\right) }\\
=\,&\frac{ \Q^*\left( (S_{w_k})_{k\le m_n} =(x_k)_{k\le m_n}, \cap_{j=1}^{m_n}\{\mathsf{L}_{n,j}(z_n+B)=0\}\right)\Q^*_{x_{m_n}}(\mathsf{L}_{\floor{n^\delta}}^{\prec w_{\lfloor n^\delta \rfloor}}(z_n+B)=0, S_{w_{\lfloor n^\delta \rfloor}}\in z_n+B) }{ \Q^*\left( (S_{w_k})_{k\le m_n} =(x_k)_{k\le m_n}, \cap_{j=1}^{m_n}\{\mathsf{L}_{n,j}(z_n+B)=0\}\right) P_{\lfloor n^\delta \rfloor}(z_n+B-x_{m_n}) }\\
=\,& \frac{ u_{\lfloor n^\delta \rfloor}(z_n+B-x_{m_n}) }{ P_{\lfloor n^\delta \rfloor}(z_n+B-x_{m_n})  }, \textrm{ and }\\
& \frac{ \Q^*( \|S_{w_{m_n}}-z_n\|< Mn^{\delta/2}, S_{w_n}\in z_n +B, \cap_{j=1}^{m_n}\{\mathsf{L}_{n,j}(z_n+B)=0\} ) }{ \Q^*( \|S_{w_{m_n}}-z_n\|< Mn^{\delta/2}, S_{w_n}\in z_n +B, \cap_{j=1}^{n}\{\mathsf{L}_{n,j}(z_n+B)=0\} ) }\\
=\,& \frac{ \sum_{y:\|z_n-y\|< Mn^{\delta/2} } \Q^*( S_{w_{m_n}}=y, \cap_{j=1}^{m_n}\{\mathsf{L}_{n,j}(z_n+B)=0\} )P_{\lfloor n^\delta\rfloor}(z_n+B-y) }{ \sum_{y: \|z_n-y\|< M n^{\delta/2}} \Q^*( S_{w_{m_n}}=y, \cap_{j=1}^{m_n}\{\mathsf{L}_{n,j}(z_n+B)=0\} ) u_{\lfloor n^\delta\rfloor}(z_n+B-y) }.
\end{align*}
Uniformly for $\|x_{m_n}-z_n\|< Mn^{\delta/2}$ and $ \|z_n-y\|< Mn^{\delta/2}$, we obtain from \eqref{unxB} that
\[
 \frac{ u_{\lfloor n^\delta \rfloor}(z_n+B-x_{m_n}) }{ P_{\lfloor n^\delta \rfloor}(z_n+B-x_{m_n})  }= \frac{1+o_n(1)}{\alpha|B|\delta \log n} \quad \textrm{ and } \quad \frac{P_{\lfloor n^\delta \rfloor}(z_n+B-y) }{u_{\lfloor n^\delta \rfloor}(z_n+B-y) } = \alpha|B|\delta\log n (1+o_n(1)).
\]
This implies that uniformly in $x_{m_n}$ satisfying $\|z_n-x_{m_n}\|< Mn^{\delta/2}$,
\begin{align*}
&\frac{\P^{\textrm{spine-I}}((S^I_{w_k})_{k\le m_n} = (x_k)_{k\le m_n} \vert F_n(S^I))}{ \P( (\Xi^{m_n}_{k})_{k\le m_n} = (x_k)_{k\le m_n} \vert F_n(\Xi) ) } =1+o_n(1).
\end{align*}
As a result, for any fixed $M>0$,
\begin{align*}
&\|  \P( (\Xi^{m_n}_{k})_{k\le m_n}\in \cdot \vert F_n(\Xi)) -  \P^{\textrm{spine-I}}((S^I_{w_k})_{k\le m_n}\in\cdot \vert F_n(S^I) \|_{TV} = o_n(1).
\end{align*}
Combining the three steps above, letting $n\to \infty$ and then $M\to \infty$, we establish the convergence \eqref{eq:D^I_TV}.

Finally we prove \eqref{goodL} to finish this section.
\begin{proof}[Proof of \eqref{goodL}]
Calculating directly under the spinal decomposition, one sees that under $\Q^*_{0,z_n+B}$,
\begin{align}\label{goodL-1}
\Q^*_{0,z_n+B}\eqref{goodL-1}:=&\,\E_{\Q^*_{0,z_n+B}}\left[ \ind{\cap_{j=1}^{m_n}\{\mathsf{L}_{n,j}(z_n+B)=0\}} \Big\vert (S_{w_k})_{k\le m_n} \right] \\
=& \prod_{j=1}^{m_n} \E\left[(1-u_{n-j}\ast\nu(z_n+B-x_{j-1}))^{L_{1}} \right]\Big\vert _{x_{j-1}=S_{w_{j-1}}} \nonumber\\
=& \prod_{j=\lfloor n^\delta\rfloor}^{n-1}\E\left[ e^{L_1 \log(1-u_j\ast\nu(z_n+B -x_{n-j-1}))} \right] \Big\vert _{x_{n-j-1}=S_{w_{n-j-1}}}.\nonumber
\end{align}
By \eqref{upbdu}, uniformly for $\lfloor n^\delta \rfloor \le j < n$, it is clear that $\sup_x u_j\ast\nu(z_n+B -x) \lesssim \frac{1}{j\log j}=o_n(1)$. 
\begin{lemma}
For a sequence of positive numbers $(h_n)_{n\geq 1}$, if $h_n=o_n(1)$ as $n\to \infty$, then 
\[
    \E\left[ \exp(L_1\log(1-  h_n)) \right] = \exp\Big(\!-(1+o(1))h_n \frac{\sigma^2}{2} +o_n(1)h_n\Big).
\]
\end{lemma}
\begin{proof}
Recall that $\E[L_1]=\frac{\sigma^2}{2}$. Using the inequality 
\[
    |1-hx -e^{-hx} |\le x \int_0^h|1-e^{-ux}|\d u\le xh(1\wedge xh) \textrm{ for } x,h>0, 
\] 
we see by the dominated convergence theorem that
\begin{align*}
\Big\vert \E\left[ \exp(- h_n L_1) \right] -  \E\left[ 1-   h_n L_1 \right] \Big\vert \le &\,\E[ L_1 h_n (1\wedge L_1 h_n )] = h_n\times o_n(1), \\
 \Big\vert \E\left[ 1-  h_n L_1 \right] - \exp(- h_n\E[L_1]) \Big\vert \le &\, \E[ L_1 ] h_n (1\wedge \E[L_1]h_n) = h_n\times o_n(1).
\end{align*}
Therefore, $\E\left[ \exp(-  h_n L_1) \right] = \exp(- h_n \frac{\sigma^2}{2} ) + o_n(1) h_n = \exp(-h_n \frac{\sigma^2}{2} +o_n(1)h_n)$. 
To conclude, we use the fact that $\log(1-h_n)=-(1+o_n(1))h_n$. 
\end{proof}

We apply this lemma to $h_n=u_j\ast\nu(z_n+B -x)$.  
This yields that uniformly for $\lfloor n^\delta\rfloor\le j< n$, and uniformly in $\|z_n\|\le K\sqrt{n}$, as $n \to \infty$,
\begin{align*}
&\E\left[ \exp\left(  L_1 \log(1-u_j\ast\nu(z_n+B -x_{n-j-1})) \right) \right] \\
= &\exp\Big(\!- (1+o_n(1)) u_j\ast\nu(z_n+B -x_{n-j-1})  \frac{\sigma^2}{2}  + \frac{o_n(1)}{j\log j} \Big),
\end{align*}
which, according to \eqref{unxB}, equals to $\exp\left(-\frac{\sigma^2}{2} \frac{P_{j+1}(z_n-x_{n-j-1})}{\alpha\log (j+1)} +  \frac{o_n(1)}{j\log j}\right) $.

Going back to \eqref{goodL-1}, we get that under $\Q^*_{0,z_n+B}$,
\begin{align*}
\Q^*_{0,z_n+B}\eqref{goodL-1}&= \exp\bigg(\sum_{j=\lfloor n^\delta\rfloor }^{n-1}\Big(- \frac{\sigma^2}{2} \frac{P_{j+1}(z_n-S_{w_{n-j-1}})}{\alpha\log (j+1)} +\frac{o_n(1)}{j\log j} \Big)  \bigg)\\
&= \exp\bigg(\!- \frac{\sigma^2}{2} \Big(\sum_{j=\lfloor n^\delta\rfloor }^{n-1}\frac{P_{j+1}(z_n-S_{w_{n-j-1}})}{\alpha\log (j+1)} \Big)+o_n(1)\log\frac{1}{\delta} \bigg).
\end{align*}
It suffices to show that for arbitrarily small $\varepsilon>0$, uniformly in $\|z_n\|\le K\sqrt{n}$,
\begin{equation}\label{goodL-2}
\Q^*_{0,z_n+B}\eqref{goodL-2}:=\Q^*_{0,z_n+B}\bigg( \Big\vert \frac{\sigma^2}{2}\sum_{j=\lfloor n^\delta\rfloor}^{n-1}\frac{P_{j+1}(z_n-S_{w_{n-j-1}})}{\alpha\log (j+1)} - \log \frac{1}{\delta} \Big\vert \ge \varepsilon \bigg) \xrightarrow[n\to\infty]{} 0.
\end{equation}
In fact, recalling that under $\mathbb{P}$, $(S_k)_{k\ge0}$ is a random walk with jump law $\nu$ starting from 0, we have
\begin{equation*}
\Q^*_{0,z_n+B}\eqref{goodL-2}\leq \frac{1}{P_n(z_n+B)}\sum_{b\in B} \mathbb{P}\bigg( \Big\vert \frac{\sigma^2}{2}\sum_{j=\lfloor n^\delta\rfloor}^{n-1}\frac{P_{j+1}(z_n-S_{{n-j-1}})}{\alpha\log (j+1)} - \log \frac{1}{\delta} \Big\vert \ge \varepsilon , S_n =z_n+b \bigg).
\end{equation*}
As $(S_k)_{0\le k\le n}$ is identically distributed as $(S_n-S_{n-k})_{0\le k\le n}$, 
\begin{align*}
&\Q^*_{0,z_n+B}\eqref{goodL-2}\\
\leq\,& \frac{1}{P_n(z_n+B)}\sum_{b\in B} \mathbb{P}\bigg( \Big\vert \frac{\sigma^2}{2}\sum_{j=\lfloor n^\delta \rfloor}^{n-1}\frac{P_{j+1}(z_n-(S_n-S_{j+1}))}{\alpha\log (j+1)} - \log \frac{1}{\delta} \Big\vert \ge \varepsilon , S_n =z_n+b \bigg)\\
=\,& \frac{1}{P_n(z_n+B)}\sum_{b\in B} \mathbb{P}\bigg( \Big\vert \frac{\sigma^2}{2}\sum_{j=\lfloor n^\delta \rfloor+1}^{n}\frac{P_{j}(-b+S_{j})}{\alpha\log (j)} - \log \frac{1}{\delta} \Big\vert \ge \varepsilon , S_n =z_n+b \bigg).
\end{align*}
By \eqref{upbdPnx}, we have the uniform estimate
\begin{equation}\label{goodL-badj}
\sum_{j=\lfloor \delta n\rfloor+1}^{n} \frac{P_{j}(-b+S_{j})}{\alpha\log j} \leq \sum_{j=\lfloor \delta n\rfloor+1}^{n} \frac{P_{j}(-b+S_{j})}{\alpha\log(\delta n) }  \lesssim \frac{-\log\delta}{\log(\delta n)} = o_n(1).
\end{equation}
It thus follows that for sufficiently large $n$,
\begin{align*}
\Q^*_{0,z_n+B}\eqref{goodL-2}\le &  \frac{1}{P_n(z_n+B)}\sum_{b\in B} \mathbb{P}\bigg( \Big\vert \frac{\sigma^2}{2}\sum_{j=\lfloor n^\delta\rfloor+1}^{\lfloor \delta n\rfloor}\frac{P_{j}(-b+S_{j})}{\alpha\log (j)} - \log \frac{1}{\delta} \Big\vert \ge \varepsilon/2 , S_n =z_n+b \bigg) \\
=&  \frac{1}{P_n(z_n+B)} \sum_{b\in B}\mathbb{E}\Bigg[\ind{ \big\vert \frac{\sigma^2}{2}\sum_{j=\lfloor n^\delta\rfloor+1}^{\lfloor \delta n\rfloor}\frac{P_{j}(-b+S_{j})}{\alpha\log (j)} - \log \frac{1}{\delta} \big\vert \ge \varepsilon/2}\mathbb{P}_{S_{\lfloor \delta n\rfloor}}(S_{n-\lfloor \delta n\rfloor} = z_n +b)\Bigg],
\end{align*}
where the second line comes from the Markov property. Note that by \eqref{Pnx}, there exists some constant $C(\delta)<\infty$ independent of $n$ so that uniformly in $\|z_n\|\le K\sqrt{n}$, 
\[
\frac{\mathbb{P}_{S_{\delta n}}(S_{n-\delta n}= z_n+b)}{P_n(z_n+B)}\le \sup_{z\in \Z^2} \frac{P_{n-\delta n}(z_n+b-z)}{P_n(z_n+B)} \le C(\delta).
\]
This entails that
\begin{align}\label{goodL-3}
\sup_{\|z_n\|\le K\sqrt{n}}\Q^*_{0,z_n+B}\eqref{goodL-2}\lesssim & \sum_{b\in B}\mathbb{P}\bigg( \Big\vert \frac{\sigma^2}{2}\sum_{j=\lfloor n^\delta\rfloor+1}^{\lfloor\delta n\rfloor}\frac{P_{j}(-b+S_{j})}{\alpha\log (j)} - \log \frac{1}{\delta} \Big\vert \ge \varepsilon/2 \bigg).
\end{align}
As $b\in B$ is a fixed point, we have known from \eqref{eq:cv-proba-Gamma} or from Lemma \ref{lem: SRW} that for any fixed integer $K\geq 1$, 
the following convergence in probability holds:
\begin{equation}
   \label{eq:cv-proba-sumP} 
    \Big|\frac{1}{\log m}\sum_{j=K}^m P_j(-b+S_{j})- \frac{5}{8\pi}\Big| \xrightarrow[m\to\infty]{\mathbb{P}}0.
\end{equation}
\begin{lemma}
    Recall that $\delta\in(0,1/2)$ is fixed. We have the convergence in probability
\begin{equation}
   \label{eq:cv-proba-max} 
   \max_{\lfloor n^\delta\rfloor \le m \le n}\Big|\frac{1}{\log m}\sum_{j=K}^m P_j(-b+S_{j})- \frac{5}{8\pi}\Big| \xrightarrow[n\to\infty]{\mathbb{P}}0.
\end{equation}
\end{lemma}
\begin{proof}
Let us write $\Gamma_K^n(-b)=\sum_{j=K}^n P_j(-b+S_j)$ as in Lemma \ref{lem: SRW}. 
Given an arbitrary $\varepsilon\in (0,\delta)$, we set $A_\varepsilon=\lfloor \delta/\varepsilon\rfloor$ and $B_\varepsilon=\lfloor 1/\varepsilon\rfloor$.
For every integer $A_\varepsilon\leq k\leq B_\varepsilon$, and for every $\lfloor n^{k\varepsilon}\rfloor\leq m\leq\lfloor n^{(k+1)\varepsilon}\rfloor$, by monotonicity we have
\[
    \frac{\Gamma_K^{\lfloor n^{k\varepsilon}\rfloor}(-b)}{\varepsilon(k+1)\log n}\leq \frac{\Gamma_K^{m}(-b)}{\log m}\leq \frac{\Gamma_K^{\lfloor n^{(k+1)\varepsilon}\rfloor}(-b)}{\varepsilon k\log n},
\]
so that
\begin{align}
    \Big|\frac{\Gamma_{K}^m(-b)}{\log m}- \frac{5}{8\pi}\Big| \leq &\,\Big|\frac{\Gamma_K^{\lfloor n^{k\varepsilon}\rfloor}(-b)}{\varepsilon(k+1)\log n}- \frac{5}{8\pi}\Big|\vee \Big|\frac{\Gamma_K^{\lfloor n^{(k+1)\varepsilon}\rfloor}(-b)}{\varepsilon k\log n}- \frac{5}{8\pi}\Big|\label{eq:Gamma-upper-bd}\\
    \leq &\, \Big|\frac{\Gamma_K^{\lfloor n^{k\varepsilon}\rfloor}(-b)}{\varepsilon k\log n} \cdot\frac{k}{k+1}- \frac{5}{8\pi}\Big|\vee \Big|\frac{\Gamma_K^{\lfloor n^{(k+1)\varepsilon}\rfloor}(-b)}{\varepsilon (k+1)\log n}\cdot \frac{k+1}{k}- \frac{5}{8\pi}\Big|.\nonumber
\end{align}
By \eqref{upbdPnx}, there exists some constant $C>0$ such that 
\[
    \sup_{n\geq K} \frac{\Gamma_K^n(-b)}{\log n}\leq C.
\]
Then, from 
\[
    \Big|\frac{\Gamma_K^{\lfloor n^{k\varepsilon}\rfloor}(-b)}{\varepsilon k\log n} \cdot\frac{k}{k+1}- \frac{5}{8\pi}\Big| \leq  \Big|\frac{\Gamma_K^{\lfloor n^{k\varepsilon}\rfloor}(-b)}{\varepsilon k\log n}- \frac{5}{8\pi}\Big|+\frac{1}{k+1} \frac{\Gamma_K^{\lfloor n^{k\varepsilon}\rfloor}(-b)}{\varepsilon k\log n} 
\]
we derive that 
\[
    \max_{A_\varepsilon\leq k\leq B_\varepsilon}\Big|\frac{\Gamma_K^{\lfloor n^{k\varepsilon}\rfloor}(-b)}{\varepsilon k\log n} \cdot\frac{k}{k+1}- \frac{5}{8\pi}\Big| \leq \max_{A_\varepsilon\leq k\leq B_\varepsilon}\Big|\frac{\Gamma_K^{\lfloor n^{k\varepsilon}\rfloor}(-b)}{\varepsilon k\log n}- \frac{5}{8\pi}\Big|+ \frac{C}{A_\varepsilon}.
\]
In the same way, we also have
\[
    \max_{A_\varepsilon\leq k\leq B_\varepsilon}\Big|\frac{\Gamma_K^{\lfloor n^{(k+1)\varepsilon}\rfloor}(-b)}{\varepsilon (k+1)\log n} \cdot\frac{k+1}{k}- \frac{5}{8\pi}\Big| \leq \max_{A_\varepsilon\leq k\leq B_\varepsilon}\Big|\frac{\Gamma_K^{\lfloor n^{(k+1)\varepsilon}\rfloor}(-b)}{\varepsilon (k+1)\log n}- \frac{5}{8\pi}\Big|+ \frac{C}{A_\varepsilon}.
\]
Going back to \eqref{eq:Gamma-upper-bd}, it yields that
\[
    \max_{\lfloor n^\delta\rfloor \le m \le n}\Big|\frac{1}{\log m}\sum_{j=K}^m P_j(-b+S_{j})- \frac{5}{8\pi}\Big|\leq \max_{A_\varepsilon\leq k\leq B_\varepsilon+1}\Big|\frac{\Gamma_K^{\lfloor n^{k\varepsilon}\rfloor}(-b)}{\varepsilon k\log n}- \frac{5}{8\pi}\Big|+ \frac{C}{A_\varepsilon}.
\]
By choosing $\varepsilon$ sufficiently small, we can make $C/A_\varepsilon$ as small as we want. Meanwhile, for fixed $\varepsilon$, 
\[
    \max_{A_\varepsilon\leq k\leq B_\varepsilon+1}\Big|\frac{\Gamma_K^{\lfloor n^{k\varepsilon}\rfloor}(-b)}{\varepsilon k\log n}- \frac{5}{8\pi}\Big|\xrightarrow[n\to\infty]{\mathbb{P}}0
\]
is a direct consequence of \eqref{eq:cv-proba-sumP}. The convergence \eqref{eq:cv-proba-max} readily follows. 
\end{proof}

Now we need another elementary result. 
\begin{lemma}
    For a sequence of positive numbers $(y_n)_{ n\geq 1}$, if 
    \[
        \max_{\lfloor n^\delta\rfloor \le m\le  n}\Big|\frac{1}{\log m}\sum_{j=K}^{m}y_j -1 \Big|\xrightarrow[n\to\infty]{} 0,
    \]
    then 
    \[
        \sum_{j=\lfloor n^\delta\rfloor+1}^{n} \frac{y_j}{\log j} \xrightarrow[n\to\infty]{} \log \frac{1}{\delta}.
    \]
\end{lemma}
\begin{proof}
We set $T_m:=\sum_{j=K}^m y_j$ for $m\geq K$. Then for sufficiently large $n$,
\begin{align*}
 \sum_{j=\lfloor n^\delta\rfloor+1}^{n} \frac{y_j}{\log j} = & \sum_{j=\lfloor n^\delta\rfloor+1}^{n} \frac{T_j-T_{j-1}}{\log j}= \sum_{j=\lfloor n^\delta\rfloor+1}^{n} \frac{T_j}{\log j} -\sum_{j=\lfloor n^\delta\rfloor}^{n-1}\frac{T_j}{\log(j+1)}\\
 =&\, \frac{T_{n}}{ \log n} -\frac{T_{\lfloor n^\delta\rfloor}}{\log(\lfloor n^\delta\rfloor+1)}+ \sum_{j=\lfloor n^\delta\rfloor+1}^{n-1}T_j\Big(\frac{1}{\log j}-\frac{1}{\log (j +1)}\Big)\\
 =&\, o_n(1) + \sum_{j=\lfloor n^\delta\rfloor+1}^{n-1} \frac{T_j}{\log j} \frac{\log(1+\frac1j)}{\log(j+1)}  \quad \textrm{ by assumption. }
\end{align*}
Uniformly for $\lfloor n^\delta\rfloor \le j \le n$, we have $\frac{\log(1+\frac1j)}{\log(j+1)}= \frac{1+o_n(1)}{j\log j}$ and $\frac{T_j}{\log j}=1+o_n(1)$. Therefore,
\[
 \sum_{j=\lfloor n^\delta\rfloor+1}^{n} \frac{y_j}{\log j} = o_n(1) + (1+o_n(1))\sum_{j=\lfloor n^\delta\rfloor+1}^{n-1} \frac{1}{j\log j}
\]
which converges to $\log\frac1\delta$ as $n\to \infty$.
\end{proof}
Recall that the convergence in probability can be characterized by the property that every subsequence has a further subsequence that converges to the same limit almost surely. 
Applying the last lemma to \eqref{eq:cv-proba-max}, we see that
\[
\frac{\sigma^2}{2}\sum_{j=\lfloor n^\delta\rfloor+1}^{ n}\frac{P_{j}(-b+S_{j})}{\alpha\log (j)} = \sum_{j=\lfloor n^\delta\rfloor+1}^{n}\frac{P_{j}(-b+S_{j})}{\frac{5}{8\pi}\log (j)} \xrightarrow[n\to\infty]{\mathbb{P}} \log \frac1\delta.
\]
Combining this with \eqref{goodL-badj} and \eqref{goodL-3}, we conclude that uniformly in $\|z_n\|\le K\sqrt{n}$,
\[
\Q^*_{0,z_n+B}\eqref{goodL-2}\xrightarrow[n\to\infty]{}0,
\]
which completes the proof of \eqref{goodL}.
\end{proof}

\section{Lalley--Zheng conjecture: proof of Theorem \ref{thm: Typical}}
\label{sec:lalley_zheng_conjecture_proof}
Recall that $u(n)$ is uniformly chosen among all alive individuals at generation $n$, and $T_n = Z_n (S_{u(n)})$ denotes the number of all alive individuals that are found at the position of $u(n)$ at the same time. 
Let us consider the joint Laplace transform of $(\frac{T_n}{\alpha\log n}, \frac{Z_n}{\sigma^2 n/2})$ conditionally on $Z_n\ge 1$: for all $\lambda,\xi \ge 0$, we define
\begin{align*}
L_n(\lambda,\xi):=&\,\E\left[\exp\Big(\!-\lambda \frac{T_n}{\alpha\log n}-\xi \frac{Z_n}{\sigma^2 n/2}\Big) \Big\vert Z_n \ge 1\right]\\
=&\,\frac{1}{\P(Z_n\ge 1)}\E\left[\exp\Big(\!-\lambda \frac{T_n}{\alpha\log n}-\xi \frac{Z_n}{\sigma^2 n/2}\Big)  \ind{Z_n \ge 1}\right].
\end{align*}
To prove Theorem \ref{thm: Typical}, we will fix $\lambda,\xi \ge 0$ and show that when $n\to\infty$, $L_n(\lambda,\xi)$ converges to $(1+\lambda)^{-2}(1+\xi)^{-1}$, which is exactly the product of the Laplace transform of $\Gamma(2,1)$ distribution and the Laplace transform of $\mathrm{Exp}(1)$ distribution. 

Recall the classical Kolmogorov estimate $\P(Z_n\ge 1)=(1+o_n(1))\frac{2}{\sigma^2 n}$ as $n\to \infty$.
As $\P(Z_n \ge \varepsilon n) = \P(Z_n/n \ge \varepsilon \vert Z_n \ge 1)\P(Z_n \ge 1)$, the classical Yaglom theorem implies that when $\varepsilon\to0+$,
\begin{equation*}
     \lim_{n\to \infty}\frac{\P(Z_n \ge \varepsilon n)}{\P(Z_n \ge 1)}=1+o_\varepsilon(1), \textrm { or equivalently }\lim_{n\to \infty}\frac{\P(1\leq Z_n < \varepsilon n)}{\P(Z_n \ge 1)}=o_\varepsilon(1).
\end{equation*}
It follows that 
\[
    \limsup_{n\to \infty} \frac{1}{\P(Z_n\ge 1)}\E\left[\exp\Big(\!-\lambda \frac{T_n}{\alpha\log n}-\xi \frac{Z_n}{\sigma^2 n/2}\Big) \ind{1\le Z_n < \varepsilon n}\right]=o_\varepsilon(1).
\]
Since 
\begin{align*}
L_n(\lambda,\xi) = &\, \frac{1}{\P(Z_n\ge 1)}\E\left[\exp\Big(\!-\lambda \frac{T_n}{\alpha\log n}-\xi \frac{Z_n}{\sigma^2 n/2}\Big)  \ind{Z_n \ge \varepsilon n}\right] \\
 &\, + \frac{1}{\P(Z_n\ge 1)}\E\left[\exp\Big(\!-\lambda \frac{T_n}{\alpha\log n}-\xi \frac{Z_n}{\sigma^2 n/2}\Big) \ind{1\le Z_n < \varepsilon n}\right],
\end{align*}
it suffices to study the convergence of 
\[
   L_n(\lambda,\xi,\varepsilon):=\E\left[\exp\Big(\!-\lambda \frac{T_n}{\alpha\log n}-\xi \frac{Z_n}{\sigma^2 n/2}\Big)  \ind{Z_n \ge \varepsilon n}\right].
\]
In the end, we will take $n\to \infty$ first and then let $\varepsilon\to0+$.
Observe that
\begin{align*}
\E\left[ \exp\Big(\!-\lambda \frac{T_n}{\alpha\log n}\Big) \Big\vert \mathcal{F}_n \right] = \ind{Z_n\ge1} \sum_{|u|=n}\frac{1}{Z_n} \exp\Big(\!-\lambda \frac{Z_n(S_u)}{\alpha \log n}\Big),
\end{align*}
from which we deduce by change of measure that
\begin{align}
L_n(\lambda,\xi,\varepsilon)= & \, \E_{\Q^*}\left[ \sum_{|u|=n}\frac{1}{Z_n} \exp\Big(\!-\lambda \frac{Z_n(S_u)}{\alpha \log n}\Big)\frac{1}{Z_n}\exp\Big(\!-\xi \frac{Z_n}{\sigma^2 n/2}\Big) \ind{Z_n \ge \varepsilon n} \right] \nonumber \\
=&\,\frac{1}{n}  \E_{\Q^*}\left[   \exp\Big(\!-\lambda \frac{Z_n(S_{w_n})}{\alpha \log n}\Big)\frac{n}{Z_n}\exp\Big(\!-\xi \frac{Z_n}{\sigma^2 n/2}\Big)\ind{Z_n \ge \varepsilon n} \right]. \nonumber 
\end{align}
For convenience, let $\B(w_k)=\cL(w_k)\cup\cR(w_k)$ denote the set of all siblings of $w_k$. Note that $\E_{\Q^*}[|\B(w_k)|]=\sigma^2$. Then, with the same notation used in Section \ref{sec:local_survival_probability} we have the decomposition
\begin{align*}
Z_n=& 1+\sum_{k=1}^n \sum_{u\in\B(w_k)}Z^u_{n-k} \textrm{ and }
Z_n(z)= \ind{S_{w_n}=z}+ \sum_{k=1}^n \sum_{u\in\B(w_k)}Z^u_{n-k}(z-S_u), \forall z\in\Z^2.
\end{align*}
For any subset (of real numbers) $I\subset[1,n]$, we set 
\[
Z_{I,n}:= \sum_{k\in I\cap\{1,\cdots,n\}} \sum_{u\in\B(w_k)}Z^u_{n-k},\textrm{ and } Z_{I,n}(z):=  \sum_{k\in I\cap\{1,\cdots,n\}} \sum_{u\in\B(w_k)}Z^u_{n-k}(z-S_u), \forall z\in\Z^2.
\]
We claim that under $\Q^*$, $Z_{(n-\frac{n}{\log n}, n],n}=o(n)$ and $ Z_{[1,n-\frac{n}{\log n}], n}(S_{w_n})=o(\log n)$ with high probability. To this end, we write $\delta := \varepsilon^3$ with $\varepsilon<\frac 1 2$. By Markov's inequality and the fact that $\E_{\Q^*}[Z^u_{n-k}]=1$,
\begin{equation}
\Q^*( Z_{(n-\frac{n}{\log n}, n],n} > \delta n) \le \frac{1}{\delta n}  \E_{\Q^*}\left[ \sum_{n-\frac{n}{\log n} < k\le n}\sum_{u\in\B(w_k)}Z^u_{n-k} \right]
\le \frac{1}{\delta n} \sigma^2 \frac{n}{\log n} \lesssim \frac{1}{\eps^3 \log n} = o_n(1).\label{goodZpart}
\end{equation}
Moreover, by Markov's inequality and \eqref{upbdPnx},
\begin{align}
\label{goodTpart}
\Q^*( Z_{[1,n-\frac{n}{\log n}], n}(S_{w_n}) > \delta \log n) \le &\, \frac{1}{\delta \log n}\E_{\Q^*}\left[ \sum_{k=1}^{\lfloor n- \frac{n}{\log n} \rfloor} \sum_{u\in\B(w_k)}Z^u_{n-k}(S_{w_n}-S_u) \right]\\
=& \, \frac{1}{\delta \log n} \E_{\Q^*}\left[ \sum_{k=1}^{\lfloor n-\frac{n}{\log n} \rfloor} \sum_{u\in\B(w_k)} P_{n-k}(S_{w_n}-S_u)\right] \nonumber\\
\lesssim & \,\frac{\sigma^2}{\delta \log n}  \sum_{k=1}^{\lfloor n-\frac{n}{\log n} \rfloor} \frac{1}{n-k} \lesssim \frac{\log \log n}{\varepsilon^3 \log n}=o_n(1). \nonumber
\end{align}
Now, we define $L_n(\lambda,\xi,\varepsilon,\delta)$ to be the following expectation
\[
\frac{1}{n}\E_{\Q^*}\left[   e^{-\lambda \frac{Z_n(S_{w_n})}{\alpha \log n}} \frac{n}{Z_n} e^{-\xi \frac{Z_n}{\sigma^2 n/2}} \ind{Z_n \ge \varepsilon n} \ind{ Z_{(n-\frac{n}{\log n}, n],n} \le \delta n, Z_{[1,n-\frac{n}{\log n}], n}(S_{w_n}) \le \delta \log n} \right].
\]
Note that
\[
    1- \ind{ Z_{(n-\frac{n}{\log n}, n],n} \le \delta n, Z_{[1,n-\frac{n}{\log n}], n}(S_{w_n}) \le \delta \log n} 
    \le \ind{Z_{(n-\frac{n}{\log n}, n],n} > \delta n} + \ind{Z_{[1,n-\frac{n}{\log n}], n}(S_{w_n}) > \delta \log n}.
\]
So we have
\begin{align*}
&\left\vert  \E_{\Q^*}\left[  e^{-\lambda \frac{Z_n(S_{w_n})}{\alpha \log n}} \frac{n}{Z_n}e^{-\xi \frac{Z_n}{\sigma^2 n/2}} \ind{Z_n \ge \varepsilon n} \right] - n L_n(\lambda,\xi, \varepsilon,\delta)\right\vert \\
\le\, & \frac1\varepsilon \left( \Q^*(  Z_{(n-\frac{n}{\log n}, n],n} > \delta n) +\Q^*(Z_{[1,n-\frac{n}{\log n}], n}(S_{w_n}) > \delta \log n) \right),
\end{align*}
which is $o_n(1)$ according to \eqref{goodZpart} and \eqref{goodTpart}. 
Consequently,
\begin{equation}
\label{LaplaceT-1-2}
\frac{1}{\P(Z_n\ge 1)}|L_n(\lambda,\xi,\varepsilon)-L_n(\lambda,\xi,\varepsilon,\delta)|=o_n(1).
\end{equation}

Recall that 
\begin{align*}
Z_n(S_{w_n}) = &\, 1+Z_{[1,n-\frac{n}{\log n}],n}(S_{w_{n}}) + Z_{(n-\frac{n}{\log n},n], n}(S_{w_n}), \\
Z_n = &\, 1+Z_{[1,n-\frac{n}{\log n}],n}+Z_{(n-\frac{n}{\log n}, n],n}.
\end{align*}
On the event $\{ Z_n\ge \varepsilon n, Z_{(n-\frac{n}{\log n}, n],n} \le \delta n, Z_{[1,n-\frac{n}{\log n}], n}(S_{w_n}) \le \delta \log n \} $, we apply the elementary inequalities
\begin{align*}
    |\exp(-\lambda(x+h))-\exp(-\lambda x)|\leq &\, \lambda h, \\
    \Big| \frac{1}{x+h}-\frac{1}{x}\Big|=\Big| \frac{h}{x(x+h)}\Big|\leq &\, \frac{2h}{\varepsilon^2} \leq \frac{2\delta}{\varepsilon^2}=2\varepsilon \quad \textrm{ if } h\leq \delta\leq \frac{\varepsilon}{2}, x+h\geq \varepsilon,
\end{align*}
to obtain 
\begin{align*}
\Big|\exp\Big(\!-\lambda \frac{Z_n(S_{w_n})}{\alpha \log n}\Big) -\exp\Big(\!-\lambda \frac{1+Z_{(n-\frac{n}{\log n},n],n}(S_{w_n})}{\alpha \log n}\Big) \Big|&\leq  \, \lambda \frac{\delta}{\alpha} = O(\varepsilon^3),\\
\Big|\exp\Big(\!-\xi \frac{Z_n}{\sigma^2 n/2}\Big) -\exp\Big(\!-\xi \frac{1+Z_{[1,n-\frac{n}{\log n}],n}}{\sigma^2 n/2}\Big) \Big|&\leq  \, \xi \frac{2\delta}{\sigma^2} = O(\varepsilon^3),\\
\Big|\frac{n}{Z_n} - \frac{n}{1+Z_{[1,n-\frac{n}{\log n}],n}} \Big|& \leq 2\varepsilon= O(\varepsilon).
\end{align*}
Meanwhile, $n/Z_n\leq 1/\varepsilon$ under the same event.
Therefore, we deduce that
\begin{equation}\label{LaplaceT-2}
\frac{1}{\P(Z_n\ge 1)}|L_n(\lambda,\xi,\varepsilon,\delta)-\widetilde{L}_n(\lambda,\xi,\varepsilon,\delta)|  = o_\varepsilon(1),
\end{equation}
where $\widetilde{L}_n(\lambda,\xi,\varepsilon,\delta)$ is defined by
{\footnotesize
\[
\frac{1}{n}\E_{\Q^*}\left[ \exp(-\lambda \frac{1+Z_{(n-\frac{n}{\log n},n],n}(S_{w_n})}{\alpha \log n}) \exp(-\xi \frac{1+Z_{[1,n-\frac{n}{\log n}],n}}{\sigma^2 n/2}) \frac{\ind{Z_n\ge \eps n,Z_{(n-\frac{n}{\log n}, n],n} \le \delta n, Z_{[1,n-\frac{n}{\log n}], n}(S_{w_n}) \le \delta \log n}}{\frac{1+Z_{[1,n-\frac{n}{\log n}],n}}{n} } \right].
\]
}

For the upper bound of the last expectation, as 
\[
    \big\{Z_n\ge \varepsilon n,Z_{(n-\frac{n}{\log n}, n],n} \le \delta n\big\}\subset\big\{ 1+Z_{[1,n-\frac{n}{\log n}],n} \ge \frac{\varepsilon}{2} n \big\},
\] 
we have
{\footnotesize
\begin{align*}
n\widetilde{L}_n(\lambda,\xi,\eps,\delta)\le  \E_{\Q^*}\left[ \exp(-\lambda \frac{Z_{(n-\frac{n}{\log n},n],n}(S_{w_n})}{\alpha \log n}) \exp(-\xi \frac{1+Z_{[1,n-\frac{n}{\log n}],n}}{\sigma^2 n/2})  \frac{n}{1+Z_{[1,n-\frac{n}{\log n}],n}} \ind{ 1+Z_{[1,n-\frac{n}{\log n}],n} \ge \frac{\eps}{2} n}\right].
\end{align*}}
The key observation here is that $Z_{(n-\frac{n}{\log n},n],n}(S_{w_n})$ is $\sigma((S_{w_k})_{k\le n}, (u, (S_v)_{u\leq v})_{u\in\cup_{n-\frac{n}{\log n}< k\le n}\B(w_k)})$-measurable, while $Z_{[1,n-\frac{n}{\log n}],n}$ is $\sigma((u, (Z^u_m)_{m\ge0})_{u\in\cup_{1\le k\le n-\frac{n}{\log n }} \B(w_k)})$-measurable. Hence, they are independent of each other under $\Q^*$. 
As a result,
{\footnotesize
\begin{align}\label{LaplaceT-3+}
n\widetilde{L}_n(\lambda,\xi,\eps,\delta)\le \E_{\Q^*}\left[ \exp(-\lambda \frac{Z_{(n-\frac{n}{\log n},n],n}(S_{w_n})}{\alpha \log n}) \right] \E_{\Q^*}\left[ \exp(-\xi \frac{1+Z_{[1,n-\frac{n}{\log n}],n}}{\sigma^2 n/2})  \frac{\ind{ 1+Z_{[1,n-\frac{n}{\log n}],n} \ge \frac{\eps}{2} n}}{\frac{1+Z_{[1,n-\frac{n}{\log n}],n}}{n} }  \right].
\end{align}
}

For the second expectation on the right-hand side of \eqref{LaplaceT-3+}, notice that 
\[
    \big\{1+Z_{[1,n-\frac{n}{\log n}],n} \ge \frac{\varepsilon}{2} n, Z_{(n-\frac{n}{\log n}, n],n}\le \delta n\big\}\subset\big\{ Z_n \ge  \frac{\varepsilon}{2} n\big\}.
\]
Using \eqref{goodZpart} and the same arguments in the derivation of \eqref{LaplaceT-2}, one sees that
\begin{align*}
&\E_{\Q^*}\left[ \exp(-\xi \frac{1+Z_{[1,n-\frac{n}{\log n}],n}}{\sigma^2 n/2})  \frac{\ind{ 1+Z_{[1,n-\frac{n}{\log n}],n} \ge \frac{\varepsilon}{2} n}}{\frac{1+Z_{[1,n-\frac{n}{\log n}],n}}{n} }  \right]\\
=\,&\E_{\Q^*}\left[ \exp(-\xi \frac{1+Z_{[1,n-\frac{n}{\log n}],n}}{\sigma^2 n/2})  \frac{ \ind{ 1+Z_{[1,n-\frac{n}{\log n}],n} \ge \frac{\varepsilon}{2} n, Z_{(n-\frac{n}{\log n}, n],n} \le \delta n}}{\frac{1+Z_{[1,n-\frac{n}{\log n}],n}}{n} }\right] + o_n(1)\\
=\,& \E_{\Q^*}\left[ \exp(-\xi \frac{Z_n}{\sigma^2 n/2})  \frac{n}{Z_n}\ind{ 1+Z_{[1,n-\frac{n}{\log n}],n} \ge \frac{\varepsilon}{2} n, Z_{(n-\frac{n}{\log n}, n],n}\le \delta n}  \right] +o_\varepsilon(1) +o_n(1)\\
\leq \, & \E_{\Q^*}\left[ \exp(-\xi \frac{Z_n}{\sigma^2 n/2})  \frac{n}{Z_n}\ind{Z_n \ge \frac{\varepsilon}{2} n } \right] + o_\varepsilon(1) +o_n(1).
\end{align*}
On the other hand, as $\{1+Z_{[1,n-\frac{n}{\log n}],n} \ge \frac{\varepsilon}{2} n, Z_{(n-\frac{n}{\log n}, n],n}\le \delta n\}\supset \{Z_n \ge \varepsilon n, Z_{(n-\frac{n}{\log n}, n],n}\le \delta n\}$, we get similarly the lower bound 
{\footnotesize
\begin{align*}
\E_{\Q^*}\left[ \exp(-\xi \frac{1+Z_{[1,n-\frac{n}{\log n}],n}}{\sigma^2 n/2})  \frac{\ind{ 1+Z_{[1,n-\frac{n}{\log n}],n} \ge \frac{\eps}{2} n}}{\frac{1+Z_{[1,n-\frac{n}{\log n}],n}}{n} }  \right]
\ge \,&\E_{\Q^*}\left[ \exp(-\xi \frac{Z_n}{\sigma^2 n/2})  \frac{n}{Z_n}\ind{ Z_n\ge \eps n, Z_{(n-\frac{n}{\log n}, n],n}\le \delta n}  \right]\\
=\, & \E_{\Q^*}\left[ \exp(-\xi \frac{Z_n}{\sigma^2 n/2})  \frac{n}{Z_n}\ind{ Z_n\ge \eps n }\right] + o_\eps(1) +o_n(1).
\end{align*}
}
Besides, observe that by change of measure,
\begin{align*}
\E_{\Q^*}\left[ \exp(-\xi \frac{Z_n}{\sigma^2 n/2})  \frac{n}{Z_n}\ind{ \frac{\eps}{2} n\le Z_n\le \eps n }\right] \leq \,&\E_{\Q^*}\left[ \frac{n}{Z_n}\ind{ \frac{\eps}{2} n\le Z_n\le \eps n }\right]\\
=\,& \E\left[ Z_n\frac{n}{Z_n}\ind{\frac{\eps}{2} n\le Z_n\le \eps n}\right]= n\P\Big( \frac{\eps}{2} n\le Z_n\le \eps n\Big)\\
=\,& \P\Big( \frac{\eps}{2} n\le Z_n\le \eps n \Big\vert Z_n\geq 1\Big)\cdot n\P(Z_n\geq 1).
\end{align*}
From the classical Kolmogorov estimate and the Yaglom theorem, it follows that
\[
    \limsup_{\varepsilon\to 0+}\bigg(\limsup_{n\to \infty}\E_{\Q^*}\left[ \exp(-\xi \frac{Z_n}{\sigma^2 n/2})  \frac{n}{Z_n}\ind{ \frac{\varepsilon}{2} n\le Z_n\le \varepsilon n }\right]\bigg)=0.
\]
Combining this with the preceding upper and lower bounds, we obtain 
{\footnotesize
\begin{equation}\label{goodZ-asymp}
\E_{\Q^*}\left[\exp(-\xi \frac{1+Z_{[1,n-\frac{n}{\log n}],n}}{\sigma^2 n/2})  \frac{\ind{ 1+Z_{[1,n-\frac{n}{\log n}],n} \ge \frac{\eps}{2} n}}{\frac{1+Z_{[1,n-\frac{n}{\log n}],n}}{n} }  \right]=\E_{\Q^*}\left[ \exp(-\xi \frac{Z_n}{\sigma^2 n/2})  \frac{n}{Z_n}\ind{ Z_n\ge \eps n }\right] + o_\eps(1) +o_n(1).
\end{equation}
}

For the first expectation on the right hand side of \eqref{LaplaceT-3+}, we continue in the same fashion, using \eqref{goodTpart} instead of \eqref{goodZpart}, to get 
\begin{align}
&\E_{\Q^*}\left[ \exp(-\lambda \frac{Z_{(n-\frac{n}{\log n},n],n}(S_{w_n})}{\alpha \log n}) \right] \label{goodT-asymp}\\
=\,&\E_{\Q^*}\left[ \exp(-\lambda \frac{Z_{(n-\frac{n}{\log n},n],n}(S_{w_n})}{\alpha \log n})\ind{ Z_{[1,n-\frac{n}{\log n}], n}(S_{w_n}) \le \delta \log n}\right]+o_n(1) \nonumber\\
=\,& \E_{\Q^*}\left[ \exp(-\lambda \frac{Z_n(S_{w_n})}{\alpha \log n}) \ind{ Z_{[1,n-\frac{n}{\log n}], n}(S_{w_n}) \le \delta \log n }\right] +o_\varepsilon(1) + o_n(1) \nonumber\\
=\,& \E_{\Q^*}\left[ \exp(-\lambda \frac{Z_n(S_{w_n})}{\alpha \log n})\right] +o_\varepsilon(1) +o_n(1), \nonumber
\end{align}
where in the last line, we have used \eqref{goodTpart} again. 
We hence end up with
\begin{align}\label{LaplaceT-upbd}
n\widetilde{L}_n(\lambda,\xi, \eps,\delta)\le  \E_{\Q^*}\left[ \exp(-\lambda \frac{Z_n(S_{w_n})}{\alpha \log n}) \right]\E_{\Q^*}\left[ \exp(-\xi \frac{Z_n}{\sigma^2 n/2})  \frac{n}{Z_n}\ind{ Z_n\ge \eps n }\right] + o_\eps(1) +o_n(1).
\end{align}

For the lower bound of $n\widetilde{L}_n(\lambda,\xi,\varepsilon,\delta)$, by use of \eqref{goodZpart} and \eqref{goodTpart} again, we get
\begin{align*}
&n\widetilde{L}_n(\lambda,\xi,\varepsilon,\delta)\\
& \ge \E_{\Q^*}\Biggl[ \exp(-\lambda \frac{1+Z_{(n-\frac{n}{\log n},n],n}(S_{w_n})}{\alpha \log n}) \exp(-\xi \frac{1+Z_{[1,n-\frac{n}{\log n}],n}}{\sigma^2 n/2})  \\
&\hspace{4cm}\times\frac{\ind{1+Z_{[1,n-\frac{n}{\log n}],n}\ge \varepsilon n,Z_{(n-\frac{n}{\log n}, n],n} \le \delta n, Z_{[1,n-\frac{n}{\log n}], n}(S_{w_n}) \le \delta \log n}}{\frac{1+Z_{[1,n-\frac{n}{\log n}],n}}{n} } \Biggr]\\
&=\,\E_{\Q^*}\left[ \exp(-\lambda \frac{1+Z_{(n-\frac{n}{\log n},n],n}(S_{w_n})}{\alpha \log n}) \exp(-\xi \frac{1+Z_{[1,n-\frac{n}{\log n}],n}}{\sigma^2 n/2})  \frac{\ind{1+Z_{[1,n-\frac{n}{\log n}],n}\ge \varepsilon n}}{\frac{1+Z_{[1,n-\frac{n}{\log n}],n}}{n} } \right]+o_n(1).
\end{align*}
By independence, it implies that
{\footnotesize
\begin{equation*}
n\widetilde{L}_n(\lambda,\xi,\eps,\delta)
\ge  \E_{\Q^*}\left[ \exp(-\lambda \frac{1+Z_{(n-\frac{n}{\log n},n],n}(S_{w_n})}{\alpha \log n}) \right]\E_{\Q^*}\left[\exp(-\xi \frac{1+Z_{[1,n-\frac{n}{\log n}],n}}{\sigma^2 n/2})   \frac{\ind{ 1+Z_{[1,n-\frac{n}{\log n}],n} \ge \eps n}}{\frac{1+Z_{[1,n-\frac{n}{\log n}],n}}{n} }  \right]+o_n(1).
\end{equation*}
}
Applying \eqref{goodZ-asymp} with $2\varepsilon$ instead of $\varepsilon$ and by \eqref{goodT-asymp}, we have
\begin{equation}\label{LaplaceT-3-}
n\widetilde{L}_n(\lambda,\xi,\eps,\delta)
\ge \E_{\Q^*}\left[ \exp(-\lambda \frac{Z_n(S_{w_n})}{\alpha \log n}) \right] \E_{\Q^*}\left[  \exp(-\xi \frac{Z_n}{\sigma^2 n/2})  \frac{n}{Z_n} \ind{ Z_n \ge 2\eps n} \right]+o_\eps(1)+o_n(1).
\end{equation}

Finally, it remains to find the asymptotics of 
\[
   \E_{\Q^*}\left[ \exp(-\xi \frac{Z_n}{\sigma^2 n/2})  \frac{n}{Z_n} \ind{ Z_n \ge \varepsilon n}  \right] \quad \textrm{ and } \quad \E_{\Q^*}\left[ \exp(-\lambda \frac{Z_n(S_{w_n})}{\alpha \log n}) \right] .
\]
For the first one, by change of measure, 
\begin{align*}
    \E_{\Q^*}\left[ \exp(-\xi \frac{Z_n}{\sigma^2 n/2})  \frac{n}{Z_n} \ind{ Z_n \ge \eps n}   \right]=&\, \E\left[ Z_n \exp(-\xi \frac{Z_n}{\sigma^2 n/2})  \frac{n}{Z_n} \ind{ Z_n \ge \eps n}   \right] \\
    =&\,  n\E\left[  \exp(-\xi \frac{Z_n}{\sigma^2 n/2}) \ind{ Z_n \ge \eps n}   \right]\\
    =&\, \E\left[  \exp(-\xi \frac{Z_n}{\sigma^2 n/2}) \ind{ Z_n \ge \eps n} \Big \vert Z_n\geq 1  \right]\times n\P(Z_n\geq 1).
\end{align*}
Using again Kolmogorov's estimate and the classical Yaglom theorem, we see that
\begin{equation}
\label{Z-asymp}
\lim_{\varepsilon\to 0+}\lim_{n\to \infty}\E_{\Q^*}\left[ \exp(-\xi \frac{Z_n}{\sigma^2 n/2})  \frac{n}{Z_n} \ind{ Z_n \ge \varepsilon n}   \right]= \frac{1}{1+\xi}\frac{2}{\sigma^2}.
\end{equation}
For the second one, it follows from the spinal decomposition that
\begin{align*}
&\E_{\Q^*}\left[ \exp(-\lambda \frac{Z_n(S_{w_n})}{\alpha \log n}) \right] =  \sum_{z\in\Z^2}\E_{\Q^*}\left[\ind{S_{w_n}=z} \exp(-\lambda \frac{Z_n(z)}{\alpha \log n}) \right] \\
=& \sum_{z\in\Z^2}\E_{\Q^*}\left[\sum_{|u|=n}\ind{w_n=u} \ind{S_{w_n}=z} \exp(-\lambda \frac{Z_n(z)}{\alpha \log n}) \right]  \\
=&\sum_{z\in\Z^2}\E\left[\sum_{|u|=n}\ind{S_u=z} \exp(-\lambda \frac{Z_n(z)}{\alpha \log n}) \right] = \sum_{z\in\Z^2}\E\left[ Z_n(z) \exp(-\lambda \frac{Z_n(z)}{\alpha \log n}) \right]. 
\end{align*}
Here we separate the last sum into two cases: $\|z\|\le K\sqrt{n}$ and $\|z\|> K\sqrt{n}$. It is immediate that 
\begin{align*}
\sum_{z\in\Z^2:\|z\| >K\sqrt{n}}\E\left[ Z_n(z) \exp(-\lambda \frac{Z_n(z)}{\alpha \log n}) \right] \le & \sum_{z\in\Z^2:|z| >K\sqrt{n}}\E\left[ Z_n(z)\right]\\
=& \sum_{z\in\Z^2:\|z\| >K\sqrt{n}} P_n(z) = \mathbb{P}(\|S_n\|> K\sqrt{n}) =o_K(1) 
\end{align*}
when $K\to \infty$.
We thus get 
\begin{align*}
&\E_{\Q^*}\left[ \exp(-\lambda \frac{Z_n(S_{w_n})}{\alpha \log n}) \right] =   \sum_{z\in\Z^2: \|z\| \le K\sqrt{n}}\E\left[ Z_n(z) \exp(-\lambda \frac{Z_n(z)}{\alpha \log n}) \right] +o_K(1)\\
=& \alpha \log n  \sum_{z\in\Z^2: \|z\| \le K\sqrt{n}}\E\left[ h_\lambda\Big(\frac{ Z_n(z) }{\alpha\log n} \Big)\Big\vert Z_n(z) \ge 1 \right]u_n(z) +o_K(1)
\end{align*}
with $h_\lambda(r) := \ind{r>0}re^{-\lambda r}$. By \eqref{asymp-unx}, uniformly for $\|z\|\le K\sqrt{n}$,
\[
u_n(z) = \frac{P_n(z)}{\alpha\log n} (1+o_n(1)).
\]
As the function $h_\lambda$ is Lipschitz continuous with Lipschitz constant 1, from the uniform convergence \eqref{dW-uniformcvg} we deduce that uniformly in $\|z\|\le K\sqrt{n}$,
\[
\E\left[ h_\lambda\Big(\frac{ Z_n(z) }{\alpha\log n} \Big)\Big\vert Z_n(z) \ge 1 \right] = \int_0^\infty h_\lambda (r) e^{-r}\d r + o_n(1)= \frac{1}{(1+\lambda)^2} (1+o_n(1)).
\]
As a consequence,
\begin{align*}
\E_{\Q^*}\left[ \exp(-\lambda \frac{Z_n(S_{w_n})}{\alpha \log n}) \right] = &\, (1+o_n(1)) \frac{1}{(1+\lambda)^2}  \sum_{z\in\Z^2: \|z\| \le K\sqrt{n}} P_n(z) +o_K(1)\\
=&\, (1+o_n(1)) \frac{1}{(1+\lambda)^2} \mathbb{P}(\|S_n\|\leq K\sqrt{n})   +o_K(1)\\
=&\, (1+o_n(1))\frac{1}{(1+\lambda)^2}(1+o_K(1)) +o_K(1) .
\end{align*}
Since $K$ is irrelevant to the last expectation on the left-hand side, it yields that
\begin{equation}
    \label{T-asymp}
    \E_{\Q^*}\left[ \exp(-\lambda \frac{Z_n(S_{w_n})}{\alpha \log n}) \right] =\frac{1}{(1+\lambda)^2}(1+o_n(1)).
\end{equation}
Taking into account \eqref{Z-asymp}, \eqref{T-asymp}, \eqref{LaplaceT-upbd} and \eqref{LaplaceT-3-}, we have by Kolmogorov's estimate
\[
\lim_{\varepsilon\to 0+}\lim_{n\to \infty} \frac{1}{\P(Z_n\geq 1)} \widetilde{L}_n(\lambda,\xi,\varepsilon,\delta) =  \frac{1}{(1+\lambda)^2} \frac{1}{1+\xi}.
\]
Plugging it into \eqref{LaplaceT-2} and then into \eqref{LaplaceT-1-2} yields that
\[
\lim_{\varepsilon\to 0+}\lim_{n\to \infty}  \frac{1}{\P(Z_n\geq 1)} L_n(\lambda,\xi,\varepsilon) = \frac{1}{(1+\lambda)^2} \frac{1}{1+\xi}.
\]
As explained at the beginning of this section, this implies the convergence $L_n(\lambda,\xi)\to \frac{1}{(1+\lambda)^2} \frac{1}{1+\xi}$ when $n\to\infty$,
which completes the proof of Theorem \ref{thm: Typical}.

\section{Proof of Theorem \ref{thm: Span}}
\label{sec:proof_of_theorem_ref_thm_span}

Throughout this section, the numbers $K, \lambda, \xi>0$ are always fixed. It suffices to consider the CBRW starting from the origin of $\Z^2$. We will follow the same strategy as in the proof of Theorem \ref{thm: Typical}, by showing that uniformly for $\|z_n\| \le K\sqrt{n}$,
\begin{equation}\label{LaplaceZ}
\mathbb{L}_n(\lambda,\xi) :=\E\left[ \exp(-\lambda \frac{Z_n}{\sigma^2 n/2} ) \exp(-\xi \frac{Z_n(z_n)}{\alpha \log n} )\Big\vert Z_n(z_n) \ge 1 \right] \xrightarrow[n\to \infty]{} \frac{1}{(1+\lambda)^2}\frac{1}{1+\xi}.
\end{equation}
Since the arguments will be carried out in parallel with the previous section, we will only present the main steps and leave some minor details to the reader. 
Note that
\[
\mathbb{L}_n(\lambda,\xi) = \frac{1}{u_n(z_n)}\E\left[ \exp(-\lambda \frac{Z_n}{\sigma^2 n/2} ) \exp(-\xi \frac{Z_n(z_n)}{\alpha \log n} ) \ind{Z_n(z_n) \ge 1} \right].
\]
Then \eqref{dW-uniformcvg} implies that uniformly for $\|z_n\|\le K\sqrt{n}$, when $\varepsilon\to0+$,
\begin{equation*}
     \lim_{n\to \infty}\frac{\P(1\le Z_n(z_n) < \varepsilon\log n) }{u_n(z_n)}=\lim_{n\to \infty}\P(1\le Z_n(z_n) < \varepsilon\log n\vert Z_n(z_n)\geq 1)=  o_\varepsilon(1).
\end{equation*}
In view of \eqref{asymp-unx}, it follows that 
\begin{align*}
\lim_{n\to \infty} \mathbb{L}_n(\lambda,\xi) = & \lim_{\eps\to 0+}\lim_{n\to \infty}\frac{1}{u_n(z_n)}\E\left[ \exp(-\lambda \frac{Z_n}{\sigma^2 n/2} ) \exp(-\xi \frac{Z_n(z_n)}{\alpha \log n} )  \ind{Z_n(z_n) \ge \eps\log n} \right] \\
=&  \lim_{\eps\to 0+}\lim_{n\to \infty}\frac{\alpha\log n}{P_n(z_n)}\E\left[ \exp(-\lambda \frac{Z_n}{\sigma^2 n/2} ) \exp(-\xi \frac{Z_n(z_n)}{\alpha \log n} )  \ind{Z_n(z_n) \ge \eps\log n} \right],
\end{align*}
after assuming the existence of all these limits.
By change of measure, 
\begin{align}\label{LaplaceZ-1}
&\frac{\alpha\log n}{P_n(z_n)}\E\left[ \exp(-\lambda \frac{Z_n}{\sigma^2 n/2} ) \exp(-\xi \frac{Z_n(z_n)}{\alpha \log n} )  \ind{Z_n(z_n) \ge \varepsilon\log n} \right]\\
=\,&\frac{\alpha\log n}{P_n(z_n)}\E\bigg[ \sum_{|u|=n}\ind{ S_u = z_n}\frac{1}{Z_n(z_n)} \exp(-\lambda \frac{Z_n}{\sigma^2 n/2} )  \exp(-\xi \frac{Z_n(z_n)}{\alpha \log n} ) \ind{Z_n(z_n) \ge \varepsilon\log n} \bigg] \nonumber \\
=\,& \frac{1}{P_n(z_n)}\E_{\Q^*}\bigg[ \ind{ S_{w_n}=z_n} \exp(-\lambda \frac{Z_n}{\sigma^2 n/2} ) \frac{\alpha \log n}{Z_n(z_n)} \exp(-\xi \frac{Z_n(z_n)}{\alpha \log n} ) \ind{Z_n(z_n) \ge \varepsilon\log n} \bigg]. \nonumber
\end{align}
We keep using the same notation introduced in the previous section. In particular, let us recall that $\delta=\varepsilon^3$.
Similarly as \eqref{goodTpart}, by Markov's inequality and \eqref{upbdPnx}, one sees that
\begin{align}\label{goodTpart-v2}
&\frac{1}{P_n(z_n)}\E_{\Q^*}\left[ \ind{ S_{w_n}=z_n}\ind{Z_{[1, n-\frac{n}{\log n}],n}(z_n) > \delta \log n } \right]\\
\le \,& \frac{1}{P_n(z_n)} \frac{1}{\delta \log n}\E_{\Q^*}\left[  \ind{ S_{w_n}=z_n}\sum_{k=1}^{\lfloor n-n/\log n\rfloor} \sum_{u\in\B(w_k)}Z^u_{n-k}(z_n-S_u) \right] \nonumber\\
=\,&   \frac{1}{P_n(z_n)} \frac{1}{\delta \log n}\E_{\Q^*}\left[  \ind{ S_{w_n}=z_n}\sum_{k=1}^{\lfloor n-n/\log n\rfloor} \sum_{u\in\B(w_k)} P_{n-k}(S_{w_n}-S_u)\right] \nonumber\\
\lesssim \,& \frac{\E_{\Q^*}[\ind{S_{w_n}=z_n }]}{P_n(z_n)} \frac{\sigma^2\sum_{k=1}^{\lfloor n-n/\log n\rfloor}\frac{1}{n-k}}{\delta \log n} = \frac{\sum_{k=1}^{\lfloor n-n/\log n\rfloor}\frac{1}{n-k}}{\delta \log n} = o_n(1). \nonumber
\end{align}
Meanwhile, the analog of \eqref{goodZpart} becomes 
\begin{align}\label{goodZpart-v2}
\frac{1}{P_n(z_n)}\E_{\Q^*}\left[ \ind{ S_{w_n}=z_n}\ind{Z_{(n-\frac{n}{\log n}, n],n} > \delta n } \right]
\le  & \,\frac{1}{P_n(z_n)} \frac{1}{\delta n}  \E_{\Q^*}\left[ \ind{S_{w_n}=z_n }\sum_{n-\frac{n}{\log n} < k\le n}\sum_{u\in\B(w_k)}Z^u_{n-k}\right] \\
\lesssim &\, \frac{\E_{\Q^*}[\ind{S_{w_n}=z_n }]}{P_n(z_n)}\frac{\sigma^2 n/\log n}{\delta n} =\frac{\sigma^2 }{\delta \log n} =o_n(1).\nonumber
\end{align}
We define $\mathbb{L}_n(\lambda,\xi, \eps,\delta)$ to be the product of $\frac{1}{P_n(z_n)}$ and
{\footnotesize
\[
    \E_{\Q^*}\left[ \ind{ S_{w_n}=z_n} \exp(-\lambda \frac{Z_n}{\sigma^2 n/2} )\frac{\alpha \log n}{Z_n(z_n)} \exp(-\xi \frac{Z_n(z_n)}{\alpha \log n} )\ind{Z_n(z_n) \ge \eps\log n, Z_{[1, n-\frac{n}{\log n}],n}(z_n) \le \delta \log n, Z_{(n-\frac{n}{\log n}, n],n} \le \delta n  }\right].
\]
}
Combining \eqref{goodTpart-v2}, \eqref{goodZpart-v2} and \eqref{LaplaceZ-1}, we see that it suffices to find the limit of $\mathbb{L}_n(\lambda,\xi, \varepsilon,\delta)$.

On the event $\{Z_n(z_n) \ge \varepsilon\log n, Z_{[1, n-\frac{n}{\log n}],n}(z_n) \le \delta \log n, Z_{(n-\frac{n}{\log n}, n],n} \le \delta n\}$, similarly as \eqref{LaplaceT-2},
we can approximate $Z_n$ by $1+Z_{[1, n-\frac{n}{\log n}], n}$, and $Z_n(z_n)$ by $1+Z_{(n-\frac{n}{\log n},n],n}(z_n)$.
This reduces the study of $\mathbb{L}_n(\lambda,\xi, \varepsilon,\delta)$ to that of 
\begin{multline*}
\widetilde{\mathbb{L}}_n(\lambda,\xi, \eps, \delta) := 
\frac{1}{P_n(z_n)}\E_{\Q^*}\left[ \ind{ S_{w_n}=z_n} \exp(-\lambda \frac{1+Z_{[1, n-\frac{n}{\log n}], n}}{\sigma^2 n/2} )  \frac{\alpha \log n}{1+ Z_{(n-\frac{n}{\log n},n],n}(z_n)} \right.\\
\left. \times \exp(-\xi \frac{1+ Z_{(n-\frac{n}{\log n},n],n}(z_n)}{\alpha \log n} )\ind{Z_n(z_n) \ge \eps\log n}  \ind{ Z_{[1, n-\frac{n}{\log n}],n}(z_n) \le \delta \log n, Z_{(n-\frac{n}{\log n}, n],n} \le \delta n  }\right].
\end{multline*}
Similarly as \eqref{LaplaceT-upbd}, by independence between $(S_{w_n}, Z_{( n-\frac{n}{\log n},n],n}(z_n) )$ and $Z_{[1, n-\frac{n}{\log n}], n}$, and then by use of \eqref{goodZpart-v2} and \eqref{goodTpart-v2} again, we deduce that
{\footnotesize
\begin{align}\label{LaplaceZ-upbd}
&\widetilde{\mathbb{L}}_n(\lambda,\xi, \eps, \delta)\\
\le & \frac{1}{P_n(z_n)}\E_{\Q^*}\left[ \ind{ S_{w_n}=z_n} \frac{\alpha\log n}{Z_n(z_n)} \exp(-\xi \frac{Z_n(z_n)}{\alpha \log n} )\ind{Z_n(z_n) \ge \eps \log n}\right] \E_{\Q^*}\left[  \exp(-\lambda \frac{Z_n}{\sigma^2 n/2})\right] + o_n(1) + o_\eps(1).\nonumber
\end{align}
}
Similarly as \eqref{LaplaceT-3-}, for the same reason mentioned above \eqref{LaplaceZ-upbd},
we obtain the lower bound
{\footnotesize
\begin{align}\label{LaplaceZ-lowerbd}
&\widetilde{\mathbb{L}}_n(\lambda, \xi,\eps, \delta)\\
\ge & \frac{1}{P_n(z_n)}\E_{\Q^*}\left[ \ind{ S_{w_n}=z_n} \frac{\alpha\log n}{Z_n(z_n)} \exp(-\xi \frac{Z_n(z_n)}{\alpha \log n} )\ind{Z_n(z_n) \ge 2\eps \log n}\right] \E_{\Q^*}\left[  \exp(-\lambda \frac{Z_n}{\sigma^2 n/2})\right] + o_n(1) + o_\eps(1).\nonumber
\end{align}
}
Comparing the upper and lower bounds, it remains to investigate
\[
    \frac{1}{P_n(z_n)}\E_{\Q^*}\left[ \ind{ S_{w_n}=z_n} \frac{\alpha\log n}{Z_n(z_n)}\exp(-\xi \frac{Z_n(z_n)}{\alpha \log n} )\ind{Z_n(z_n) \ge \eps \log n}\right] \textrm{ and }\E_{\Q^*}\left[  \exp(-\lambda \frac{Z_n}{\sigma^2n/2})\right] .
\] 
To this end, by change of measure, Kolmogorov's estimate and the Yaglom theorem, one gets 
\begin{align}\label{Z+asymp}
\E_{\Q^*}\left[  \exp(-\lambda \frac{Z_n}{ \sigma^2 n/2 })\right] =\,& \E\left[Z_n\exp(-\lambda \frac{Z_n}{ \sigma^2 n/2 })\right]= n \E\left[\frac{Z_n}{n}\exp(-\lambda \frac{Z_n}{ \sigma^2 n/2 })\ind{Z_n\ge1}\right]\\
=\,& \frac{\sigma^2 n}{2} \P(Z_n \ge 1 ) \E\left[\frac{Z_n}{ \sigma^2 n/2 }\exp(-\lambda \frac{Z_n}{ \sigma^2 n/2 })\Big\vert Z_n\ge1\right] \nonumber\\
=\,&(1+o_n(1)) \Big(\int_0^\infty re^{-\lambda r} e^{-r} \d r+o_n(1)\Big)= \frac{1}{(1+\lambda)^2}(1+o_n(1)) . \nonumber
\end{align}
On the other hand, by change of measure, we have
\begin{align}
&\frac{1}{P_n(z_n)}\E_{\Q^*}\left[ \ind{ S_{w_n}=z_n} \frac{\alpha\log n}{Z_n(z_n)}\exp(-\xi \frac{Z_n(z_n)}{\alpha \log n} )\ind{Z_n(z_n) \ge \varepsilon \log n}\right]\nonumber \\
=\,& \frac{1}{P_n(z_n)}\E_{\Q^*}\left[\sum_{|u|=n}\ind{w_n=u} \ind{S_{w_n}=z_n}  \frac{\alpha\log n}{Z_n(z_n)}\exp(-\xi \frac{Z_n(z_n)}{\alpha \log n} )\ind{Z_n(z_n) \ge \varepsilon \log n}\right] \nonumber \\
=\,&\frac{\alpha \log n}{P_n(z_n)} \E\left[ \exp(-\xi \frac{Z_n(z_n)}{\alpha \log n} )\ind{Z_n(z_n) \ge \varepsilon \log n}\right] \nonumber\\
=\,&  \frac{\alpha\log n\times u_n(z_n)}{P_n(z_n) } \E\left[\exp(-\xi \frac{Z_n(z_n)}{\alpha \log n} )\ind{Z_n(z_n) \ge \varepsilon \log n}\Big\vert Z_n(z_n) \ge 1\right] .\nonumber
\end{align}
Using \eqref{asymp-unx} and \eqref{dW-uniformcvg}, one sees that uniformly in $\|z_n\|\le K\sqrt{n}$,
\begin{equation}\label{Znz+asymp}
 \lim_{\eps\to0+}\lim_{n\to \infty}   \frac{1}{P_n(z_n)}\E_{\Q^*}\left[ \ind{ S_{w_n}=z_n} \frac{\alpha\log n}{Z_n(z_n)}\exp(-\xi \frac{Z_n(z_n)}{\alpha \log n} )\ind{Z_n(z_n) \ge \eps \log n}\right]=\frac{1}{1+\xi}.
\end{equation}
Taking into account \eqref{LaplaceZ-upbd}, \eqref{LaplaceZ-lowerbd}, \eqref{Z+asymp} and \eqref{Znz+asymp},  we end up with
\[
 \lim_{\varepsilon\to0+}\lim_{n\to \infty}  \widetilde{\mathbb{L}}_n(\lambda, \xi,\varepsilon, \delta) =\frac{1}{(1+\lambda)^2} \frac{1}{1+\xi},
\]
which yields that $\mathbb{L}_n(\lambda,\xi)\to \frac{1}{(1+\lambda)^2} \frac{1}{1+\xi}$ when $n\to\infty$.
Theorem \ref{thm: Span} is therefore established.

\appendix
\section{Appendix}\label{appendix}

\begin{proof}[Proof of Lemma \ref{lem: SRW}] 
In view of \eqref{upbdPnx}, we only need to prove \eqref{eq: SRWcvg}. Let us fix $\varepsilon\in (0,1)$.

Inside this proof, we abuse slightly our notation by writing 
\[
    p_t(x):=\frac{1}{2\pi t}e^{-\frac{\|x\|^2}{2t}}, \qquad \forall t>0, x\in \R^2.
\]
Recall from \eqref{Pnx} that 
\[
    P_n(x) = \frac52 \,p_n(\sqrt{\frac52}x)+\frac{O(1)}{n^2}
\] 
uniformly in $x\in \Z^2$.
It follows that
\[
\Gamma_K^n(z_n)=\sum_{j=K}^n P_j(z_n +S_j) =\frac52  \sum_{j=K}^n p_j\Big(\sqrt{\frac52}z_n + \sqrt{\frac52}S_j\Big) + O(1).
\]
So it suffices to treat 
\[
\widehat{\Gamma}_K^n(z_n):=  \sum_{j=K}^n p_j\Big( z_n + \sqrt{\frac52}S_j\Big) 
\]
for any sequence $(z_n)_{n\geq 1}$ in $\R^2$ such that $\frac{\log (1+\|z_n\|)}{\log n} \to 0$ as $n\to \infty$.
Since there exists a constant $C_0>0$ such that $\sup_x p_t(x)\le C_0/t$, by taking $\epsilon_0\leq \frac{\epsilon}{2C_0}$, we have
\[
\widehat{\Gamma}_K^{\lfloor n^{\epsilon_0}\rfloor}(z_n) =  \sum_{j=K}^{\lfloor n^{\epsilon_0}\rfloor} p_j\Big( z_n + \sqrt{\frac52}S_j\Big) \le \frac{\epsilon}{2}\log n.
\]
The same arguments imply that if \eqref{eq: SRWcvg} holds, we can replace the fixed integer $K$ by $\lfloor \log n\rfloor+1$.
So it remains to study 
\[
    \widehat{\Gamma}_{\lfloor n^{\epsilon_0}\rfloor+1}^n (z_n)= \sum_{j=\lfloor n^{\epsilon_0}\rfloor+1}^{n} p_j\Big( z_n + \sqrt{\frac52}S_j\Big). 
\]

Let $(B_t)_{t\ge0}$ be a standard Brownian motion in $\R^2$ with transition densities $p_t(x)$. Donsker's invariance principle says that 
\[
\Big(\sqrt{\frac52} \frac{S_{\lfloor nt \rfloor }}{\sqrt{n}}\Big)_{t\ge0}
\]
converges in law to $(B_t)_{t\ge0}$.  
By Koml\'os--Major--Tusn\'ady strong invariance principle and some standard estimates on the Brownian motion, we can define the random walk $(S_n)$ and the Brownian motion $(B_t)$ on the same probability space $(\Omega, \mathbb{P})$, in such a way that for some constant $C>0$ and for sufficiently small (fixed) $\epsilon_0\in(0,\frac{1}{2})$, the event
\begin{equation}
\label{eq:KMT-set}
 E_n:=\bigg\{\sup_{0\le t\le n}\Big\|\sqrt{\frac52} S_{\lfloor t\rfloor} - B_t\Big\| \le C\log n\bigg\}\cap \Big\{\| B_s\| \le s^{2/3}, \forall s\in [n^{\epsilon_0}, n]\Big\}   
\end{equation}
happens with probability $1-O(\frac1n)$. On this event $E_n$ with high probability, let us compare
\[
\sum_{j=\lfloor n^{\epsilon_0}\rfloor+1}^{n} p_j\Big(z_n + \sqrt{\frac52}S_j\Big) \textrm{ and } \int_{\lfloor n^{\epsilon_0}\rfloor+1}^n p_s(B_s)\d s .
\]
Observe that
\begin{align*}
& \bigg|\sum_{j=\lfloor n^{\epsilon_0}\rfloor+1}^{n} p_j\Big(z_n + \sqrt{\frac52}S_j\Big) -  \int_{\lfloor n^{\epsilon_0}\rfloor+1}^n p_s(B_s)\d s\bigg| \\
\le  & \sum_{j = \lfloor n^{\epsilon_0} \rfloor}^{n-1} \int_{j}^{j+1} \Big|p_j\Big(z_n + \sqrt{\frac52}S_{\lfloor s\rfloor}\Big) - p_s(B_s) \Big| \d s\\
\le & \sum_{j = \lfloor n^{\epsilon_0} \rfloor}^{n-1} \int_{j}^{j+1} \Big| p_j\Big(z_n + \sqrt{\frac52}S_{\lfloor s\rfloor}\Big) - p_j(B_s) \Big| \d s +  \sum_{j = \lfloor n^{\epsilon_0} \rfloor}^{n-1}  \int_{j}^{j+1} | p_j(B_s) - p_s(B_s) | \d s.
\end{align*}
Recall that $\frac{\log (1+\|z_n\|)}{\log n} \to 0$. By definition of $E_n$, we can assume $n$ to be sufficiently large so that, on the event $E_n$, 
\[
    \|z_n + \sqrt{\frac52}S_{\lfloor s\rfloor} - B_s\| \leq \lfloor s^{\epsilon_0} \rfloor \textrm{ for }n^{\epsilon_0}\leq s\leq n.
\]
Thus, applying the inequality 
\[
    |p_t(x+z) - p_t(x)| \le \frac{1}{2\pi t}\Big|\frac{\|x+z\|^2-\|x\|^2}{2t}\Big|\le \frac{2\|x\|\|z\| + \|z\|^2}{2\pi t^2}
\] 
with $x=B_s$ and $z=z_n + \sqrt{\frac52}S_{\lfloor s\rfloor}-B_s$, we get that on the event $E_n$,
\begin{align}
     \sum_{j = \lfloor n^{\epsilon_0} \rfloor}^{n-1} \int_{j}^{j+1} \Big| p_j\Big(z_n + \sqrt{\frac52}S_{\lfloor s\rfloor}\Big) - p_j(B_s) \Big| \d s 
    \le &  \sum_{j = \lfloor n^{\epsilon_0} \rfloor}^{n-1} \int_{j}^{j+1} \frac{2 \|B_s\| j^{\epsilon_0} + j^{2\epsilon_0}}{2\pi j^2} \d s  \label{eq:pj-En}\\
    \lesssim & \sum_{j = \lfloor n^{\epsilon_0} \rfloor}^{n-1} \Big(j^{-\frac43+\epsilon_0}+j^{-3/2}\Big) \lesssim n^{-\epsilon_0/4}.\nonumber
\end{align}
On the other hand, as $|\partial_t p_t(x)|\leq (2\pi t^2)^{-1}$ for all $x\in \R^2$, we have $|p_j(x) - p_s(x)|\le \frac{1}{2\pi}|\frac1j-\frac1s|\le \frac{1}{2\pi j^2}$ for $s\in[j,j+1)$, and thus
\[
\sum_{j = \lfloor n^{\epsilon_0} \rfloor}^{n-1}  \int_{j}^{j+1} | p_j(B_s) - p_s(B_s) | \d s \le  \sum_{j = \lfloor n^{\epsilon_0} \rfloor}^{n-1}  \frac{1}{2\pi j^2}\lesssim n^{-\epsilon_0}.
\]
Hence, we deduce that on the event $E_n$,
\[
 \bigg|\sum_{j=\lfloor n^{\epsilon_0}\rfloor+1}^{n}p_j\Big(z_n + \sqrt{\frac52}S_j\Big) -  \int_{ \lfloor n^{\epsilon_0} \rfloor+1 }^n p_s(B_s)\d s\bigg| \lesssim n^{-\epsilon_0/4}.
\]

Therefore, we only need to check that for any $p\ge 2$, as $n\to \infty$,
\begin{equation}\label{eq: BMcvg}
\mathbb{P}\bigg(\Big| \int_{n^{\epsilon_0} }^n p_s(B_s)\d s - \frac{1}{4\pi} \log n \Big| \ge \epsilon \log n\bigg) = o((\log n)^{-p}).
\end{equation} 
In fact, we are going to show that for any $m\geq 1$, there exists a constant $C_{\epsilon, m}>0$ such that for all $n\geq1$,
\begin{equation}\label{eq: BMcvg-v0}
\mathbb{P}\bigg(\Big| \int_{1}^{e^n} p_s(B_s)\d s - \frac{1}{4\pi} n \Big| \ge \epsilon n\bigg) \le C_{\epsilon, m} n^{-m}.
\end{equation}
Note that for sufficiently small $\epsilon_0$,
\[
\int_1^{n^{\epsilon_0}}p_s(B_s)\d s \le \int_1^{n^{\epsilon_0}}\frac{1}{2\pi s}\d s= \frac{1}{2\pi }\epsilon_0 \log n \le \frac{\epsilon}{2}\log n,
\]
and that $\int_{e^{\lfloor \log n\rfloor}}^n p_s(B_s)\d s \le 1$. 
Consequently, \eqref{eq: BMcvg} follows from \eqref{eq: BMcvg-v0}.

Let us turn to prove \eqref{eq: BMcvg-v0}. By setting $X_u:=e^{-u/2} B_{e^u}$, we see that for any integer $k\ge 1$,
\begin{equation}
    \label{eq:Yk-stationary}
    Y_k:=\int_{e^{k-1}}^{e^k} p_s(B_s)\d s = \int_{e^{k-1}}^{e^k} \frac{1}{2\pi s} e^{-\frac{\| B_s \|^2}{2s}}\d s = \int_{k-1}^k \frac{1}{2\pi} e^{-\frac{\|X_u\|^2}{2}}\d u.
\end{equation}
Since $(X_u)_{u\ge0}$ is a two-dimensional stationary Ornstein--Uhlenbeck process, with the invariant probability measure $\mathcal{N}(0, \mathrm{Id})$, we have for all $k\geq 1$,
\[
\mathbb{E}[Y_k]=\frac{1}{2\pi}\mathbb{E}\Big[e^{-\frac{\|X_0\|^2}{2}}\Big]= \frac{1}{4\pi}.
\]
Now we have a uniformly bounded stationary sequence $(Y_k-\frac{1}{4\pi})_{k\ge1}$ with zero mean. By adapting the proof of Proposition 2.13 in \cite{BW22}, we can conclude that for any $m\geq 1$, there exists a constant $C_{\epsilon, m}>0$ such that for all $n\geq1$,
\[
    \mathbb{P}\Big(\Big|Y_1+\cdots +Y_n - \frac{1}{4\pi}n \Big| \ge \epsilon n \Big)\le C_{\epsilon, m} n^{-m}.
\]
The proof is therefore complete.
\end{proof}

\begin{proof}[Proof of Lemma \ref{lem: SRW-sqrt}] 
We will proceed analogously to the proof of Lemma \ref{lem: SRW}. Let us fix $\varepsilon\in (0,1/2)$, and keep the same notation used in the proof of Lemma \ref{lem: SRW}.

First of all, it suffices to prove \eqref{eq:SRWcvg-sqrt} after replacing $\Gamma_{\lfloor \log n\rfloor +1}^n(0)$ by 
\[
\widehat{\Gamma}_{\lfloor \log n\rfloor+1}^n(0)=  \sum_{j=\lfloor \log n\rfloor+1}^n p_j\Big(\sqrt{\frac52}S_j\Big).
\]
Since
\[
\sum_{j=\lfloor \log n\rfloor+1}^{\lfloor\exp( \sqrt{\log n})\rfloor} p_j\Big(\sqrt{\frac52}S_j\Big) \lesssim  \sqrt{\log n},
\]
we only need to focus on 
\[
    \sum_{j=\lfloor\exp( \sqrt{\log n})\rfloor+1}^{n} p_j\Big(\sqrt{\frac52}S_j\Big). 
\]
Over the same event $E_n$ defined in \eqref{eq:KMT-set}, we can approximate this sum by 
\[
    \int_{\lfloor \exp( \sqrt{\log n})\rfloor+1}^{n} p_s(B_s)\d s .
\]
More precisely, similarly as in \eqref{eq:pj-En}, we get on the $E_n$,
\begin{align*}
     \sum_{j = \lfloor \exp( \sqrt{\log n}) \rfloor}^{n-1} \int_{j}^{j+1} \Big| p_j\Big(\sqrt{\frac52}S_{\lfloor s\rfloor}\Big) - p_j(B_s) \Big| \d s 
\le &  \sum_{j = \lfloor \exp( \sqrt{\log n}) \rfloor}^{n-1} \int_{j}^{j+1} \frac{2 \|B_s\| C\log n + (C\log n)^2}{2\pi j^2} \d s\\
\lesssim & \sum_{j = \lfloor \exp( \sqrt{\log n}) \rfloor}^{n-1} \Big(j^{-\frac43}\log n +j^{-2}(\log n)^2\Big) \lesssim (\log n)^{-p}
\end{align*}
for any integer $p\geq 1$.
Meanwhile, we also have
\[
\sum_{j = \lfloor \exp( \sqrt{\log n}) \rfloor}^{n-1}  \int_{j}^{j+1} | p_j(B_s) - p_s(B_s) | \d s \le \sum_{j = \lfloor \exp( \sqrt{\log n}) \rfloor}^{n-1}  \frac{1}{2\pi j^2}\lesssim (\log n)^{-p}.
\]
Hence, we deduce that on the event $E_n$, for any integer $p\geq 1$, 
\[
 \bigg|\sum_{j=\lfloor \exp( \sqrt{\log n}) \rfloor+1}^{n}p_j\Big(\sqrt{\frac52}S_j\Big) -  \int_{ \lfloor \exp( \sqrt{\log n})  \rfloor+1 }^n p_s(B_s)\d s\bigg| \lesssim (\log n)^{-p}.
\]

Finally, as $n\to \infty$, we have
\begin{equation*}
\mathbb{P}\bigg(\Big| \int_{\exp( \sqrt{\log n})}^n p_s(B_s)\d s - \frac{1}{4\pi} \log n \Big| \ge (\log n)^{\frac12+\varepsilon}\bigg) = \frac{O(1)}{(\log n)^{3/2-\varepsilon}},
\end{equation*} 
by showing the following analog of \eqref{eq: BMcvg-v0} that
\begin{equation}
\label{eq: BMcvg-v1}
\mathbb{P}\bigg(\Big| \int_{1}^{e^n} p_s(B_s)\d s - \frac{1}{4\pi} n \Big| \ge n^{\frac12+\varepsilon}\bigg) =\frac{O(1)}{n^{3/2-\varepsilon}}.
\end{equation}
For the same stationary sequence $(Y_k)$ defined in \eqref{eq:Yk-stationary}, it follows from Proposition 2.13 in \cite{BW22} that for any $m\geq 1$, there exists some constant $C(m)$ such that, for any $n\geq 1$, 
\[
    \mathbb{E}\Bigg[\bigg|\sum_{k=1}^n Y_k-\frac{1}{4\pi}n\bigg|^m\Bigg]\leq C(m) \cdot n^{\frac{m}{2}}.
\]
Applying a Chebyshev-type inequality with $m=\lfloor \frac{3}{2\varepsilon}\rfloor$, we deduce that there exists a constant $C_{\varepsilon}$ such that for all $n\geq 1$,
\[
    \mathbb{P}\Big(\Big|Y_1+\cdots +Y_n - \frac{1}{4\pi}n \Big| \ge n^{\frac12+\varepsilon} \Big)\le C_{\varepsilon} n^{-m\varepsilon}.
\]
As $(m+1)\varepsilon >\frac32$, this gives \eqref{eq: BMcvg-v1} and finishes the proof of Lemma \ref{lem: SRW-sqrt}.
\end{proof}

\bibliographystyle{alpha}

\end{document}